\documentclass[11pt,a4paper,usenames,dvipsnames]{article}

\usepackage{url,amsmath,amssymb,latexsym,mathrsfs,comment,amsthm,enumerate,geometry,calc,ifthen,mathdots,textcomp}

\usepackage{cases}

\usepackage[noadjust]{cite}

\usepackage[colorlinks]{hyperref}

\usepackage[margin=10pt,font=small,labelfont=bf, labelsep=period]{caption}

\usepackage{makecell}

\usepackage[x11names, svgnames, rgb,table]{xcolor}
\usepackage{hhline}
\usepackage{multirow}

\usepackage{enumitem}

\usepackage[utf8]{inputenc}
\usepackage[T1]{fontenc}

\usepackage{tikz}
\usetikzlibrary{decorations.markings}
\usetikzlibrary{matrix}
\usepgflibrary{snakes,arrows,shapes}

\renewcommand{\arraystretch}{1.2}

\newcommand\nc\newcommand
\nc\rnc\renewcommand

\nc\al\alpha
\nc\be\beta
\nc\ga\gamma
\nc\de\delta
\rnc\th\theta
\nc\lam\lambda
\nc\om\omega
\nc\ze\zeta
\nc\si\sigma
\nc\Si\Sigma
\nc\Ga\Gamma
\nc\De\Delta
\nc\Th\Theta
\nc\Om\Omega
\nc\nab\nabla

\nc\AND{\qquad\text{and}\qquad}
\nc\ANd{\quad\text{and}\quad}
\nc\COMMA{,\qquad}
\nc\COMMa{,\quad}

\rnc\iff{\ \Leftrightarrow\ }
\nc\IFf{\quad \Leftrightarrow\quad }
\nc\Iff{\ \ \Leftrightarrow\ \ }
\nc\IFF{\qquad \Leftrightarrow\qquad }
\rnc\implies{\ \Rightarrow\ }
\nc\IMPLIES{\qquad \Rightarrow\qquad }

\nc\set[2]{\{#1:#2\}}
\nc\bigset[2]{\big\{#1:#2\big\}}

\nc\bit{\begin{itemize}[label=\textbullet, leftmargin=5mm]}
\nc\eit{\end{itemize}}
\nc\ben{\begin{thmenumerate}}
\nc\bena{\begin{enumerate}[label=\textup{(\alph*)},leftmargin=10mm]}
\nc\een{\end{thmenumerate}}
\nc\eena{\end{enumerate}}

\nc\pf{\begin{proof}}
\nc\epf{\end{proof}}
\nc\pfclaim{\begin{quote}\begin{proof}}
\nc\epfclaim{\end{proof}\end{quote}}
\nc\epfres{\hfill\qed}
\nc\epfreseq{\tag*{\qed}}
\let\oldproofname=\proofname
\renewcommand{\proofname}{\rm\bf{\oldproofname}}

\renewcommand{\H}{\mathscr H}
\renewcommand{\L}{\mathscr L}
\newcommand{\R}{\mathscr R}
\newcommand{\D}{\mathscr D}
\newcommand{\J}{\mathscr J}
\newcommand{\K}{\mathscr K}
\nc\rH{\mathrel{\H}}
\nc\rL{\mathrel{\L}}
\nc\rR{\mathrel{\R}}
\nc\rD{\mathrel{\D}}
\nc\rJ{\mathrel{\J}}
\nc\rK{\mathrel{\K}}
\nc\rsi{\mathrel{\si}}

\renewcommand{\P}{\mathcal P} 
\newcommand{\Ptw}[1]{\mathcal{P}_{#1}^\Phi}
\rnc\S{\mathcal S}
\nc\A{\mathcal A}
\nc\F{\mathcal F}
\nc\I{\mathcal I}
\nc\C{\mathcal C}
\nc\LL{\mathcal L}
\newcommand{\T}{\mathcal T}
\newcommand{\PT}{\mathcal P\mathcal T} 
\renewcommand{\I}{\mathcal I}  

\newcommand{\id}{\operatorname{id}}

\newcommand{\ba}{\boldsymbol{a}}
\newcommand{\bb}{\boldsymbol{b}}
\newcommand{\bc}{\boldsymbol{c}}
\newcommand{\bd}{\boldsymbol{d}}
\nc\bs[1]{\boldsymbol{#1}}

\newcommand{\Float}{\operatorname{\Phi}}
\newcommand{\multw}{}
\newcommand{\multwk}[1]{\cdot}

\nc\permdiff\partial
\nc\pd\permdiff
\nc\Proj[1]{\overline{#1}}
\nc\norm[1]{\langle\!\langle#1\rangle\!\rangle}
\nc\fd\De

\newcommand{\N}{\mathbb{N}}
\nc\bn{{\bf n}}
\nc\bm{{\bf m}}
\nc\bnz{{\bf n}_0}
\nc\bmz{{\bf m}_0}
\nc\bdz{{\bf d}_0}
\nc\bq{{\bf q}}
\nc\zero{{\bf 0}}

\newcommand{\coker}{\operatorname{coker}}
\newcommand{\dom}{\operatorname{dom}} 
\newcommand{\codom}{\operatorname{codom}}
\newcommand{\rank}{\operatorname{rank}}

\newcommand{\per}{\operatorname{per}}
\newcommand{\mumin}{\operatorname{\mu in}}

\DeclareRobustCommand{\stirling}{\genfrac\{\}{0pt}{}}

\newcommand{\Cong}{\text{\sf Cong}}

\newcommand{\muup}{\mu^\uparrow}
\newcommand{\mudown}{\mu^\downarrow}

\newcommand{\cg}{\text{\sf cg}}
\newcommand{\cgx}{\text{\sf cgx}}
\newcommand{\fcg}{\text{\sf cg}}
\nc\thx{\theta^{\textsf{x}}}

\newcommand{\pos}{\text{\sf pos}}

\nc\Pair\Pi

\newcommand{\Cmatsetup}{\arraycolsep=1.4pt\def\arraystretch{1}\footnotesize}
\makeatletter
\DeclareRobustCommand
  \vvdots{\vbox{\baselineskip4\p@ \lineskiplimit\z@\kern4\p@
    \hbox{.}\hbox{.}\hbox{.}}}
\makeatother

\newcommand{\restr}{\mathord{\restriction}}

\newcommand{\leqc}{\leq_C}

\newcommand{\lessc}{<_C}
\newcommand{\geqc}{\geq_C}
\newcommand{\grc}{>_C}
\newcommand{\hgt}{{\text{\sf x}}}
\rnc\emptyset{\varnothing}
\nc\sub\subseteq
\nc\mt\mapsto
\nc\wh\widehat
\nc\normal\unlhd
\nc\sm\setminus
\nc\mr\mathrel
\nc\pc[2]{(#1,#2)^\sharp}

\nc{\uv}[1]{\fill (#1,2)circle(.17);}
\nc{\lv}[1]{\fill (#1,0)circle(.17);}
\nc{\uvs}[1]{{\foreach \x in {#1} { \uv{\x}}}}
\nc{\lvs}[1]{{\foreach \x in {#1} { \lv{\x}}}}
\nc{\darcx}[3]{\draw(#1,0)arc(180:90:#3) (#1+#3,#3)--(#2-#3,#3) (#2-#3,#3) arc(90:0:#3);}
\nc{\darc}[2]{\darcx{#1}{#2}{.4}}
\nc{\uarcx}[3]{\draw(#1,2)arc(180:270:#3) (#1+#3,2-#3)--(#2-#3,2-#3) (#2-#3,2-#3) arc(270:360:#3);}
\nc{\uarc}[2]{\uarcx{#1}{#2}{.4}}
\nc{\stline}[2]{\draw(#1,2)--(#2,0);}

\numberwithin{equation}{section}

\newtheorem{thm}[equation]{Theorem}
\newtheorem{lemma}[equation]{Lemma}
\newtheorem{cor}[equation]{Corollary}
\newtheorem{prop}[equation]{Proposition}

\theoremstyle{definition}

\newtheorem{defn}[equation]{Definition}

\newtheorem{rem}[equation]{Remark}
\newtheorem{exa}[equation]{Example}

\newenvironment{thmenumerate}{
   \begin{enumerate}[label=\textup{(\roman*)}, widest=(5), leftmargin=10mm]}{
    \end{enumerate}}
\newenvironment{thmsubenumerate}{
   \begin{enumerate}[label=\textup{(\alph*)}, widest=(b), leftmargin=10mm]}{
    \end{enumerate}}

\newcounter{caseco}
\newcommand{\case}{\refstepcounter{caseco}\medskip\noindent\textbf{Case \thecaseco:}\ } 

\newcounter{subcaseco}
\newcommand{\subcase}{\refstepcounter{subcaseco}\medskip\noindent\textbf{Subcase \thecaseco.\thesubcaseco:}\ }    

\newcounter{stepco}
\newcommand{\step}{\refstepcounter{stepco}\medskip\noindent\textbf{Step \thestepco. }} 

\newcounter{stageco}
\newcommand{\stage}{\refstepcounter{stageco}\medskip\noindent\textbf{Stage \thestageco: }}

\newcounter{RT}\renewcommand{\theRT}{RT\arabic{RT}}

\newcounter{fR}\renewcommand{\thefR}{fR\arabic{fR}}

\definecolor{delcol}{HTML}{F4F3F4}
\definecolor{excepcol}{HTML}{D50B53}
\definecolor{mucol}{HTML}{B9C406}
\definecolor{Ncol}{HTML}{FAA565}
\definecolor{Rcol}{HTML}{A882C1}

\newcounter{ncols}
\newcounter{incols}
\newenvironment{partn}[1]{
  \setcounter{ncols}{#1} \setcounter{incols}{\thencols - 1}\setlength{\arraycolsep}{1pt}
  \Bigl( \hspace{-1.5truemm}\scriptsize \renewcommand*{\arraystretch}{1}
    \begin{array}{@{\hskip 3pt} c *{\theincols}{|c} @{\hskip 3pt}  }
}{
     \end{array}
     \normalsize \hspace{-1.5truemm}\Bigr)\setlength{\arraycolsep}{6pt}
}

\begin{document}

\title{\vspace{-1cm}Properties of congruences of twisted partition monoids and their lattices}

\author{James East\footnote{Centre for Research in Mathematics and Data Science, Western Sydney University, Locked Bag 1797, Penrith NSW 2751, Australia. {\it Email:} {\tt j.east\,@\,westernsydney.edu.au}} 
\ and
Nik Ru\v{s}kuc\footnote{Mathematical Institute, School of Mathematics and Statistics, University of St Andrews, St Andrews, Fife KY16 9SS, UK. {\it Email:} {\tt nik.ruskuc\,@\,st-andrews.ac.uk}}}
\date{}

\maketitle

\vspace{-1cm}
\begin{abstract}
\noindent
We build on the recent characterisation of congruences on the infinite twisted partition monoids $\Ptw{n}$ and their finite $d$-twisted homomorphic images $\Ptw{n,d}$, and investigate their algebraic and order-theoretic properties.
We prove that each congruence of $\Ptw{n}$ is (finitely) generated by at most $\lceil\frac{5n}2\rceil$ pairs, and we characterise the principal ones.
We also prove that the congruence lattice $\Cong(\Ptw{n})$ is not modular (or distributive); it has no infinite ascending chains, but it does have infinite descending chains and infinite antichains.
By way of contrast, the lattice $\Cong(\Ptw{n,d})$ is modular but still not distributive for $d>0$, while $\Cong(\Ptw{n,0})$ is distributive. 
We also calculate the number of congruences of~$\Ptw{n,d}$, 
showing that the array $\big(|\Cong(\Ptw{n,d})|\big)_{n,d\geq 0}$ has a rational generating function,
and that for a fixed $n$ or $d$, $|\Cong(\Ptw{n,d})|$ is a polynomial in $d$ or $n\geq 4$, respectively.

\medskip

\medskip


\end{abstract}

\setcounter{tocdepth}{1}
\tableofcontents

\section{Introduction}\label{sec:intro}

The twisted partition monoid $\Ptw{n}$ is a countably infinite monoid obtained from the classical finite partition monoid $\P_n$ \cite{HR2005,Martin1994,Jones1994_2} by taking into account the number of floating components formed when multiplying partitions.
Its finite $d$-twisted images $\Ptw{n,d}$ are obtained by limiting the number of floating components to at most $d$, and collapsing all other elements to zero.
The congruences on these monoids have been determined in \cite{ERtwisted1}, where the reader may also find further background and additional references; see also \cite{EMRT2018,ERPX,ER2020}.
In the current article we harness the power of the classification from \cite{ERtwisted1} to investigate the algebraic properties of congruences on $\Ptw n$ and~$\Ptw{n,d}$, and the combinatorial/order-theoretic properties of the congruence lattices $\Cong(\Ptw{n})$ and~$\Cong(\Ptw{n,d})$.
We refer the reader to \cite[Section 1]{ERtwisted1} and references therein for the context and background for the investigation presented in the two papers.

Before we discuss the results of the current article, it is instructive to recall the situation for the partition monoid $\P_n$ itself, whose congruences were determined in \cite{EMRT2018}.  The classification is stated in Theorem \ref{thm:CongPn} below, and the congruence lattice $\Cong(\P_n)$ is shown in Figure \ref{fig:CongPn}.  From the figure, a number of facts are easily verified: the lattice has size $3n+8$ for $n\geq4$; it has three atoms, and a single co-atom; it is distributive, and hence also modular (because it does not contain any five-element diamond or pentagon sublattice).  It was shown in \cite[Section 5]{EMRT2018} that all but three congruences of $\P_n$ are principal, and those that are not (denoted $\mu_{\S_2}$, $\lam_{\S_2}$ and~$\rho_{\S_2}$) are generated by two pairs.  
Certain structural properties of the monoid $\P_n$ are responsible for the `neat' structure of the lattice $\Cong(\P_n)$, as seen in Figure \ref{fig:CongPn}.  These include the following:  the ideals of $\P_n$ form a (finite) chain; the maximal subgroups of $\P_n$ are (finite) symmetric groups $\S_q$ ($q=0,1,\ldots,n$), the normal subgroups of which also form (finite) chains; and the minimal ideal is a rectangular band.

The ideal structure of the twisted monoid $\Ptw n$ is very different to that of $\P_n$: there are infinitely many ideals; they do not form a chain; and there is no minimal ideal.  These facts alone already lead to a substantially more complicated structure of the lattice $\Cong(\Ptw n)$, including the existence of infinite descending chains; we will also see that it has infinite anti-chains, though by contrast it has no infinite ascending chains (Theorem \ref{thm:chains}).  This property of ascending chains implies that every congruence on $\Ptw n$ is finitely generated, and this in turn implies that $\Cong(\Ptw n)$ is countable, as also follows directly from the classification in \cite{ERtwisted1}.  In fact, a much stronger property than finite generation holds: Theorem \ref{thm:fg} shows that every congruence is generated by at most $\lceil\frac{5n}2\rceil$ pairs.
We also show that there is no constant bound on the number of pairs needed to generate congruences of $\Ptw{n}$ (Remark \ref{re:nobound}), and classify
 the principal congruences (Theorem~\ref{thm:pc}).
Other order-theoretic properties of the lattice $\Cong(\Ptw n)$ include the following.  The lattice has a single co-atom, but no atoms (Theorem \ref{thm:atoms}).  The latter says that the trivial congruence has no covers in the lattice; in fact, every congruence has only finitely many covers, though some congruences cover infinitely many congruences (Theorem~\ref{thm:covers}).  
In another marked divergence with the ordinary finite partition monoid $\P_n$, the lattice $\Cong(\Ptw n)$ contains five-element diamond and pentagon sublattices, meaning that it is neither distributive nor modular (Theorem \ref{thm:pent}).

Other infinite structures related to the finite partition monoids arise by allowing an infinite base set \cite{ERPX}, or by considering partition categories \cite{ER2020}, where the sizes of the (finite) base sets can be unbounded.  Again contrasting with $\Ptw n$, the congruence lattices in these cases are always distributive and well quasi-ordered (so have no infinite descending chains and no infinite anti-chains), but have infinite ascending chains; congruences can be non-finitely generated.

The situation for the finite $d$-twisted monoids $\Ptw{n,d}$ is as follows.  First, the lattice $\Cong(\Ptw{n,d})$ is of course finite.  In fact, we will obtain an exact formula for $|\Cong(\Ptw{n,d})|$,
and show that for fixed $n\geq 0$ or $d\geq0$, $|\Cong(\Ptw{n,d})|$ is a polynomial in~$d\geq0$ or $n\geq4$, respectively (Theorem~\ref{th:uberenu}).
 Asymptotically, for fixed $d\geq0$ we have $|\Cong(\Ptw{n,d})|\sim(3n)^{d+1}/(d+1)!$ as $n\to\infty$, and for fixed $n\geq4$ we have $|\Cong(\Ptw{n,d})|\sim13d^{3n-1}/(3n-1)!$ as $d\to\infty$ 
 (Remark~\ref{rem:as}).  We also show that the array $\big(|\Cong(\Ptw{n,d})|\big)_{n,d\geq 0}$ has a rational generating function; this is given explicitly in~\eqref{eq:GF}.
 Congruences on $\Ptw{n,d}$ continue to be generated by at most $\lceil\frac{5n}2\rceil$ pairs (Corollary \ref{co:ffg}); the principal congruences are classified in Theorem \ref{thm:fpc}.  Finiteness of the lattice $\Cong(\Ptw{n,d})$ implies the existence of atom(s) and coatom(s), and as we explain at the beginning of Section~\ref{sec:CongPnd}, it has precisely one of each.  In Theorem \ref{thm:distrib} we show that $\Cong(\Ptw{n,0})$ is distributive (and hence modular), while for $d\geq1$, $\Cong(\Ptw{n,d})$ is modular, but not distributive.  Distributivity of the lattices $\Cong(\Ptw{n,0})$ and $\Cong(\P_n)$ is one of several properties shared by the $0$-twisted monoid~$\Ptw{n,0}$ and the (ordinary) monoid $\P_n$; for some others, see Remarks~\ref{rem:CongPn0} and \ref{rem:fg0}, and compare Figures~\ref{fig:CongPn} and~\ref{fig:CongPn0}.  

The article is organised as follows.  After giving definitions and preliminaries in Section \ref{sec:prelim}, we review the classification \cite{ERtwisted1} of congruences on $\Ptw n$ in Section \ref{sec:CongPtwn}.  Section \ref{sec:props} establishes the above-mentioned properties of the lattice $\Cong(\Ptw n)$, including (co)atoms, (anti-)chains, covers and (non)distributivity.  Section \ref{sec:fg} proves the results on generation of congruences.  The $d$-twisted monoids $\Ptw{n,d}$ are treated in Sections \ref{sec:fintwist} (review of the classification), \ref{sec:CongPnd} (order-theoretic properties of the lattice), \ref{sec:fgPnd} (generation of congruences), and \ref{sec:enumeration} (enumeration).
The cases where $n\leq1$ are somewhat different, and are given a special treatment throughout.

\subsection*{Acknowledgements}

The first author is supported by ARC Future Fellowship FT190100632.  The second author is supported by EPSRC grant EP/S020616/1.  We thank Volodymyr Mazorchuk for his suggestion to look at congruences of the $0$-twisted partition monoids $\Ptw{n,0}$, which (eventually) led to the current paper and \cite{ERtwisted1}.  We also thank the referee for their insightful questions, which led to Theorem \ref{thm:sm}, and Remarks \ref{rem:no_covers} and \ref{rem:sts}.

\begin{figure}[ht]

\begin{center}

\scalebox{0.8}{
\begin{tikzpicture}[scale=0.7]
\begin{scope}[minimum size=8mm,inner sep=0.5pt, outer sep=1pt]

\node (m0) at (6,0) [draw=blue,fill=white,circle,line width=2pt] { $\mu_0$};
\node (R0) at (9,4) [draw=blue,fill=white,circle,line width=2pt] { $R_0$};
\node (R1) at (6,6) [draw=blue,fill=white,circle,line width=2pt] { $R_1$};
\node (R2) at (3,10) [draw=blue,fill=white,circle,line width=2pt] { $R_2$};
\node (R3) at (3,16) [draw=blue,fill=white,circle,line width=2pt] { $R_3$};
\node (Rn) at (3,20) [draw=blue,fill=white,circle,line width=2pt] { $R_n$};

\node (r0) at (9,2) [draw,circle] { $\rho_0$};
\node (r1) at (6,4) [draw,circle] { $\rho_1$};
\node (rS2) at (3,6) [draw,circle] { $\rho_{\S_2}$};

\node (l0) at (6,2) [draw,circle] { $\lam_0$};
\node (l1) at (3,4) [draw,circle] { $\lam_1$};
\node (lS2) at (0,6) [draw,circle] { $\lam_{\S_2}$};

\node (m1) at (3,2) [draw,circle] { $\mu_1$};
\node (mS2) at (0,4) [draw,circle] { $\mu_{\S_2}$};

\node (RS2) at (3,8) [draw,circle] { $R_{\S_2}$};
\node (RA3) at (3,12) [draw,circle] { $R_{\A_3}$};
\node (RS3) at (3,14) [draw,circle] { $R_{\S_3}$};

\node (dots) at (3,18) [draw,circle, white] { $\vdots$}; \node (dotss) at (3,18.15)  { $\vdots$};
%extra node on previous to get dots centred in white space

\node () at (7.6,0) { $=\De_{\P_n}$};
\node () at (4.6,20) { $=\nab_{\P_n}$};

\end{scope}

\draw 
(m0) -- (r0)
(m0) -- (l0)
(m0) -- (m1)
(m1) -- (mS2)
(m1) -- (l1)
(m1) -- (r1)
(l0) -- (l1)
(l0) -- (R0)
(r0) -- (r1)
(r0) -- (R0)
(mS2) -- (lS2)
(mS2) -- (rS2)
(l1) -- (lS2)
(l1) -- (R1)
(r1) -- (rS2)
(r1) -- (R1)
(R0) -- (R1)
(lS2) -- (RS2)
(rS2) -- (RS2)
(R1) -- (RS2)
(RS2) -- (R2)
(R2) -- (RA3)
(RA3) -- (RS3)
(RS3) -- (R3)
(R3) -- (dots)
(dots) -- (Rn)
;

\end{tikzpicture}
}
\caption{The Hasse diagram of $\Cong(\P_n)$; see Theorem \ref{thm:CongPn}.  Rees congruences are indicated in blue outline.}
\label{fig:CongPn}
\end{center}

\end{figure}
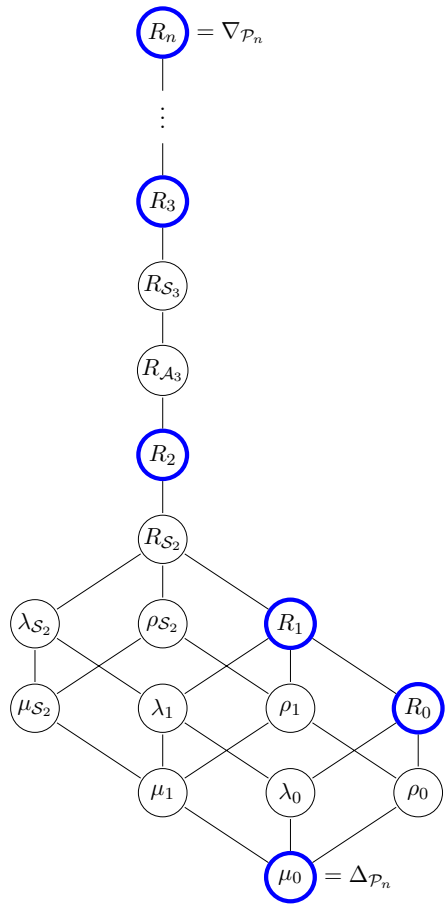

\section{Preliminaries}\label{sec:prelim}

In this section we give a basic overview of monoids and congruences (Subsection \ref{subsec:cong}), the partition monoid~$\P_n$ and its congruences (Subsections \ref{subsec:Pn} and \ref{subsec:CongPn}), and the twisted and $d$-twisted partition monoids~$\Ptw{n}$ and $\Ptw{n,d}$ (Subsection~\ref{subsec:Ptw}).  The exposition will be rapid, and the reader can find a more detailed introduction in \cite{ERtwisted1}.  For congruences of $\Ptw{n}$ and $\Ptw{n,d}$ we will need a more detailed account, and we postpone this until Sections \ref{sec:CongPtwn} and \ref{sec:fintwist}, respectively.

\subsection{Monoids and congruences}\label{subsec:cong}

Recall that a \emph{congruence} on a monoid $M$ is an equivalence relation $\si$ compatible with the operation, in the sense that $(x,y)\in\si\implies(axb,ayb)\in\si$ for all $x,y,a,b\in M$.  The set $\Cong(M)$ of all congruences on $M$ is a lattice under inclusion, with top and bottom elements $\nab_M=M\times M$ and $\De_M=\bigset{(x,x)}{x\in M}$.  

Given a (possibly empty) ideal $I$ of $M$, we have the \emph{Rees congruence} $R_I=\De_M\cup\nab_I$.  Associated to any congruence $\si$ is the ideal $I(\si)$ defined to be the largest ideal $I$ such that $R_I\sub\si$; when $I(\si)$ is non-empty, it is the unique $\si$-class that is an ideal \cite[Lemma 6.1.3]{GMbook}.

For a set of pairs $\Om\sub M\times M$ we write $\Om^\sharp$ for the congruence on $M$ generated by $\Om$.  We write $\pc xy$ for $\Om^\sharp$ when $\Om=\big\{(x,y)\big\}$, and refer to such a congruence as 
\emph{principal}.

Recall also that \emph{Green's relations} on the monoid $M$ are defined, for $x,y\in M$, by
\[
x\rR y \iff xM=yM \COMMA x\rL y \iff Mx=My \COMMA x\rJ y \iff MxM = MyM,
\]
and further $\H = \R\cap\L$ and $\D = \R\vee\L=\R\circ\L=\L\circ\R$.  
The set $M/\J=\set{J_x}{x\in M}$ of all $\J$-classes of $M$ has a partial order $\leq$ defined, for $x,y\in M$, by
\[
J_x\leq J_y \iff x\in MyM.
\]

\subsection[The partition monoid $\P_n$]{The partition monoid \boldmath{$\P_n$}}\label{subsec:Pn}

Let $n\geq 1$ and write $\bn=\{1,\dots,n\}$ and $\bnz=\bn\cup\{0\}$.  We also fix two disjoint copies of $\bn$, namely $\bn'=\{1',\ldots,n'\}$ and $\bn''=\{1'',\ldots,n''\}$.  The elements of the \emph{partition monoid} $\P_n$ are the set partitions of $\bn\cup\bn'$; as usual, such a partition is identified with any graph on vertex set $\bn\cup\bn'$ whose components are the blocks of the partition.  
Given two partitions $\al,\be\in\P_n$, the product $\al\be$ is defined as follows.  First, let $\al_\downarrow$ be the graph on vertex set $\bn\cup\bn''$ obtained by changing every lower vertex $x'$ of $\al$ to $x''$, and let $\be^\uparrow$ be the graph on vertex set $\bn''\cup\bn'$ obtained by changing every upper vertex $x$ of $\be$ to $x''$.  The \emph{product graph} of the pair $(\al,\be)$ is the graph $\Ga(\al,\be)$ on vertex set $\bn\cup\bn''\cup\bn'$ whose edge set is the union of the edge sets of $\al_\downarrow$ and $\be^\uparrow$.  We then define $\al\be$ to be the partition of $\bn\cup\bn'$ such that vertices $x,y\in\bn\cup\bn'$ belong to the same block of $\al\be$ if and only if $x,y$ belong to the same connected component of $\Ga(\al,\be)$. Partitions and the formation of the product can be visualised as shown in
 Figure \ref{fig:P6} .

\begin{figure}[ht]
\begin{center}
\begin{tikzpicture}[scale=.5]

\begin{scope}[shift={(0,0)}]	
\uvs{1,...,6}
\lvs{1,...,6}
\uarcx14{.6}
\uarcx23{.3}
\uarcx56{.3}
\darc12
\darcx26{.6}
\darcx45{.3}
\stline34
\draw(0.6,1)node[left]{$\alpha=$};
\draw[->](7.5,-1)--(9.5,-1);
\end{scope}

\begin{scope}[shift={(0,-4)}]	
\uvs{1,...,6}
\lvs{1,...,6}
\uarc12
\uarc34
\darc45
\darc56
\darc23
\stline31
\stline55
\draw(0.6,1)node[left]{$\beta=$};
\end{scope}

\begin{scope}[shift={(10,-1)}]	
\uvs{1,...,6}
\lvs{1,...,6}
\uarcx14{.6}
\uarcx23{.3}
\uarcx56{.3}
\darc12
\darcx26{.6}
\darcx45{.3}
\stline34
\draw[->](7.5,0)--(9.5,0);
\end{scope}

\begin{scope}[shift={(10,-3)}]	
\uvs{1,...,6}
\lvs{1,...,6}
\uarc12
\uarc34
\darc45
\darc56
\stline31
\stline55
\darc23
\end{scope}

\begin{scope}[shift={(20,-2)}]	
\uvs{1,...,6}
\lvs{1,...,6}
\uarcx14{.6}
\uarcx23{.3}
\uarcx56{.3}
\darc14
\darc45
\darc56
\stline21
\darcx23{.2}
\draw(6.4,1)node[right]{$=\alpha\beta$};
\end{scope}

\end{tikzpicture}
\caption{Multiplication of two partitions in $\P_6$.}
\label{fig:P6}
\end{center}
\end{figure}
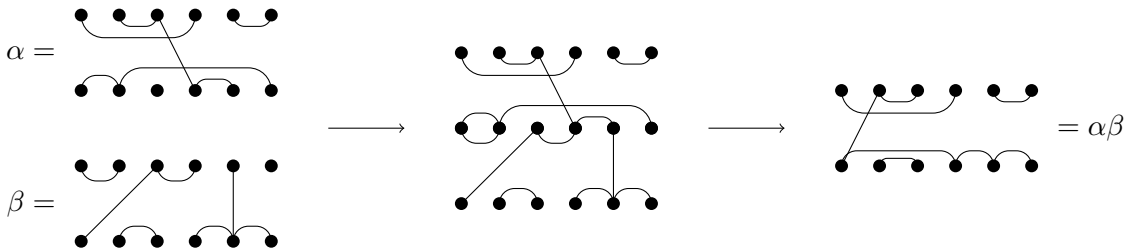

A block of a partition $\al\in\P_n$ is called a \emph{transversal} if it contains both dashed and un-dashed elements; any other block is either an \emph{upper non-transversal} (only un-dashed elements) or a \emph{lower non-transversal} (only dashed elements).  

The \emph{(co)domain} and \emph{(co)kernel} of $\al$ are defined by:
\begin{align*}
\dom\al &:= \set{x\in\bn}{x \text{ belongs to a transversal of }\al}, \\
\codom\al &:= \set{x\in\bn}{x' \text{ belongs to a transversal of }\al}, \\
\ker\al &:= \set{(x,y)\in\bn\times\bn}{x\text{ and }y \text{ belong to the same block of }\al},\\
\coker\al &:= \set{(x,y)\in\bn\times\bn}{x'\text{ and }y' \text{ belong to the same block of }\al}.
\end{align*}
The \emph{rank} of $\al$, denoted $\rank\al$, is the number of transversals of $\al$.  
We will typically use the following characterisation of Green's relations on $\P_n$ from \cite{Wilcox2007,FL2011} without explicit reference.

\begin{lemma}\label{la:Green_Pn}
For $\al,\be\in\P_n$, we have
\ben
\item $\al \rR \be \iff \dom\al=\dom\be$ and $\ker\al=\ker\be$,
\item $\al \rL \be \iff \codom\al=\codom\be$ and $\coker\al=\coker\be$,
\item $\al \rD \be \iff \al \rJ \be \iff \rank\al=\rank\be$.  
\een
The $\D=\J$-classes and non-empty ideals of $\P_n$ are the sets
\[
D_q := \set{\al\in\P_n}{\rank\al=q} \AND I_q := \set{\al\in\P_n}{\rank\al\leq q}  \qquad\text{for $q\in\bnz$,}
\]
and these are ordered by $D_q \leq D_r \iff I_q \sub I_r \iff q\leq r$.  \epfres
\end{lemma}

The above notation for the $\D$-classes and ideals of $\P_n$ will be fixed throughout the paper.

Given a partition $\al\in\P_n$, we write
\[
\al = \begin{partn}{6} A_1&\dots&A_q&C_1&\dots&C_s\\ \hhline{~|~|~|-|-|-} B_1&\dots&B_q&E_1&\dots&E_t\end{partn}
\]
to indicate that $\al$ has transversals $A_i\cup B_i'$ ($1\leq i\leq q$), upper non-transversals $C_i$ ($1\leq i\leq s$), and lower non-transversals $E_i'$ ($1\leq i\leq t$).  Here for any $A\sub\bn$ we write $A'=\set{a'}{a\in A}$.

\subsection[Congruences of $\P_n$]{Congruences of \boldmath{$\P_n$}}\label{subsec:CongPn}

We now present a brief account of the classification of congruences of the partition monoid $\P_n$ from~\cite{EMRT2018}.  First, we have a map
\[
\P_n \to D_0 : \al=\begin{partn}{6} A_1&\dots&A_q&C_1&\dots&C_s\\ \hhline{~|~|~|-|-|-} B_1&\dots&B_q&E_1&\dots&E_t\end{partn} \mt \wh\al = \begin{partn}{6} A_1&\dots&A_q&C_1&\dots&C_s\\ \hhline{-|-|-|-|-|-} B_1&\dots&B_q&E_1&\dots&E_t\end{partn},
\]
whose effect is to break apart all transversals of $\al$ into their upper and lower parts.  
Next we have a family of relations on $D_q$ ($2\leq q\leq n$), denoted $\nu_N$, indexed by normal subgroups $N$ of the symmetric group $\S_q$.  To define these relations consider a pair $(\al,\be)$ of $\H$-related elements from $D_q$: 
\[
\al = \begin{partn}{6} A_1&\dots&A_q&C_1&\dots&C_s\\ \hhline{~|~|~|-|-|-} B_1&\dots&B_q&E_1&\dots&E_t\end{partn}
\AND
\be = \begin{partn}{6} A_1&\dots&A_q&C_1&\dots&C_s\\ \hhline{~|~|~|-|-|-} B_{1\pi}&\dots&B_{q\pi}&E_1&\dots&E_t\end{partn} \qquad\text{for some $\pi\in\S_q$.}
\]
We then define $\pd(\al,\be)=\pi$, which we think of as the \emph{permutational difference} of $\al$ and~$\be$.  Note that~$\pd(\al,\be)$ is only well-defined up to conjugacy in $\S_q$, as $\pi$ depends on the above ordering on the transversals of $\al$ and $\be$.  Nevertheless, for any normal subgroup $N\normal\S_q$, we have a well-defined equivalence relation (see~\cite[Lemmas 3.17 and 5.6]{EMRT2018}):
\[
\nu_N = \bigset{(\al,\be)\in\H\restr_{D_q}}{\pd(\al,\be)\in N}.
\]
As extreme cases, note that $\nu_{\S_q}=\H\restr_{D_q}$ and $\nu_{\{\id_q\}}=\De_{D_q}$.

\begin{thm}[{\cite[Theorem 5.4]{EMRT2018}}]
\label{thm:CongPn}
For $n\geq1$, the congruences on the partition monoid~$\P_n$ are precisely:
\bit
\item
the Rees congruences $R_q:= R_{I_q} = \bigset{ (\alpha,\beta)\in\P_n\times\P_n}{ \alpha=\beta\text{ or } \rank\alpha,\rank\beta\leq q}$ for $q\in\{0,\dots,n\}$, including $\nab_{\P_n} = R_n$;
\item
the relations $R_N:=R_{q-1}\cup\nu_N$ for $q\in\{2,\dots,n\}$ and $\{\id_q\}\neq N\unlhd \S_q$;
\item
the relations 
\begin{align*}
\lam_q &:= \bigset{ (\alpha,\beta)\in I_q\times I_q}{ \widehat{\alpha}\rL\widehat{\beta}}\cup\Delta_{\P_n}, \\
\rho_q &:= \bigset{ (\alpha,\beta)\in I_q\times I_q}{ \widehat{\alpha}\rR\widehat{\beta}}\cup\Delta_{\P_n}, \\
\mu_q &:= \bigset{ (\alpha,\beta)\in I_q\times I_q}{ \widehat{\alpha}=\widehat{\beta}}\cup\Delta_{\P_n}, 
\end{align*}
for $q\in\{0,1\}$, including $\De_{\P_n} = \mu_0$, and the relations
\[
\lam_{\S_2}:=\lam_1\cup\nu_{\S_2} \COMMA \rho_{\S_2}:=\rho_1\cup\nu_{\S_2} \COMMA \mu_{\S_2}:=\mu_1\cup\nu_{\S_2}.
\]
\eit
The congruence lattice $\Cong(\P_n)$ is shown in Figure \ref{fig:CongPn}.  
 \epfres
\end{thm}

The above notation for the congruences of $\P_n$ will be fixed and used throughout the paper.

\subsection{Twisted partition monoids}\label{subsec:Ptw}

For two partitions
$\alpha,\beta\in\P_n$ denote by $\Float(\alpha,\beta)$ the number of
\emph{floating components} of the product graph $\Gamma(\alpha,\beta)$, i.e.~the number of components that are wholly contained in the middle row, $\bn''$. 
The \emph{twisted partition monoid} $\Ptw n$ is defined by
\[
\Ptw{n}:=\N\times\P_n
\qquad\text{with product}\qquad
(i,\alpha)\multw(j,\beta):= \big(i+j+\Float(\alpha,\beta),\alpha\beta\big).
\]%
That the above multiplication is associative follows from \cite[Lemma~4.1]{FL2011}.  Green's relations on~$\Ptw{n}$ work as follows:

\begin{lemma}[{\cite[Lemma 2.10]{ERtwisted1}}]     \label{la:Green_Ptwn}
If $\K$ is any of Green's relations, and if $\al,\be\in\P_n$ and $i,j\in\N$, then
\[
(i,\al) \rK (j,\be) \text{ in } \Ptw n \IFF i=j \ANd \al \rK \be \text{ in } \P_n.
\]
The $\D=\J$-classes and principal ideals of $\Ptw{n}$ are the sets
\[
D_{qi} := \{i\}\times D_q \AND I_{qi} := \{i,i+1,i+2,\dots\} \times I_q  \qquad\text{for $q\in\bnz$ and $i\in\N$,}
\]
and these are ordered by $D_{qi} \leq D_{rj} \iff I_{qi} \sub I_{rj} \iff q\leq r$ and $i\geq j$.
\end{lemma}

It follows that the poset $(\Ptw n/\D,\leq)$ of $\J=\D$-classes is isomorphic to the direct product $(\bnz,\leq)\times(\N,\geq)$, and that an arbitrary ideal of $\Ptw n$ has the form $I_{q_1i_1}\cup\cdots\cup I_{q_ki_k}$, for a sequence of incomparable elements $(q_1,i_1),\ldots,(q_k,i_k)$ of this poset.  
As in \cite{ERtwisted1}, we often view $\Ptw{n}$ as a rectangular grid of $\D$-classes indexed by~$\bnz\times\N$, as in Figure \ref{fig:gridid}, which also pictures the ideal $I_{33}\cup I_{12}\cup I_{00}$ of $\Ptw4$.

\begin{figure}[ht]
\begin{center}
$
\begin{array}{c|*{5}{>{$}m{6mm}<{$}|}c|}\hhline{~|-|-|-|-|-|-|}
\renewcommand{\arraystretch}{2}
\text{\scriptsize 4} &\cellcolor{delcol}D_{40} &\cellcolor{delcol} D_{41}& \cellcolor{delcol} D_{42}& \cellcolor{delcol}D_{43}&\cellcolor{delcol}D_{44}&\cellcolor{delcol}\cdots\\ 
\hhline{~|-|-|-|-|-|-|}
\renewcommand{\arraystretch}{2}
\text{\scriptsize 3} &\cellcolor{delcol} D_{30}&\cellcolor{delcol}D_{31} & \cellcolor{delcol}D_{32} & \cellcolor{Rcol}D_{33}&\cellcolor{Rcol}D_{34}&\cellcolor{Rcol}\cdots\\
\hhline{~|-|-|-|-|-|-|}
\renewcommand{\arraystretch}{2}
\text{\scriptsize 2} &\cellcolor{delcol}D_{20} &\cellcolor{delcol}D_{21} & \cellcolor{delcol}D_{22} & \cellcolor{Rcol}D_{23}&\cellcolor{Rcol}D_{24}&\cellcolor{Rcol}\cdots\\
\hhline{~|-|-|-|-|-|-|}
\renewcommand{\arraystretch}{2}
\text{\scriptsize 1} &\cellcolor{delcol}D_{10} &\cellcolor{delcol}D_{11} & \cellcolor{Rcol}D_{12} & \cellcolor{Rcol}D_{13}&\cellcolor{Rcol}D_{14}&\cellcolor{Rcol}\cdots\\
\hhline{~|-|-|-|-|-|-|}
\renewcommand{\arraystretch}{2}
\text{\scriptsize 0} &\cellcolor{Rcol}D_{00} &\cellcolor{Rcol}D_{01} & \cellcolor{Rcol} D_{02}& \cellcolor{Rcol}D_{03}&\cellcolor{Rcol}D_{04}&\cellcolor{Rcol}\cdots\\
\hhline{~------}
\multicolumn{1}{c}{\text{\scriptsize $q/i$}}&\multicolumn{1}{c}{\text{\scriptsize 0}}&\multicolumn{1}{c}{\text{\scriptsize 1}}&\multicolumn{1}{c}{\text{\scriptsize 2}}&\multicolumn{1}{c}{\text{\scriptsize 3}}&\multicolumn{1}{c}{\text{\scriptsize 4}}&\multicolumn{1}{c}{\cdots}
\end{array}
$
\\
\caption{$\Ptw{4}$ as a grid, and the ideal determined by $(0,0)$, $(1,2)$ and $(3,3)$.}
\label{fig:gridid}
\end{center}
\end{figure}

In addition to the infinite monoid $\Ptw{n}$, we are also interested in its finite homomorphic images~$\Ptw{n,d}$, called the 
\emph{$d$-twisted partition monoids}.  For $d\geq0$, $\Ptw{n,d}$ is defined as the Rees quotient
\[
\Ptw{n,d} := \Ptw n/R_{I_{n,d+1}}
\]
by the  ideal $I_{n,d+1} = \{d+1,d+2,\dots\} \times \P_n$.  We can also think of $\Ptw{n,d}$ as being $\Ptw{n}$ with all elements with more than $d$ floating components equated to a zero element $\zero$. Thus we may take~$\Ptw{n,d}$ to be the set
\[
\Ptw{n,d}:=(\bdz\times\P_n)\cup\{\zero\},
\]
with multiplication
\begin{equation}\label{eq:multwk}
\ba\multwk{d}\bb:=
\begin{cases} 
\ba\multw\bb & \text{if } \ba=(i,\alpha),\ \bb=(j,\beta) \text{ and } i+j+\Float(\alpha,\beta)\leq d,\\
\zero & \text{otherwise}. 
\end{cases}
\end{equation}
In this interpretation, $\Ptw{n,d}$ consists of columns $0,1,\ldots,d$ of $\Ptw n$, plus the zero element $\zero$.

\section{Congruences of the twisted partition monoid \boldmath{$\Ptw{n}$}}
\label{sec:CongPtwn}

We now recount the description of congruences of $\Ptw{n}$ from \cite{ERtwisted1}.
It will involve congruences on the additive monoid $\N$ of natural numbers, and we recall that every such non-trivial congruence~$\theta$ has the form
\[
\th = (m,m+d)^\sharp = \De_\N \cup \bigset{ (i,j)\in\N\times\N}{i,j\geq m,\ i\equiv j\!\!\!\!\pmod{d}} \qquad\text{for some $m\geq0$ and $d\geq1$.}
\]
Here $m=\min\th$ is the \emph{minimum} of $\th$, and $d=\per\th$ is the \emph{period} of $\th$.  
For the trivial congruence we define $\min\De_\N=\per\De_\N=\infty$.  If~$\th_1$ and $\th_2$ are congruences on $\N$, then 
\begin{equation}
\label{eq:th1th2}
\th_1\sub\th_2 \IFF \min\th_1\geq\min\th_2 \ANd \per\th_2\mid\per\th_1.
\end{equation}
Here $\mid$ is the division relation on $\N\cup\{\infty\}$, with the understanding that every element of this set divides $\infty$.

\begin{defn}[\bf C-pair, C-chain, C-matrix]
\label{de:Cpair}
An ordered pair $\Pair=(\Theta,M)$ is called a \emph{C-pair} if the following are satisfied:
\bit
\item
$\Th=(\th_0,\ldots,\th_n)$ is a chain of congruences on $\N$ satisfying $\th_0\supseteq\cdots\supseteq\th_n$.  
\item
$M=(M_{qi})_{\bnz\times\N}$ is a matrix with entries drawn from the following set of symbols:
\[
\{\Delta,\muup,\mudown,\mu,\lambda,\rho, R\}\cup\set{ N}{ \{\id_q\}\neq N\unlhd \S_q,\ 2\leq q\leq n}.
\]
We refer to the entries in the second set collectively as the \emph{$N$-symbols}.
\item
Rows $0$ and $1$ of $M$ must be of one of the row types \ref{RT1}--\ref{RT7} shown in Table \ref{tab:RT}.
\item
Every row $q\geq 2$ must be of one of the row types \ref{RT8}--\ref{RT10} from Table \ref{tab:RT}.
\item
An $N$-symbol cannot be immediately above $\Delta$, $\muup$, $\mudown$ or another $N$-symbol.
\item
Every entry equal to $R$ in row $q\geq 1$ must be directly above an $R$ entry from row~$q-1$.  
\eit
The last two items were referred to as the  \emph{verticality conditions} in \cite{ERtwisted1}, and denoted (V1) and~(V2).
We refer to $\Theta$ as a \emph{C-chain}, and to $M$ as a \emph{C-matrix}.  
With a slight abuse of terminology, we will say that $M$ is of type \ref{RT1}--\ref{RT7}, as appropriate, according to the type of rows $0$ and $1$.
\end{defn}

Some concrete examples of C-pairs can be seen in Examples \ref{ex:fg1}--\ref{ex:fg4} further on.

\begin{table}[ht]
\begin{center}
\footnotesize
\begin{tabular}{|l|l|l|l|}
\hline
\bf Type & \bf Row(s) & \bf\boldmath Relationship with $\th_q$ & \bf Further conditions\\
\hline\hline
%RT1
\bf\refstepcounter{RT}\theRT
&
$
\Cmatsetup
\begin{array}{r|c|c|c|c|}\hhline{~|-|-|-|-|}
\text{\scriptsize 1} &\cellcolor{delcol}\Delta & \cellcolor{delcol}\Delta & \cellcolor{delcol}\Delta& \cellcolor{delcol}\dots\\ \hhline{~|-|-|-|-|}
\text{\scriptsize 0}&\cellcolor{delcol}\Delta & \cellcolor{delcol}\Delta & \cellcolor{delcol}\Delta& \cellcolor{delcol}\dots\\ \hhline{~|-|-|-|-|}
\end{array}\rule[-5mm]{0mm}{12mm}
$
&
&\\
\hline
%RT2
\bf\refstepcounter{RT}\theRT
&
$
\Cmatsetup
\begin{array}{r|c|c|c|c|c|c|c|}\hhline{~|-|-|-|-|-|-|-|}
\renewcommand{\arraystretch}{2}
\text{\scriptsize 1} &\cellcolor{delcol}\Delta &\cellcolor{delcol}\dots & \cellcolor{delcol}\Delta & \cellcolor{excepcol}\zeta&\cellcolor{mucol}\mu&\cellcolor{mucol}\mu& \cellcolor{mucol}\dots\\ \hhline{~|-|-|-|-|-|-|-|}
\text{\scriptsize 0}&\cellcolor{delcol}\Delta &\cellcolor{delcol}\dots &\cellcolor{delcol}\Delta & \cellcolor{mucol}\mu&\cellcolor{mucol}\mu&\cellcolor{mucol}\mu& \cellcolor{mucol}\dots\\ \hhline{~|-|-|-|-|-|-|-|}
\multicolumn{1}{c}{}&\multicolumn{1}{c}{}&\multicolumn{1}{c}{}&\multicolumn{1}{c}{}&\multicolumn{1}{c}{\text{\scriptsize $i$}}&\multicolumn{1}{c}{}&\multicolumn{1}{c}{}&\multicolumn{1}{c}{}
\end{array}\rule[-5mm]{0mm}{15mm}
$
& $\th_0=\th_1=\De_\N$ &
$\ze\in\{\De,\muup,\mudown,\mu\}$ 
 \\
\hline
%RT3
\bf\refstepcounter{RT}\theRT
&
$
\Cmatsetup
\begin{array}{r|c|c|c|c|c|c|c|}\hhline{~|-|-|-|-|-|-|-|}
\renewcommand{\arraystretch}{2}
\text{\scriptsize 1} &\cellcolor{delcol}\Delta &\cellcolor{delcol}\dots & \cellcolor{delcol}\Delta & \cellcolor{delcol}\Delta&\cellcolor{delcol}\Delta&\cellcolor{delcol}\Delta& \cellcolor{delcol}\dots\\ \hhline{~|-|-|-|-|-|-|-|}
\text{\scriptsize 0}&\cellcolor{delcol}\Delta &\cellcolor{delcol}\dots &\cellcolor{delcol}\Delta & \cellcolor{Rcol}\xi&\cellcolor{Rcol}\xi&\cellcolor{Rcol}\xi& \cellcolor{Rcol}\dots\\ \hhline{~|-|-|-|-|-|-|-|}
\multicolumn{1}{c}{}&\multicolumn{1}{c}{}&\multicolumn{1}{c}{}&\multicolumn{1}{c}{}&\multicolumn{1}{c}{\text{\scriptsize $m$}}&\multicolumn{1}{c}{}&\multicolumn{1}{c}{}&\multicolumn{1}{c}{}
\end{array}\rule[-5mm]{0mm}{15mm}
$
& $\th_0=(m,m+1)^\sharp$
& $\xi\in\{\rho,\lam,R\}$ \\
\hline
%RT4
\bf\refstepcounter{RT}\theRT
&
$
\Cmatsetup
\begin{array}{r|c|c|c|c|c|c|c|}\hhline{~|-|-|-|-|-|-|}
\renewcommand{\arraystretch}{2}
\text{\scriptsize 1} &\cellcolor{delcol}\Delta &\cellcolor{delcol}\dots & \cellcolor{delcol}\Delta & \cellcolor{Rcol}\xi&\cellcolor{Rcol}\xi& \cellcolor{Rcol}\dots\\ \hhline{~|-|-|-|-|-|-|}
\text{\scriptsize 0}&\cellcolor{delcol}\Delta &\cellcolor{delcol}\dots &\cellcolor{delcol}\Delta & \cellcolor{Rcol}\xi&\cellcolor{Rcol}\xi& \cellcolor{Rcol}\dots\\ \hhline{~|-|-|-|-|-|-|}
\multicolumn{1}{c}{}&\multicolumn{1}{c}{}&\multicolumn{1}{c}{}&\multicolumn{1}{c}{}&\multicolumn{1}{c}{\text{\scriptsize $m$}}&\multicolumn{1}{c}{}&\multicolumn{1}{c}{}
\end{array}\rule[-5mm]{0mm}{15mm}
$
& \makecell[l]{$\th_0=\th_1=(m,m+d)^\sharp$}
& \makecell[l]{$\xi\in\{\mu,\rho,\lam,R\}$ if $d=1$,\\ $\xi=\mu$ if $d>1$} \\
\hline
%RT5
\bf\refstepcounter{RT}\theRT
&
$
\Cmatsetup
\begin{array}{r|c|c|c|c|c|c|c|c|c|c|c|c|}\hhline{~|-|-|-|-|-|-|-|-|-|-|-|}
\renewcommand{\arraystretch}{2}
\text{\scriptsize 1} &\cellcolor{delcol}\Delta &\cellcolor{delcol}\dots & \cellcolor{delcol}\Delta & 
\cellcolor{excepcol}\zeta & \cellcolor{mucol} \mu & \cellcolor{mucol} \dots &\cellcolor{mucol} \mu &\cellcolor{mucol} \mu &
\cellcolor{Rcol}\xi&\cellcolor{Rcol}\xi& \cellcolor{Rcol}\dots\\ \hhline{~|-|-|-|-|-|-|-|-|-|-|-|}
\text{\scriptsize 0}&\cellcolor{delcol}\Delta &\cellcolor{delcol}\dots &\cellcolor{delcol}\Delta & 
\cellcolor{mucol}\mu & \cellcolor{mucol} \mu & \cellcolor{mucol} \dots &\cellcolor{mucol} \mu &
\cellcolor{Rcol} \xi &\cellcolor{Rcol}\xi&\cellcolor{Rcol}\xi& \cellcolor{Rcol}\dots\\ \hhline{~|-|-|-|-|-|-|-|-|-|-|-|}
\multicolumn{1}{c}{}&\multicolumn{1}{c}{}&\multicolumn{1}{c}{}&\multicolumn{1}{c}{}&\multicolumn{1}{c}{\text{\scriptsize $i$}}&\multicolumn{1}{c}{}&\multicolumn{1}{c}{}&\multicolumn{1}{c}{}&\multicolumn{1}{c}{\text{\scriptsize $m$}}&\multicolumn{1}{c}{}&\multicolumn{1}{c}{}
\end{array}\rule[-5mm]{0mm}{15mm}
$
& \makecell[l]{$\th_1=(m+1,m+1+d)^\sharp$, \\  $\th_0=(m,m+d)^\sharp$}
& \makecell[l]{$\ze\in\{\De,\muup,\mudown,\mu\}$,\\   $\xi\in\{\mu,\rho,\lam,R\}$ if $d=1$,\\ $\xi=\mu$ if $d>1$,\\   $i<m$ }\\
\hline
%RT6
\bf\refstepcounter{RT}\theRT
&
$
\Cmatsetup
\begin{array}{r|c|c|c|c|c|c|c|c|c|}\hhline{~|-|-|-|-|-|-|-|-|-|}
\renewcommand{\arraystretch}{2}
\text{\scriptsize 1} &\cellcolor{delcol}\Delta &\cellcolor{delcol}\dots & \cellcolor{delcol}\Delta & 
\cellcolor{delcol}\Delta & \cellcolor{delcol}\dots & \cellcolor{delcol}\Delta &
\cellcolor{excepcol} \zeta & \cellcolor{Rcol}\xi& \cellcolor{Rcol}\dots
\\ \hhline{~|-|-|-|-|-|-|-|-|-|}
\text{\scriptsize 0}&\cellcolor{delcol}\Delta &\cellcolor{delcol}\dots &\cellcolor{delcol}\Delta & 
\cellcolor{Rcol} \xi &\cellcolor{Rcol}\dots&\cellcolor{Rcol}\xi&\cellcolor{Rcol}\xi&\cellcolor{Rcol}\xi& \cellcolor{Rcol}\dots
\\ \hhline{~|-|-|-|-|-|-|-|-|-|}
\multicolumn{1}{c}{}&\multicolumn{1}{c}{}&\multicolumn{1}{c}{}&\multicolumn{1}{c}{}&\multicolumn{1}{c}{\text{\scriptsize $m$}}&\multicolumn{1}{c}{}&\multicolumn{1}{c}{}&\multicolumn{1}{c}{}&\multicolumn{1}{c}{\text{\scriptsize $l$}}&\multicolumn{1}{c}{}
\end{array}\rule[-5mm]{0mm}{15mm}
$
& \makecell[l]{$\th_1=(l,l+d)^\sharp$, \\  $\th_0=(m,m+d)^\sharp$}
& \makecell[l]{$\ze\in\{\De,\muup,\mudown,\mu\}$,\\   $\xi\in\{\mu,\rho,\lam,R\}$ if $d=1$,\\ $\xi=\mu$ if $d>1$,\\ $m<l$}\\
\hline
%RT7
\bf\refstepcounter{RT}\theRT
&
$
\Cmatsetup
\begin{array}{r|c|c|c|c|c|c|c|c|c|c|}\hhline{~|-|-|-|-|-|-|-|-|-|-|}
\renewcommand{\arraystretch}{2}
\text{\scriptsize 1} &\cellcolor{delcol}\Delta &\cellcolor{delcol}\dots & \cellcolor{delcol}\Delta  & \cellcolor{delcol}\Delta &
\cellcolor{delcol}\Delta & \cellcolor{delcol}\dots  & \cellcolor{delcol}\Delta &
\cellcolor{mucol} \mu & \cellcolor{Rcol}\xi& \cellcolor{Rcol}\dots
\\ \hhline{~|-|-|-|-|-|-|-|-|-|-|}
\text{\scriptsize 0}&\cellcolor{delcol}\Delta &\cellcolor{delcol}\dots &\cellcolor{delcol}\Delta  
&\cellcolor{mucol}\mu&
\cellcolor{Rcol} \xi &\cellcolor{Rcol}\dots&\cellcolor{Rcol}\xi&\cellcolor{Rcol}\xi&\cellcolor{Rcol}\xi& \cellcolor{Rcol}\dots
\\ \hhline{~|-|-|-|-|-|-|-|-|-|-|}
\multicolumn{1}{c}{}&\multicolumn{1}{c}{}&\multicolumn{1}{c}{}&\multicolumn{1}{c}{}&\multicolumn{1}{c}{}&\multicolumn{1}{c}{\text{\scriptsize $m$}}&\multicolumn{1}{c}{}&\multicolumn{1}{c}{}&\multicolumn{1}{c}{}&\multicolumn{1}{c}{\text{\scriptsize $l$}}
\end{array}\rule[-5mm]{0mm}{15mm}
$
& \makecell[l]{$\th_1=(l,l+d)^\sharp$, \\  $\th_0=(m,m+d)^\sharp$}
& \makecell[l]{$\xi\in\{\mu,\rho,\lam,R\}$ if $d=1$,\\ $\xi=\mu$ if $d>1,$\\ $0<m<l-1$, \\ $m\equiv l-1\pmod d$}\\
\hline\hline
%RT8
\bf\refstepcounter{RT}\theRT
&
$
\Cmatsetup
\begin{array}{r|c|c|c|c|}\hhline{~|-|-|-|-|}
\text{\scriptsize $q$} &\cellcolor{delcol}\Delta & \cellcolor{delcol}\Delta & \cellcolor{delcol}\Delta& \cellcolor{delcol}\dots\\ \hhline{~|-|-|-|-|}
\end{array}\rule[-3mm]{0mm}{9mm}
$
& 
& \\
\hline
%RT9
\bf\refstepcounter{RT}\theRT
&
$
\Cmatsetup
\begin{array}{r|c|c|c|c|c|c|c|c|c|c|}\hhline{~|-|-|-|-|-|-|-|-|-|-|}
\renewcommand{\arraystretch}{2}
\text{\scriptsize $q$} &\cellcolor{delcol}\Delta &\cellcolor{delcol}\dots & \cellcolor{delcol}\Delta  & 
\cellcolor{Ncol} N_i &\cellcolor{Ncol} N_{i+1} & \cellcolor{Ncol}\dots & \cellcolor{Ncol} N_{k-1} &
\cellcolor{Ncol} N & \cellcolor{Ncol} N& \cellcolor{Ncol}\dots
\\ \hhline{~|-|-|-|-|-|-|-|-|-|-|}
\multicolumn{1}{c}{}&\multicolumn{1}{c}{}&\multicolumn{1}{c}{}&\multicolumn{1}{c}{}&\multicolumn{1}{c}{\text{\scriptsize $i$}}&\multicolumn{1}{c}{}&\multicolumn{1}{c}{}&\multicolumn{1}{c}{}&\multicolumn{1}{c}{\text{\scriptsize $k$}}&\multicolumn{1}{c}{}&\multicolumn{1}{c}{}
\end{array}\rule[-3mm]{0mm}{11mm}
$
& $\min\th_q\geq k$
& \makecell[l]{ $N_i,\dots,N_{k-1},N\unlhd\S_q$, \\ $N_i\leq \dots \leq N_{k-1}\leq N$}\\
\hline
%RT10
\bf\refstepcounter{RT}\theRT
&
$
\Cmatsetup
\begin{array}{r|c|c|c|c|c|c|c|c|c|c|}\hhline{~|-|-|-|-|-|-|-|-|-|-|}
\renewcommand{\arraystretch}{2}
\text{\scriptsize $q$} &\cellcolor{delcol}\Delta &\cellcolor{delcol}\dots & \cellcolor{delcol}\Delta  & 
\cellcolor{Ncol} N_i &\cellcolor{Ncol} N_{i+1} & \cellcolor{Ncol}\dots & \cellcolor{Ncol} N_{m-1} &
\cellcolor{Rcol} R & \cellcolor{Rcol} R& \cellcolor{Rcol}\dots
\\ \hhline{~|-|-|-|-|-|-|-|-|-|-|}
\multicolumn{1}{c}{}&\multicolumn{1}{c}{}&\multicolumn{1}{c}{}&\multicolumn{1}{c}{}&\multicolumn{1}{c}{\text{\scriptsize $i$}}&\multicolumn{1}{c}{}&\multicolumn{1}{c}{}&\multicolumn{1}{c}{}&\multicolumn{1}{c}{\text{\scriptsize $m$}}&\multicolumn{1}{c}{}&\multicolumn{1}{c}{}
\end{array}\rule[-3mm]{0mm}{11mm}
$
& $\th_q=(m,m+1)^\sharp$
& \makecell[l]{ $N_i,\dots,N_{m-1}\unlhd\S_q$, \\ $N_i\leq\dots\leq N_{m-1}$}\\
\hline
\end{tabular}
\normalsize
\caption{The specification of row types in a C-matrix.}
\label{tab:RT}
\end{center}
\end{table}

\setcounter{RT}{0}
\refstepcounter{RT}\label{RT1}
\refstepcounter{RT}\label{RT2}
\refstepcounter{RT}\label{RT3}
\refstepcounter{RT}\label{RT4}
\refstepcounter{RT}\label{RT5}
\refstepcounter{RT}\label{RT6}
\refstepcounter{RT}\label{RT7}
\refstepcounter{RT}\label{RT8}
\refstepcounter{RT}\label{RT9}
\refstepcounter{RT}\label{RT10}

\begin{defn}[\bf Congruence corresponding to a C-pair]
\label{de:cg}
The congruence associated with a C-pair $(\Th,M)$ is the relation $\cg(\Th,M)$ on $\Ptw{n}$ consisting of all pairs $((i,\alpha),(j,\beta))\in \Ptw{n}\times \Ptw{n}$ such that one of the following holds, writing $q=\rank\alpha$ and $r=\rank\beta$:
\begin{enumerate}[label=\textsf{(C\arabic*)}, widest=(C8), leftmargin=10mm]
\item
\label{C1}
$M_{qi}=M_{rj}=\De$, $(i,j)\in\th_q$ and $\alpha=\beta$;
\item
\label{C2}
$M_{qi}=M_{rj}=R$;
\item
\label{C3}
$M_{qi}=M_{rj}=N$, $(i,j)\in \th_q$, $\alpha\rH\beta$ and $\pd(\alpha,\beta)\in N$;
\item
\label{C4}
$M_{qi}=M_{rj}=\lambda$ and $\widehat{\alpha}\rL\widehat{\beta}$;
\item
\label{C5}
$M_{qi}=M_{rj}=\rho$ and $\widehat{\alpha}\rR\widehat{\beta}$;
\item
\label{C6}
$M_{qi}=M_{rj}=\mudown$,  $\widehat{\alpha}=\widehat{\beta}$ and $\alpha \rL\beta$;
\item
\label{C7}
$M_{qi}=M_{rj}=\muup$,  $\widehat{\alpha}=\widehat{\beta}$ and $\alpha \rR\beta$;
\item
\label{C8}
$M_{qi}=M_{rj}=\mu$,  $\widehat{\alpha}=\widehat{\beta}$ and one of the following holds:
\bit
\item $q=r$ and $(i,j)\in\th_q$, or
\item $q\neq r$, $(i+r,j+q)\in\th_0$, $i<\min\th_q$ and $j<\min\th_r$, or
\item $q\neq r$, $(i+r,j+q)\in\th_0$, $i\geq\min\th_q$ and $j\geq\min\th_r$.
\eit
\end{enumerate}
\end{defn}

One special kind of C-pair also has a second kind of congruence associated to it:

\begin{defn}[\bf Exceptional C-pairs and congruences]
\label{def:exc}
A C-pair $(\Th,M)$ is \emph{exceptional} if there exists $q\geq 2$ such that:
\bit
\item
$\th_q=(m,m+2d)^\sharp$ for some $m\geq0$ and $d\geq1$; 
associated with this relation, we also set
$\thx_q:=(m,m+d)^\sharp$;
\item
$M_{qm}=\A_q$ if $q>2$;
\item
if $q=2$ then $M_{2m}=\Delta$, $M_{1m}\in\{\mu,\rho,\lam,R\}$ and $\thx_2\subseteq\theta_1$.
\eit
This $q$ is necessarily unique (by the verticality conditions in Definition \ref{de:Cpair}), and we call row $q$ the \emph{exceptional row}, and write $q=:\hgt(M)$.
To the exceptional C-pair $(\Th,M)$, in addition to $\cg(\Th,M)$, we also associate the \emph{exceptional congruence} $\cgx(\Th,M)$ consisting of all pairs $((i,\al),(j,\be))$ such that one of \ref{C1}--\ref{C8} holds, or else:
\begin{enumerate}[label=\textsf{(C\arabic*)}, widest=(C8), leftmargin=10mm]
\addtocounter{enumi}{8}
\item \label{C9} $(i,j)\in\thx_q\sm\th_q$, $\al\rH\be$ and $\pd(\alpha,\beta)\in \S_q\setminus \A_q$. 
\end{enumerate}
\end{defn}

The following is the main result of \cite{ERtwisted1}:

\begin{thm}
\label{thm:main}
For $n\geq1$, the congruences on the twisted partition monoid $\Ptw{n}$ are precisely:
\bit
\item
$\cg(\Th,M)$ where $(\Th,M)$ is any C-pair;
\item
$\cgx(\Th,M)$ where $(\Th,M)$ is any exceptional C-pair.  \epfres
\eit
\end{thm}

\begin{rem}\label{rem:amb}
Consider a congruence $\si$ on $\Ptw n$.  As in \cite[Section 5]{ERtwisted1}, the C-pair $(\Th,M)$ associated to~$\si$ is determined as follows.  For the C-chain $\Th=(\th_0,\ldots,\th_n)$, we have
\begin{align*}
\th_q &= \bigset{(i,j)\in\N\times\N}{((i,\al),(j,\al))\in\si \ (\exists \al\in D_q)} \\
&= \bigset{(i,j)\in\N\times\N}{((i,\al),(j,\al))\in\si \ (\forall \al\in D_q)} \qquad\text{for $q\in\bnz$.}
\end{align*}
For the C-matrix $M=(M_{qi})_{\bnz\times\N}$, an entry $M_{qi}$ is uniquely determined by the restriction
\[
\si_{qi} := \bigset{(\al,\be)\in D_q\times D_q}{((i,\al),(i,\be))\in\si},
\]
apart from two possible cases:
\bit
\item If $\si_{0i}=\De_{D_0}$, then $M_{0i}$ is either $\mu$ or $\De$, depending  on whether there are $\si$-relationships between elements of $D_{0i}$ and those of some $D_{1j}$.
\item If $\si_{ni}=\nab_{D_n}$, then $M_{ni}$ is either $R$ or $\S_n$, depending  on whether $D_{ni}$ is contained in the ideal class $I(\sigma)$ of $\sigma$.
\eit
The full relationships between the $M_{qi}$ and $\si_{qi}$ is summarised in Table \ref{tab:M}.
The two hitherto undefined relations making an appearance in the table are
\[
\muup:=\bigset{ (\alpha,\beta)\in D_1\times D_1}{ \widehat{\alpha}=\widehat{\beta},\ \alpha\rR\beta}
\quad\text{and}\quad
\mudown:=\bigset{ (\alpha,\beta)\in D_1\times D_1}{ \widehat{\alpha}=\widehat{\beta},\ \alpha\rL\beta}.
\]
\end{rem}

\begin{table}[ht]
\begin{center}
$
\begin{array}{|c|c|c|c|} 
\hline
\text{\boldmath{$q$}} & \text{\boldmath{$M_{qi}$}}& \text{\boldmath{$\si_{qi}$}}& \text{\bf Additional criteria}
\\  \hhline{|=|=|=|=|}
\multirow{3}{*}{$q\geq2$} & \De & \De_{D_q} & 
\\  \hhline{~|-|-|-|}
 & N & \nu_N & D_{qi}\not\subseteq I(\si)
\\ \hhline{~|-|-|-|}
 & R & \nab_{D_q} & D_{qi}\subseteq I(\si) 
\\  \hhline{|=|=|=|=|}
\multirow{7}{*}{$q=1$} & \De & \De_{D_1} & 
\\  \hhline{~|-|-|-|}
 & \muup & \muup & 
\\  \hhline{~|-|-|-|}
 & \mudown & \mudown & 
\\  \hhline{~|-|-|-|}
 & \mu & \mu_1\restr_{D_1} & 
\\  \hhline{~|-|-|-|}
 & \lam &\lam_1\restr_{D_1} & 
\\  \hhline{~|-|-|-|}
 & \rho & \rho_1\restr_{D_1} & 
\\ \hhline{~|-|-|-|}
 & R & \nab_{D_q} & 
\\  \hhline{|=|=|=|=|}
\multirow{5}{*}{$q=0$} & \De & \De_{D_0} & \text{$\si\cap(D_{0i}\times D_{1j})=\emptyset$ ($\forall j\in\N$})
\\  \hhline{~|-|-|-|}
 & \mu & \De_{D_0} & \text{$\si\cap(D_{0i}\times D_{1j})\neq\emptyset$ ($\exists j\in\N$})
\\  \hhline{~|-|-|-|}
 & \lam &\lam_0\restr_{D_0} & 
\\  \hhline{~|-|-|-|}
 & \rho & \rho_0\restr_{D_0} & 
\\ \hhline{~|-|-|-|}
 & R & \nab_{D_q} & 
\\ \hline
\end{array}
$
\caption{The relationship between the entries $M_{qi}$ of a C-matrix and the restrictions $\si_{qi}$ of its associated congruence(s) $\si$ on $\Ptw n$.}
\label{tab:M}
\end{center}
\end{table}

Building on the above classification of congruences on $\Ptw n$, \cite{ERtwisted1} also characterised the inclusion ordering
constituting the congruence lattice $\Cong(\Ptw{n})$, as stated below in Theorem \ref{thm:comparisons}.
This ordering is closely related, but not identical, to the lexicographic ordering on C-pairs, based on the inclusion ordering of congruences on $\N$, and the ordering of C-matrix entries shown in 
Figure~\ref{fig:M}.  Specifically, given two C-pairs $\Pair^1=(\Th^1,M^1)$ and $\Pair^2=(\Th^2,M^2)$, we write:
\bit
\item $\Th^1\leqc\Th^2 \iff \th_q^1\sub\th_q^2$ for all $q\in\bnz$,
\item $M^1\leqc M^2 \iff M_{qi}^1\leqc M_{qi}^2$ for all $q\in\bnz$ and $i\in\N$, and
\item $\Pair^1\leqc\Pair^2 \iff \Th^1\leqc\Th^2$ and $M^1\leqc M^2$.
\eit

\begin{figure}[ht]
\begin{center}

\begin{tikzpicture}[scale=1]
\node (Delta) at (2,-0.5) {$\Delta$};
\node (S2) at (-0.5,3) {$\S_2$};
\node (A3) at (1,1.5) {$\A_3$};
\node (S3) at (1,3) {$\S_3$};
\node (K4) at (2,0.5) {$\mathcal K_4$};
\node (A4) at (2,1.5) {$\A_4$};
\node (S4) at (2,3) {$\S_4$};
\node (A5) at (3,1.5) {$\A_5$};
\node (S5) at (3,3) {$\S_5$};
\node at (4,2.25) {$\dots$};
\node (An) at (5,1.5) {$\A_n$};
\node (Sn) at (5,3) {$\S_n$};
\node (R) at (2,4.5) {$R$};

\node (mudown) at (-1.5,1) {$\mudown$};
\node (muup) at (-3,1) {$\muup$};
\node (mu) at (-2.25,2) {$\mu$};
\node (lambda) at (-1.5,3) {$\rho$};
\node (rho) at (-3,3) {$\lam$};

\draw
(Delta)--(S2)
(Delta)--(A3)--(S3)
(Delta)--(K4)--(A4)--(S4)
(Delta)--(A5)--(S5)
(Delta)--(An)--(Sn)
(Delta)--(mudown)--(mu)--(lambda)
(Delta)--(muup)--(mu)--(rho)
(rho)--(R)
(lambda)--(R)
(S2)--(R)
(S3)--(R)
(S4)--(R)
(S5)--(R)
(Sn)--(R)
;
\end{tikzpicture}

\caption{The partial ordering $\leqc$ on  the C-matrix entries.}
\label{fig:M}

\end{center}
\end{figure}
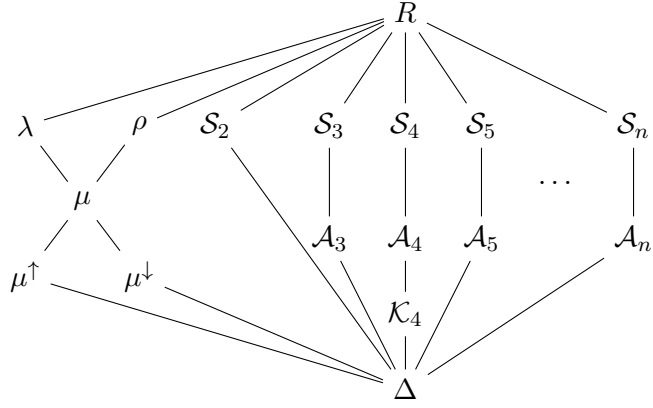

The deviations from the lexicographic ordering are caused by pairs of matching $\mu$s in rows~$0$ and~$1$, as well as by the exceptional congruences.
To capture the former,
suppose $M$ is a C-matrix of type \ref{RT2}, \ref{RT5} or~\ref{RT7}.
These are precisely the types that have `initial $\mu$s' in row $0$, by which we mean entries $M_{0j}=\mu$ with $j<\min\th_0$.  These initial $\mu$s are coloured green in 
Table~\ref{tab:RT}.
We define $\mumin_{0}(M)$ to be the position of the first initial $\mu$ in row $0$.  We then define $\mumin_{1}(M)$ to be the position of its `matching $\mu$' in row $1$.  Thus, in the notation of Table \ref{tab:RT}:
\[
\mumin_{0}(M) = \begin{cases}
i &\text{for \ref{RT2} and \ref{RT5}}\\
m-1 &\text{for \ref{RT7},}
\end{cases}
\AND
\mumin_{1}(M) = \begin{cases}
i+1 &\text{for \ref{RT2} and \ref{RT5}}\\
l-1 &\text{for \ref{RT7}.}
\end{cases}
\]
Note that $\mumin_{1}(M)$ need not be the position of the first $\mu$ in row $1$, as we could have $\zeta=\mu$ in types \ref{RT2} and \ref{RT5}.  

Also, when dealing with exceptional congruences, for an exceptional C-pair $\Pair=(\Th,M)$, recall that~$\hgt(M)$ is the index of the exceptional row (see Definition~\ref{def:exc}).
With this notation we have:

\begin{thm}[{\cite[Theorem 6.5]{ERtwisted1}}]
\label{thm:comparisons}
Let $n\geq 1$, and let $\Pair^1=(\Th^1,M^1)$ and  $\Pair^2=(\Th^2,M^2)$ be two C-pairs for~$\Ptw{n}$.
\begin{thmenumerate}
\item
\label{it:comp1}
We have $\cg(\Pair^1)\subseteq\cg(\Pair^2)$ if and only if both of the following hold:
\begin{thmsubenumerate}
\item
\label{it:comp1a}
$\Pair^1\leqc\Pair^2$, and
\item
\label{it:comp1b}
if $M^1$ has type \ref{RT2}, \ref{RT5} or \ref{RT7}, then at least one of the following hold:
\begin{enumerate}[label=\textup{(b\arabic*)},leftmargin=10mm]
\item \label{1b1} $\min\th_0^2\leq\mumin_{0}(M^1)$ and $\min\th_1^2\leq\mumin_{1}(M^1)$, or
\item \label{1b2} $M^2$ also has type \ref{RT2}, \ref{RT5} or \ref{RT7} (not necessarily the same as $M^1$), and $\mumin_{1}(M^2)-\mumin_{0}(M^2)=\mumin_{1}(M^1)-\mumin_{0}(M^1)$.
\end{enumerate}
\end{thmsubenumerate}
\item
\label{it:comp2}
When $\Pair^2$ is exceptional, we have $\cg(\Pair^1)\subseteq\cgx(\Pair^2)$ if and only if $\cg(\Pair^1)\subseteq \cg(\Pair^2)$.
\item
\label{it:comp3}
When $\Pair^1$ is exceptional, we have $\cgx(\Pair^1)\subseteq\cg(\Pair^2)$ if and only if all of the following hold, where $q:=\hgt(M^1)$:
\begin{thmsubenumerate}
\item
\label{it:comp3a}
$\cg(\Pair^1)\subseteq\cg(\Pair^2)$,
\item
\label{it:comp3b}
$2\per\th_q^2\mid\per\th_q^1$, and
\item
\label{it:comp3c}
$M_{qi}^2\in\{\S_q,R\}$ for all $i\geq\min\th_q^2$.
\end{thmsubenumerate}
\item
\label{it:comp4}
When both $\Pair^1$ and $\Pair^2$ are exceptional, we have $\cgx(\Pair^1)\subseteq\cgx(\Pair^2)$ if and only if  both of the following hold:
\begin{thmsubenumerate}
\item
\label{it:comp4a} $\cg(\Pair^1)\subseteq\cg(\Pair^2)$, and
\item
\label{it:comp4b} if $\hgt(M^1)=\hgt(M^2)=:q$, then the ratio $\per\th_q^1/\per\th_q^2$ is an odd integer.  \epfres
\end{thmsubenumerate}
\end{thmenumerate}
\end{thm}

\begin{rem}\label{rem:comp}
Let $\si^1$ and $\si^2$ be congruences on $\Ptw n$, with associated C-pairs $\Pair^1$ and $\Pair^2$.  Then by inspecting the various parts of Theorem \ref{thm:comparisons}, we have $\si^1\sub\si^2 \implies \Pair^1\leqc\Pair^2$.  \end{rem}

\begin{rem}\label{rem:01}
The cases of $n=0,1$ were discussed in \cite[Section 8]{ERtwisted1}.  
When $n=0$, $\Ptw0\cong\N$ and its congruence lattice is completely described by \eqref{eq:th1th2}.  
When $n=1$, Theorems \ref{thm:main} and \ref{thm:comparisons} remain valid, though there are some redundancies/degeneracies:
\bit
\item There are no rows $q\geq2$, and so no $N$-symbols, and no exceptional congruences.
\item Since $\wh\al=\wh\be$ for all $\al,\be\in\P_1$, it follows that certain symbols play the same role:  ${\muup\equiv\mudown\equiv\De}$ and $\lam\equiv\rho\equiv R$, and there are no unmatched~$\mu$s.
\eit
Thus, C-matrices have labels from $\{\De,\mu,R\}$, and only items \ref{C1}, \ref{C2} and \ref{C8} from Definition~\ref{de:cg} are needed to specify the congruence $\cg(\Th,M)$.  
Based on these observations, from now on the case $n=0$ will be ignored, while 
$n=1$ will be accompanied with necessary extra considerations.
\end{rem}

\begin{rem}\label{rem:leqc}
By inspection of the row types in Table \ref{tab:RT}, we see that the rows of a C-matrix $M=(M_{qi})_{\bnz\times\N}$ are weakly increasing with respect to the $\leqc$ order.  That is, we have
\[
M_{q0}\leqc M_{q1}\leqc\cdots \qquad\text{for all $q\in\bnz$.}
\]
\end{rem}

\section{Properties of the congruence lattice \boldmath{$\textsf{\textbf{Cong}}({\Ptw{n}})$}}
\label{sec:props}

At the end of Section 3 of \cite{ERtwisted1}, some properties of the lattice $\Cong(\Ptw{n})$ were derived as relatively easy corollaries of the main results, including the fact that $\Cong(\Ptw{n})$ is countable.  
In this section we provide a much more detailed analysis of the combinatorial and algebraic properties of the lattice, proving results on (co)atoms (Theorem \ref{thm:atoms}), infinite (anti)chains (Theorem \ref{thm:chains}), the covering relation (Theorem \ref{thm:covers}), and non-modularity (Theorem~\ref{thm:pent}).

In what follows we adopt the way of writing down explicit C-pairs introduced in \cite{ERtwisted1}: the C-matrix is written as a rectangular grid of its entries, and the members of the C-sequence are written to the right of their corresponding rows.

\begin{thm}\label{thm:atoms}
For $n\geq1$, the lattice $\Cong(\Ptw{n})$ has a unique co-atom, and no atoms.
\end{thm}

\begin{proof}
It follows immediately from Theorem \ref{thm:comparisons} that $\cg(\Pair)$ is the unique co-atom, where for $n\geq2$ and $n=1$, respectively:
\begin{equation}\label{eq:coatom}
\Pair=
\Cmatsetup 
\begin{array}{|c|c|c|c|c|cc}
\hhline{|-|-|-|-|-|~~}
\renewcommand{\arraystretch}{2}
\cellcolor{Ncol}\S_n & \cellcolor{Rcol}R & \cellcolor{Rcol}R & \cellcolor{Rcol}R & \cellcolor{Rcol}\cdots&\hspace{2mm}&(1,2)^\sharp \\ \hhline{|-|-|-|-|-|~~}
\cellcolor{Rcol}R & \cellcolor{Rcol}R & \cellcolor{Rcol}R & \cellcolor{Rcol}R & \cellcolor{Rcol}\cdots&&\nabla_\N \\ \hhline{|-|-|-|-|-|~~}
\cellcolor{Rcol}\vvdots & \cellcolor{Rcol}\vvdots & \cellcolor{Rcol}\vvdots & \cellcolor{Rcol}\vvdots & \cellcolor{Rcol}\vvdots&&\vvdots \\ \hhline{|-|-|-|-|-|~~}
\cellcolor{Rcol}R & \cellcolor{Rcol}R & \cellcolor{Rcol}R & \cellcolor{Rcol}R & \cellcolor{Rcol}\cdots&&\nabla_\N \\ \hhline{|-|-|-|-|-|~~}
\end{array}
\qquad\text{or}\qquad
\Pair=
\Cmatsetup 
\begin{array}{|c|c|c|c|c|cc}
\hhline{|-|-|-|-|-|~~}
\renewcommand{\arraystretch}{2}
\cellcolor{delcol}\De & \cellcolor{Rcol}R & \cellcolor{Rcol}R & \cellcolor{Rcol}R & \cellcolor{Rcol}\cdots&\hspace{2mm}&(1,2)^\sharp \\ \hhline{|-|-|-|-|-|~~}
\cellcolor{Rcol}R & \cellcolor{Rcol}R & \cellcolor{Rcol}R & \cellcolor{Rcol}R & \cellcolor{Rcol}\cdots&&\nabla_\N \\ \hhline{|-|-|-|-|-|~~}
\end{array}
\ .
\end{equation}

To prove that there are no atoms, we show that for every non-trivial congruence $\si$ there exists a non-trivial congruence $\tau\subsetneq\si$.  If $\si=\cgx(\Th,M)$ is exceptional we can take $\tau=\cg(\Th,M)$. Now suppose $\si=\cg(\Th,M)$.

Consider first the case where not all $\th_q$ equal $\Delta_\N$, and let $q$ be the largest such.  If $M$ has a non-$\De$ entry, then we take $\tau= \cg(\Th,M')$, where $M'$ consists entirely of $\De$s.  Otherwise, we take $\tau=\cg(\Th',M)$, where $\Th'$ is obtained from $\Th$ by replacing $\th_q=(m,m+d)^\sharp$ with $(m,m+2d)^\sharp$.  

If all $\th_q$ are equal to $\Delta_\N$, then, since $\si$ is non-trivial, $M$ must be of type \ref{RT2}; we then take $\tau=\cg(\Th,M')$, where $M'$ is obtained from $M$ by increasing the parameter $i=\mumin_{0}(M)$, and (if~$n\geq2$) replacing any $\S_2$ in row $2$ by $\De$.
\end{proof}

\begin{thm}
\label{thm:chains}
Let $n\geq 1$.
\ben
\item \label{chains1} $\Cong(\Ptw{n})$ contains no infinite ascending chains.
\item \label{chains2} $\Cong(\Ptw{n})$ contains infinite descending chains.
\item \label{chains3} $\Cong(\Ptw{n})$ contains infinite antichains.
\een
\end{thm}

\pf
\ref{chains1}  
We need to show that a weakly increasing sequence $\si^1\sub\si^2\sub\cdots$ of congruences must be eventually constant.  Writing $\Pair^i=(\Th^i,M^i)$ for the C-pair corresponding to each $\si^i$, it follows from Theorem \ref{thm:comparisons} (see Remark \ref{rem:comp}) that $\Pair^1\leqc\Pair^2\leqc\cdots$.  Since a C-pair defines at most two congruences, it is sufficient to show that this sequence of C-pairs is eventually constant.  
It is well known, and follows from \eqref{eq:th1th2}, that there are no infinite ascending chains of congruences on~$\N$.  It follows that the sequence $\Th^1\leqc\Th^2\leqc\cdots$ is eventually constant.  It remains to show that the sequence $M^1 \leqc M^2 \leqc \cdots$ is eventually constant as well.
Recall that in a C-matrix~$M$ every row is eventually constant; say row $q$ ends with an infinite sequence of the symbol $L_q$.  Gather these symbols into a sequence $\lim M:=(L_0,\dots,L_n)$.  Note that there are only finitely many limit sequences.
Let $k\in\N$ be such that the sequence $\lim M^k,\lim M^{k+1},\dots$ is constant.
Note that any $M^t$ with $t>k$ can only differ from $M^k$ in finitely many places, corresponding to the entries where $M_{qi}^k\neq L_q$.
Since the set of possible C-matrix entries is finite, it follows that our sequence of C-matrices is eventually constant, as required.

\ref{chains2}  This follows immediately from the absence of atoms (Theorem \ref{thm:atoms}).  It also follows from the obvious fact that $\N$ has infinite descending chains of congruences, or that $\Ptw n$ has infinite descending chains of ideals and hence Rees congruences.

\ref{chains3}  For $l\in\{2,3,\dots\}$, let $\si^l := \cg(\Pair^l)$, where
\[
\Pair^l=(\Theta^l,M^l) :=
\Cmatsetup
\begin{array}{|c|c|c|c|c|c|c|cc}\hhline{|-|-|-|-|-|-|-|~~}
\renewcommand{\arraystretch}{2}
\cellcolor{delcol}\Delta & \cellcolor{delcol}\Delta & \cellcolor{delcol}\cdots & \cellcolor{delcol}\Delta & \cellcolor{delcol}\Delta & \cellcolor{delcol}\Delta& \cellcolor{delcol}\cdots&\hspace{2mm}&\Delta_\N
\\ \hhline{|-|-|-|-|-|-|-|~~}
\cellcolor{delcol}\vvdots & \cellcolor{delcol}\vvdots & \cellcolor{delcol}\vvdots & \cellcolor{delcol}\vvdots & \cellcolor{delcol}\vvdots & \cellcolor{delcol}\vvdots& \cellcolor{delcol}\vvdots&&\vvdots
\\ \hhline{|-|-|-|-|-|-|-|~~}
\cellcolor{delcol}\Delta & \cellcolor{delcol}\Delta & \cellcolor{delcol}\cdots & \cellcolor{delcol}\Delta & \cellcolor{delcol}\Delta & \cellcolor{delcol}\Delta& \cellcolor{delcol}\cdots&&\Delta_\N
\\ \hhline{|-|-|-|-|-|-|-|~~}
\cellcolor{delcol}\Delta & \cellcolor{delcol}\Delta & \cellcolor{delcol}\cdots & \cellcolor{delcol}\Delta & \cellcolor{mucol} \mu & \cellcolor{Rcol}\mu& \cellcolor{Rcol}\cdots&&(l,l+1)^\sharp
\\ \hhline{|-|-|-|-|-|-|-|~~}
\cellcolor{mucol}\mu&\cellcolor{Rcol} \mu &\cellcolor{Rcol}\cdots&\cellcolor{Rcol}\mu&\cellcolor{Rcol}\mu&\cellcolor{Rcol}\mu& \cellcolor{Rcol}\cdots&&(1,2)^\sharp
\\ \hhline{|-|-|-|-|-|-|-|~~}
\multicolumn{1}{c}{}&\multicolumn{1}{c}{}&\multicolumn{1}{c}{}&\multicolumn{1}{c}{}&\multicolumn{1}{c}{}&\multicolumn{1}{c}{\text{\scriptsize $l$}}
\end{array}
\ .
\]
Then each $M^l$ has type \ref{RT7}.  For distinct $k,l\geq2$, we have
\[
\min\th_0^l =1>0= \mumin_{0}(M^k) \quad\text{and}\quad \mumin_{1}(M^l)-\mumin_{0}(M^l)=l-1 \neq k-1= \mumin_{1}(M^k)-\mumin_{0}(M^k),
\]
so that condition \ref{it:comp1b} of Theorem \ref{thm:comparisons}\ref{it:comp1} fails.  It follows that $\si^k$ and $\si^l$ are incomparable.
Therefore the set $\{\si^2,\si^3,\dots\}$ is an antichain in $\Cong(\Ptw{n})$.
\epf

\begin{thm}
\label{thm:covers}
For $n\geq1$, every element of $\Cong(\Ptw{n})$ is covered by only finitely many elements, but there are elements of $\Cong(\Ptw{n})$ that cover infinitely many elements.
\end{thm}

\begin{proof}
Beginning with the second assertion, let $M:=(\De)_{\bnz\times\N}$, let $\Theta:=(\nabla_\N,\Delta_\N,\dots,\Delta_\N)$, and for each prime $p$ let  $\Theta^p:=((0,p)^\sharp,\Delta_\N,\dots,\Delta_\N)$.  Then $\cg(\Theta,M)$ covers all $\cg(\Theta^p,M)$.

For the first assertion, let $\sigma^1$ be an arbitrary congruence of $\Ptw{n}$.  The trivial congruence has no covers since $\Cong(\Ptw{n})$ has no atoms (Theorem \ref{thm:atoms}), so for the rest of the proof we assume that $\sigma^1$ is non-trivial.
Suppose $\sigma^2\in\Cong(\Ptw{n})$ covers $\sigma^1$. For $t=1,2$, let $\Pair^t=(\Theta^t,M^t)$ be the C-pair associated with $\sigma^t$, noting that $\Pair^1\leqc\Pair^2$ by Theorem \ref{thm:comparisons}.
We will prove that, given $\sigma^1$, there are only finitely many possible choices for~$\sigma^2$ by showing that there are only finitely many choices for $\Pair^2$.

For $t=1,2$, let $r_t$ be the largest index of a non-trivial row in $\Pair^t$, by which we mean that $\theta^t_{r_t}\neq\Delta_\N$ or $(M^t_{r_t,0},M^t_{r_t,1},\dots)\neq (\Delta,\Delta,\dots)$.  Note that $0\leq r_1\leq r_2$.  
We first claim that $r_1=r_2$.  To prove this, suppose to the contrary that $r_1<r_2$.  
\bit
\item If $\theta_{r_2}^2\neq\Delta_\N$, then we alter  $\Pair^2$ to obtain a new pair $\Pair^3=(\Theta^3,M^3)$ as follows.  We replace $\theta_{r_2}^2=(m,m+d)^\sharp$ by $(m,m+2d)^\sharp$, and we change all symbols in row $r_2$ of $M^2$ to $\Delta$.  If $r_2=1$, then we also change any $\mu$ entries in row $0$ to $\Delta$.  Then $\sigma^1\subsetneq \cg(\Pair^3)\subsetneq \sigma^2$, contradicting the assumption that $\sigma^2$ covers $\sigma^1$.
\item If $r_2\geq2$ and $\theta_{r_2}^2=\Delta_\N$, then row $r_2$ of $M^2$ must contain symbols distinct from $\Delta$, and hence is of type \ref{RT9}. Replace the leftmost $N$-symbol in this row by $\Delta$ to obtain a new C-pair $\Pair^3$, which then satisfies $\sigma^1\subsetneq \cg(\Pair^3)\subsetneq \sigma^2$, a contradiction again.
\item If $r_2=1$ and $\theta_{1}^2=\Delta_\N$, then $M^2$ must be of type \ref{RT2}.  Hence $\theta_0^2=\Delta_\N$, and then from $\sigma^1\subsetneq\sigma^2$ we conclude that $\theta_0^1=\Delta_\N$ and that row $0$ of $M^1$ consists entirely of $\Delta$s, contradicting the assumption that $\sigma^1$ is non-trivial.
\eit
With the claim proved, we now write $r:=r_1=r_2$.  

\setcounter{caseco}{0}

\case  \label{ca:th_r_NT}
$\th_r^1\not=\De_\N$.  Since $\th_r^1\sub\th_r^2\sub\th_{r-1}^2\sub\dots\sub\th_0^2$, there are only finitely many choices for the congruences $\th_0^r,\dots,\th_r^2$, all of which are non-trivial.  Given such a choice of these congruences, let $0\leq i\leq r$ and write $m:=\min\th_i^2$.  There are only finitely many choices for $M_{i0}^2,\dots,M_{i,m-1}^2$ and for $M_{im}^2=M_{i,m+1}^2=\dots$, i.e.~there are only finitely many choices for row~$i$ of $M^2$.  Consequently, there are only finitely many possibilities for $M^2$.

\case
$r\geq2$ and $\th_r^1=\De_\N$.  In this case, row $r$ of $M^1$ is of type \ref{RT9}, say
\[
\Cmatsetup
 \begin{array}{|c|c|c|c|c|c|c|c|c|c|}
\hline
\cellcolor{delcol}\Delta & \cellcolor{delcol}\cdots &\cellcolor{delcol}\Delta &\cellcolor{Ncol}N_i &\cellcolor{Ncol}N_{i+1} &\cellcolor{Ncol}\cdots &\cellcolor{Ncol}N_{k-1} &\cellcolor{Ncol}N &\cellcolor{Ncol}N &\cellcolor{Ncol}\cdots  \\ \hline
\end{array}\ .
\]
If $\th_r^2\not=\De_\N$, say $\th_r^2=(m,m+d)^\sharp$, then we modify $\Pi^1$ by replacing $\theta_r^1=\De_\N$ by the congruence $\theta_r^3$ on~$\N$ with $\min\theta_r^3 =\max(k,m)$ and $\per\th_r^3=2d$, 
to produce a C-pair $\Pi^3$ with $\sigma^1\subsetneq \cg(\Pair^3)\subsetneq \sigma^2$, and hence a contradiction.  

So now suppose $\th_r^2=\De_\N$.  Then row $r$ of $M^2$ is also of type \ref{RT9}.  If this differed from row $r$ of $M^1$ in more than one place, then we could replace the first such entry in $M^2$ by the corresponding entry in $M^1$, to obtain a contradiction in the usual way.  So in fact there is at most one difference between the two $r$th rows, and hence only finitely many choices for row $r$ of~$M^2$.  Now fix some such choice, and note that $M_{ri}^2,M_{r,i+1}^2,\dots$ are all $N$-symbols.

\setcounter{subcaseco}{0}

\subcase $r\geq3$.
In this case the verticality conditions tell us that $M_{qj}^2=R$ for all $0\leq q<r$ and $j\geq i$.  There are therefore only finitely many ways to complete the matrix $M^2$, and the congruences $\th_0^2,\dots,\th_{r-1}^2$ are then completely determined by the $R$s in $M^2$.

\subcase $r=2$.  
If $\th_1^1\not=\De_\N$, then the argument from Case \ref{ca:th_r_NT} remains valid and shows there are only finitely many possibilities for rows $0$ and $1$ of $\Pair^2$, including the associated congruences on $\N$.  So suppose now that $\th_1^1=\De_\N$, which means that $M^1$ is of type \ref{RT2}.  If also $\th_1^2=\De_\N$, then $M^2$ is also of type \ref{RT2}, and since $\si^1\sub\si^2$ we must have $\mumin_0(M^2)\leq\mumin_0(M^1)$; there are therefore only finitely many choices 
for $M^2$.  So now suppose $\th_1^2\not=\De_\N$, say $\th_1^2=(m,m+d)^\sharp$, and note that $\min\th_0^2\leq m$.  Then with $l:=\max(m,\mumin_0(M^1))$, we modify $\Pair^1$ to $\Pair^3$ by replacing $\th_0^1=\th_1^1=\De_\N$ by $\th_0^3=(l,l+d)^\sharp$ and $\th_1^3=(l+1,l+1+d)^\sharp$, and obtain $\sigma^1\subsetneq \cg(\Pair^3)\subsetneq \sigma^2$, a contradiction again.

\case
$r=1$ and $\th_1^1=\De_\N$.  
The C-matrix $M^1$ is of type \ref{RT2}, since its row 1 cannot consist entirely of $\Delta$s.
It then follows that the type of $M^2$ is one of \ref{RT2} or \ref{RT4}--\ref{RT7}.
If $M^2$ is of type~\ref{RT2} then $\mumin_0(M^2)\leq\mumin_0(M^1)$, and hence there are only finitely many choices for $\Pair^2$.
In any of the remaining row types, we can modify $\Pi^2$ by doubling $\per\theta_0^2=\per\theta_1^2$, and changing all the $\xi$-entries to $\mu$, to obtain a C-pair $\Pi^3$ with $\sigma^1\subsetneq \cg(\Pair^3)\subsetneq \sigma^2$,
and a contradiction.

\case 
$r=0$ and $\th_0^1=\De_\N$.  
Since row $1$ of $\Pair^1$ is trivial, and since $\th_0^1=\De_\N$, $M^1$ can only have type \ref{RT1}.  But then $\si^1=\De_{\Ptw n}$, a contradiction.

This exhausts all the cases, and completes the proof of the theorem.
\end{proof}

\begin{rem}\label{rem:no_covers}
We observed in the proof that the trivial congruence has no covers; so too, obviously, does the universal congruence.  However, these are not the only such congruences.  For example, it is easy to see that the Rees congruence $R_I$ corresponding to an ideal of the form~$I = I_{q0}$ for $q\in\bnz$ has no covers.  Indeed, the C-pair associated to $R_I$ is
\[
\Cmatsetup
\begin{array}{r|c|c|c|c|cc}\hhline{~|-|-|-|-|~~}
\renewcommand{\arraystretch}{2}
\text{\scriptsize $n$} & \cellcolor{delcol}\Delta & \cellcolor{delcol}\Delta & \cellcolor{delcol}\Delta & \cellcolor{delcol}\cdots && \De_\N \\ \hhline{~|-|-|-|-|~~}
 & \cellcolor{delcol}\vvdots & \cellcolor{delcol}\vvdots & \cellcolor{delcol}\vvdots & \cellcolor{delcol}\vvdots && \vdots \\ \hhline{~|-|-|-|-|~~}
 & \cellcolor{delcol}\Delta & \cellcolor{delcol}\Delta & \cellcolor{delcol}\Delta & \cellcolor{delcol}\cdots && \De_\N \\ \hhline{~|-|-|-|-|~~}
\text{\scriptsize $q$} & \cellcolor{Rcol}R & \cellcolor{Rcol}R & \cellcolor{Rcol}R & \cellcolor{Rcol}\cdots && \nab_\N \\ \hhline{~|-|-|-|-|~~}
 & \cellcolor{Rcol}\vdots & \cellcolor{Rcol}\vdots & \cellcolor{Rcol}\vdots & \cellcolor{Rcol}\vdots && \vdots \\ \hhline{~|-|-|-|-|~~}
\text{\scriptsize $0$} & \cellcolor{Rcol}R & \cellcolor{Rcol}R & \cellcolor{Rcol}R & \cellcolor{Rcol}\cdots && \nab_\N \\ \hhline{~|-|-|-|-|~~}
\end{array}.
\]
One then proceeds, as in the proof of Theorem \ref{thm:atoms}, to show that for any congruence ${\si\supsetneq R_I}$, we have $\si\supsetneq\tau\supsetneq R_I$ for some $\tau\in\Cong(\Ptw n)$.  This is done by analysing the C-pair $(\Th,M)$ corresponding to $\si$, and considering separate cases according to whether any of the congruences $\th_{q+1},\ldots,\th_n$ is non-trivial.
\end{rem}

\begin{thm}
\label{thm:pent}
For $n\geq1$, the lattice
$\Cong(\Ptw{n})$ contains five-element diamond and pentagon sublattices, and hence is not modular (or distributive).
\end{thm}

\begin{proof}
For $n\geq2$, define four C-chains:
\begin{align*}
\Th^1 &:= (\nab_\N,\nab_\N,(0,2)^\sharp,\De_\N,\dots,\De_\N) , &\Th^3 &:= (\nab_\N,\nab_\N,(1,2)^\sharp,\De_\N,\dots,\De_\N) , \\
\Th^2 &:= (\nab_\N,\nab_\N,\nab_\N,\De_\N,\dots,\De_\N), &\Th^4 &:= (\nab_\N,\nab_\N,(1,3)^\sharp,\De_\N,\dots,\De_\N) ,
\end{align*}
and three C-matrices:
\[
M^1 := 
\Cmatsetup
\begin{array}{|c|c|c|c|}\hline
\renewcommand{\arraystretch}{2}
\cellcolor{delcol}\Delta & \cellcolor{delcol}\Delta & \cellcolor{delcol}\Delta & \cellcolor{delcol}\cdots \\ \hline
\cellcolor{delcol}\vvdots & \cellcolor{delcol}\vvdots & \cellcolor{delcol}\vvdots & \cellcolor{delcol}\vvdots \\ \hline
\cellcolor{delcol}\Delta & \cellcolor{delcol}\Delta & \cellcolor{delcol}\Delta & \cellcolor{delcol}\cdots \\ \hline
\cellcolor{delcol}\Delta & \cellcolor{delcol}\Delta & \cellcolor{delcol}\Delta & \cellcolor{delcol}\cdots \\ \hline
\cellcolor{Rcol}R & \cellcolor{Rcol}R & \cellcolor{Rcol}R & \cellcolor{Rcol}\cdots \\ \hline
\cellcolor{Rcol}R & \cellcolor{Rcol}R & \cellcolor{Rcol}R & \cellcolor{Rcol}\cdots \\ \hline
\end{array}
\normalsize
\ \COMMA
M^2 := 
\Cmatsetup
\begin{array}{|c|c|c|c|}\hline
\renewcommand{\arraystretch}{2}
\cellcolor{delcol}\Delta & \cellcolor{delcol}\Delta & \cellcolor{delcol}\Delta & \cellcolor{delcol}\cdots \\ \hline
\cellcolor{delcol}\vvdots & \cellcolor{delcol}\vvdots & \cellcolor{delcol}\vvdots & \cellcolor{delcol}\vvdots \\ \hline
\cellcolor{delcol}\Delta & \cellcolor{delcol}\Delta & \cellcolor{delcol}\Delta & \cellcolor{delcol}\cdots \\ \hline
\cellcolor{Ncol}\S_2 & \cellcolor{Ncol}\S_2 & \cellcolor{Ncol}\S_2 & \cellcolor{Ncol}\cdots \\ \hline
\cellcolor{Rcol}R & \cellcolor{Rcol}R & \cellcolor{Rcol}R & \cellcolor{Rcol}\cdots \\ \hline
\cellcolor{Rcol}R & \cellcolor{Rcol}R & \cellcolor{Rcol}R & \cellcolor{Rcol}\cdots \\ \hline
\end{array}
\ \COMMA
M^3 := 
\Cmatsetup
\begin{array}{|c|c|c|c|}\hline
\renewcommand{\arraystretch}{2}
\cellcolor{delcol}\Delta & \cellcolor{delcol}\Delta & \cellcolor{delcol}\Delta & \cellcolor{delcol}\cdots \\ \hline
\cellcolor{delcol}\vvdots & \cellcolor{delcol}\vvdots & \cellcolor{delcol}\vvdots & \cellcolor{delcol}\vvdots \\ \hline
\cellcolor{delcol}\Delta & \cellcolor{delcol}\Delta & \cellcolor{delcol}\Delta & \cellcolor{delcol}\cdots \\ \hline
\cellcolor{delcol}\Delta & \cellcolor{Ncol}\S_2 & \cellcolor{Ncol}\S_2 & \cellcolor{Ncol}\cdots \\ \hline
\cellcolor{Rcol}R & \cellcolor{Rcol}R & \cellcolor{Rcol}R & \cellcolor{Rcol}\cdots \\ \hline
\cellcolor{Rcol}R & \cellcolor{Rcol}R & \cellcolor{Rcol}R & \cellcolor{Rcol}\cdots \\ \hline
\end{array}
\ .
\]
Then the following are all valid C-pairs:
\[
(\Th^1,M^1) \COMMa\ \ 
(\Th^1,M^2) \COMMa\ \ 
(\Th^2,M^1) \COMMa\ \ 
(\Th^2,M^2) \COMMa\ \ 
(\Th^3,M^2) \COMMa\ \ 
(\Th^3,M^3) \COMMa\ \ 
(\Th^4,M^1) ,
\]
and furthermore the pairs $(\Th^1,M^1)$ and $(\Th^4,M^1)$ are exceptional.  By Theorem \ref{thm:comparisons}, the resulting congruences are ordered as shown in Figure \ref{fig:notmod}.  Moreover, it is easy to check that their meets and joins are as indicated in the figure.  

For $n=1$, we cannot use exceptional congruences to construct pentagon and diamond sublattices.
Nonetheless, examples of such sublattices are shown in Figure \ref{fig:pediPtw1}, with the congruences indicated by their associated C-pairs.
As a curiosity we mention that for $n\geq2$ the lattice $\Cong(\Ptw{n})$ contains the analogue of this pentagon (with the remaining rows consisting entirely of $\Delta$s, and their corresponding congruences on $\N$ trivial), but not of the
diamond, as the latter relies on the fact that for $n=1$ both $\mu$ and $\Delta$ in row $1$ correspond to the trivial restriction.
\end{proof}

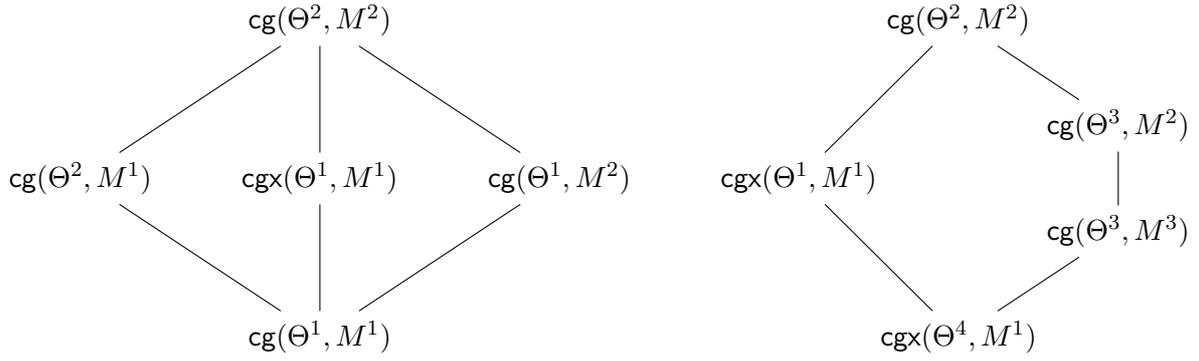
\begin{figure}[ht]
\begin{center}
\begin{tikzpicture}[scale=.7]
\begin{scope}
\node (11) at (0,0) {$\cg(\Th^1,M^1)$};
\node (21) at (-4.5,3) {$\cg(\Th^2,M^1)$};
\node (11') at (0,3) {$\cgx(\Th^1,M^1)$};
\node (12) at (4.5,3) {$\cg(\Th^1,M^2)$};
\node (22) at (0,6) {$\cg(\Th^2,M^2)$};
\draw (11)--(21)--(22)--(11')--(11)--(12)--(22);
\end{scope}
\begin{scope}[shift={(12,0)}]
\node (41') at (0,0) {$\cgx(\Th^4,M^1)$};
\node (11') at (-3,3) {$\cgx(\Th^1,M^1)$};
\node (22) at (0,6) {$\cg(\Th^2,M^2)$};
\node (33) at (3,2) {$\cg(\Th^3,M^3)$};
\node (32) at (3,4) {$\cg(\Th^3,M^2)$};
\draw (41')--(33)--(32)--(22)--(11')--(41');
\end{scope}
\end{tikzpicture}
\caption{Diamond and pentagon sublattices in $\Cong(\Ptw n)$ for $n\geq2$.  See the proof of Theorem \ref{thm:pent} for more details.}
\label{fig:notmod}
\end{center}
\end{figure}

\begin{figure}[ht]
\begin{center}
\begin{tikzpicture}[xscale=0.6,yscale=0.8]
\node (B) at (0,0) {
$\Cmatsetup
\begin{array}{c|c|c|c|c|c}\hhline{~-|-|-|-~}
&\cellcolor{delcol}\De &\cellcolor{delcol}\De & \cellcolor{delcol}\De& \cellcolor{delcol}\cdots & (1,3)^\sharp\\ \hhline{~-|-|-|-~}
&\cellcolor{delcol}\De &\cellcolor{delcol}\De & \cellcolor{delcol}\De& \cellcolor{delcol}\cdots & (0,2)^\sharp\\ \hhline{~-|-|-|-~}
\end{array}$};
\node (L) at (-3,3) {$
\Cmatsetup
\begin{array}{c|c|c|c|c|c}\hhline{~-|-|-|-~}
&\cellcolor{delcol}\De &\cellcolor{mucol}\mu & \cellcolor{mucol}\mu& \cellcolor{mucol}\cdots & (1,3)^\sharp\\ \hhline{~-|-|-|-~}
&\cellcolor{mucol}\mu &\cellcolor{mucol}\mu & \cellcolor{mucol}\mu& \cellcolor{mucol}\cdots & (0,2)^\sharp\\ \hhline{~-|-|-|-~}
\end{array}
$};
\node (T) at (0,6) {$
\Cmatsetup
\begin{array}{c|c|c|c|c|c}\hhline{~-|-|-|-~}
&\cellcolor{delcol}\De &\cellcolor{Rcol}R & \cellcolor{Rcol}R& \cellcolor{Rcol}\cdots & (1,2)^\sharp\\ \hhline{~-|-|-|-~}
&\cellcolor{Rcol}R &\cellcolor{Rcol}R & \cellcolor{Rcol}R& \cellcolor{Rcol}\cdots & (0,1)^\sharp\\ \hhline{~-|-|-|-~}
\end{array}
$};
\node (RB) at (3.5,2) {$
\Cmatsetup
\begin{array}{c|c|c|c|c|c}\hhline{~-|-|-|-~}
&\cellcolor{delcol}\De &\cellcolor{delcol}\De & \cellcolor{delcol}\De& \cellcolor{delcol}\cdots & (1,3)^\sharp\\ \hhline{~-|-|-|-~}
&\cellcolor{Rcol}R &\cellcolor{Rcol}R & \cellcolor{Rcol}R& \cellcolor{Rcol}\cdots & (0,1)^\sharp\\ \hhline{~-|-|-|-~}
\end{array}
$};
\node (RT) at (3.5,4) {$
\Cmatsetup
\begin{array}{c|c|c|c|c|c}\hhline{~-|-|-|-~}
&\cellcolor{delcol}\De &\cellcolor{delcol}\De & \cellcolor{delcol}\De& \cellcolor{delcol}\cdots & (1,2)^\sharp\\ \hhline{~-|-|-|-~}
&\cellcolor{Rcol}R &\cellcolor{Rcol}R & \cellcolor{Rcol}R& \cellcolor{Rcol}\cdots & (0,1)^\sharp\\ \hhline{~-|-|-|-~}
\end{array}
$};
\draw (B)--(RB)--(RT)--(T)--(L)--(B);
\begin{scope}[shift={(-14,0)}]
\node (B) at (0,0) {
$\Cmatsetup
\begin{array}{c|c|c|c|c|c}\hhline{~-|-|-|-~}
&\cellcolor{delcol}\De &\cellcolor{delcol}\De & \cellcolor{Rcol}R& \cellcolor{Rcol}\cdots & (2,3)^\sharp\\ \hhline{~-|-|-|-~}
&\cellcolor{delcol}\De &\cellcolor{Rcol}R & \cellcolor{Rcol}R& \cellcolor{Rcol}\cdots & (1,2)^\sharp\\ \hhline{~-|-|-|-~}
\end{array}$};
\node (L) at (-4.8,3) {$
\Cmatsetup
\begin{array}{c|c|c|c|c|c}\hhline{~-|-|-|-~}
&\cellcolor{delcol}\De &\cellcolor{mucol}\mu & \cellcolor{Rcol}R& \cellcolor{Rcol}\cdots & (1,2)^\sharp\\ \hhline{~-|-|-|-~}
&\cellcolor{mucol}\mu &\cellcolor{Rcol}R & \cellcolor{Rcol}R& \cellcolor{Rcol}\cdots & (0,1)^\sharp\\ \hhline{~-|-|-|-~}
\end{array}
$};
\node (T) at (0,6) {$
\Cmatsetup
\begin{array}{c|c|c|c|c|c}\hhline{~-|-|-|-~}
&\cellcolor{delcol}\De &\cellcolor{Rcol}R & \cellcolor{Rcol}R& \cellcolor{Rcol}\cdots & (1,2)^\sharp\\ \hhline{~-|-|-|-~}
&\cellcolor{Rcol}R &\cellcolor{Rcol}R & \cellcolor{Rcol}R& \cellcolor{Rcol}\cdots & (0,1)^\sharp\\ \hhline{~-|-|-|-~}
\end{array}
$};
\node (M) at (0,3) {$
\Cmatsetup
\begin{array}{c|c|c|c|c|c}\hhline{~-|-|-|-~}
&\cellcolor{delcol}\De &\cellcolor{Rcol}R & \cellcolor{Rcol}R& \cellcolor{Rcol}\cdots & (1,2)^\sharp\\ \hhline{~-|-|-|-~}
&\cellcolor{delcol}\De &\cellcolor{Rcol}R & \cellcolor{Rcol}R& \cellcolor{Rcol}\cdots & (1,2)^\sharp\\ \hhline{~-|-|-|-~}
\end{array}
$};
\node (R) at (4.8,3) {$
\Cmatsetup
\begin{array}{c|c|c|c|c|c}\hhline{~-|-|-|-~}
&\cellcolor{delcol}\De &\cellcolor{delcol}\De & \cellcolor{Rcol}R& \cellcolor{Rcol}\cdots & (2,3)^\sharp\\ \hhline{~-|-|-|-~}
&\cellcolor{Rcol}R &\cellcolor{Rcol}R & \cellcolor{Rcol}R& \cellcolor{Rcol}\cdots & (0,1)^\sharp\\ \hhline{~-|-|-|-~}
\end{array}
$};
\draw (B)--(L)--(T)--(M)--(B)--(R)--(T);
\end{scope}
\end{tikzpicture}
\caption{Diamond and pentagon sublattices in $\Cong(\Ptw{1})$.}
\label{fig:pediPtw1}
\end{center}
\end{figure}
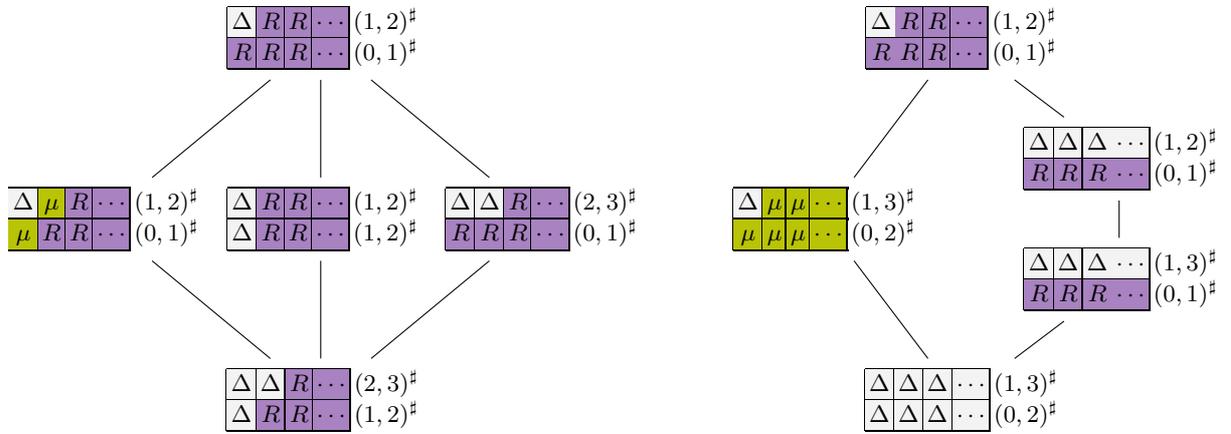

We used the pentagon sublattices displayed in Figures \ref{fig:notmod} and \ref{fig:pediPtw1} to show that $\Cong(\Ptw n)$ is not modular for $n\geq1$.  In fact, these show something more.  Consider a lattice $L$, and write $\prec$ for the covering relation in $L$.  Recall \cite{Stern1999} that $L$ is \emph{upper} or \emph{lower semimodular} if for all $a,b\in L$ we have
\begin{equation}\label{eq:sm}
a \wedge b \prec a \implies b \prec a\vee b \qquad\text{or}\qquad a\prec a\vee b \implies a\wedge b\prec b,
\end{equation}
respectively.  
%Consider a lattice that contains the following pentagon sublattice:
%\[
%\begin{tikzpicture}[xscale=0.3,yscale=0.5]
%\node (B) at (0,0) {$x$};
%\node (L) at (-3,3) {$a$};
%\node (T) at (0,6) {$y$};
%\node (RB) at (3.5,2) {$b_1$};
%\node (RT) at (3.5,4) {$b_2$};
%\draw (B)--(RB)--(RT)--(T)--(L)--(B);
%\end{tikzpicture},
%\]
%%so that $x=a\wedge b_1=a\wedge b_2$ and $y=a\vee b_1=a\vee b_2$.  
%If $x\prec a$, then it is easy to see that the elements $a,b_1\in L$ witness the failure of the first implication in \eqref{eq:sm}, so that $L$ is not upper semimodular in this case.  Dually, $L$ is not lower semimodular if $a\prec y$.  

\begin{thm}\label{thm:sm}
For $n\geq1$, the lattice $\Cong(\Ptw n)$ is neither upper nor lower semimodular.
\end{thm}

\pf
Consider the diamond sublattice of $\Cong(\Ptw n)$ displayed in Figure \ref{fig:notmod} (for $n\geq2$) or Figure~\ref{fig:pediPtw1} (for~$n=1$), and label the congruences in the following way:
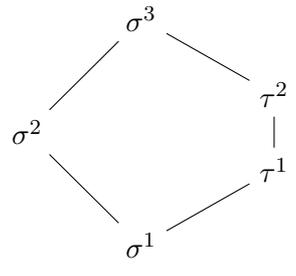
\[
\begin{tikzpicture}[scale=0.5]
\node (B) at (0,0) {$\si^1$};
\node (L) at (-3,3) {$\si^2$};
\node (T) at (0,6) {$\si^3$};
\node (RB) at (3.5,2) {$\tau^1$};
\node (RT) at (3.5,4) {$\tau^2$};
\draw (B)--(RB)--(RT)--(T)--(L)--(B);
\end{tikzpicture},
\]
It sufices to show that $\si^1\prec\si^2\prec\si^3$.  Indeed, then $(a,b)=(\si^2,\tau^1)$ or $(a,b)=(\si^2,\tau^2)$ will witness the failure of the first and second implications in \eqref{eq:sm}, respectively.  Checking these covering relations is fairly routine.  We just give the details for $\si^1\prec\si^2$ in the case $n\geq2$.  

So suppose $\si^1 \sub \si \sub \si^2$ for some $\si\in\Cong(\Ptw n)$, where $n\geq2$, and let $\Pair$ be the C-pair associated to $\si$.  As in Remark \ref{rem:comp}, $\si^1\sub\si\sub\si^2$ implies $(\Th^4,M^1) \leqc \Pair \leqc (\Th^1,M^1)$.  It then follows by the form of these pairs that $\Pair$ is in fact one of $(\Th^4,M^1)$ or $(\Th^1,M^1)$, and hence that~$\si$ is one of 
\[
\cg(\Th^4,M^1) \COMMA \cgx(\Th^4,M^1) \COMMA \cg(\Th^1,M^1) \COMMA \cgx(\Th^1,M^1).
\]
But $\si^1=\cgx(\Th^4,M^1)$ is contained in neither $\cg(\Th^4,M^1)$ nor $\cg(\Th^1,M^1)$, so it follows that $\si$ is one of $\cgx(\Th^4,M^1)=\si^1$ or~$\cgx(\Th^1,M^1)=\si^2$.  This completes the proof that $\si^1\prec\si^2$.
\epf

\section{Generators of congruences of \boldmath{$\Ptw{n}$}}\label{sec:fg}

It follows immediately from Theorem \ref{thm:chains}\ref{chains1} that every congruence of $\Ptw{n}$ is finitely generated.
We can in fact do better than that, and in Theorem \ref{thm:fg} below we show that the number of pairs needed to generate an arbitrary congruence is bounded above by $\lceil\frac{5n}2\rceil$.
As a stepping stone, Theorem~\ref{thm:pc} classifies the principal congruences.  The main technical result needed for all these calculations is the following:

\begin{lemma}\label{lem:fg}
Let $\Pair$ be a C-pair, let $\Om\sub\Ptw n\times\Ptw n$, and let $\Pair'$ be the C-pair corresponding to $\Om^\sharp$.  Then:
\ben
\item \label{it:fg1} $\cg(\Pair)=\Om^\sharp$ if and only if $\Om\sub\cg(\Pair)$ and $\Pair\leqc\Pair'$;
\item \label{it:fg2} if $\Pair$ is exceptional, then $\cgx(\Pair)=\Om^\sharp$ if and only if $\Om\sub\cgx(\Pair)$, $\Om\not\sub\cg(\Pair)$ and $\Pair\leqc\Pair'$.
\een
\end{lemma}

\pf
For the direct implication in each part the inequality $\Pair\leqc\Pair'$ follows from Remark \ref{rem:comp},
and the remaining (non-)inclusions are obvious.
So we we just need to
prove the converse.

\ref{it:fg1}  Since $\cg(\Pair)$ is a congruence, it follows from $\Om\sub\cg(\Pair)$ that $\Om^\sharp\sub\cg(\Pair)$.  It follows from this (and Remark \ref{rem:comp}) that $\Pair'\leqc\Pair$.  Together with $\Pair\leqc\Pair'$, it follows that $\Pair'=\Pair$.  Thus, $\Om^\sharp$ is either $\cg(\Pair)$ or else possibly $\cgx(\Pair)$ if $\Pair$ is exceptional; but the latter is ruled out by $\Om^\sharp\sub\cg(\Pair)$.

\ref{it:fg2}  The proof is almost identical to the previous part.  In the first step, $\cg(\Pair)$ is replaced by $\cgx(\Pair)$.  In the last step we use $\Om\not\sub\cg(\Pair)$ to rule out $\Om^\sharp=\cg(\Pair)$.
\epf

We now determine the principal congruences on~$\Ptw n$, for which it suffices by symmetry to determine $\pc\ba\bb$ where $\ba\in D_{qi}$ and $\bb\in D_{rj}$, $q\geq r$, and $i\leq j$ when $q=r$.  
The congruences will be described by means of their associated C-pairs in the usual manner;
additionally, we will omit rows ${q+1,\ldots,n}$, as these consist entirely of $\De$s and have $\De_\N$ as their associated congruence on~$\N$.
For a permutation $\pi\in\S_q$, we write $\norm\pi$ for the 
\emph{normal closure} of $\pi$ in $\S_q$, i.e.~the smallest normal subgroup of $\S_q$ containing $\pi$.

\begin{thm}\label{thm:pc}
Let $n\geq 1$, and let
$\ba=(i,\al)\in D_{qi}$, $\bb=(j,\be)\in D_{rj}$ be two elements of $\Ptw{n}$ with $q< r$, or $q=r$ and  $i\leq j$.
\ben
\item \label{pc1} 
If $\al=\be$ and $i=j$ (i.e.~$\ba=\bb$), then $\pc\ba\bb=\De_{\Ptw n}=\cg(\Th,M)$,
where $M=(\De)_{\bnz\times\N}$ and $\theta_s=\Delta_\N$ for all $s$.
\item \label{pc2} 
If $\al=\be$ and $i<j$, then $\pc\ba\bb=\cg(\Th,M)$, where $M=(\De)_{\bnz\times\N}$ and
$\th_s = (i,j)^\sharp $ for all $s=0,\dots,q$.
\item \label{pc3} 
If $q\geq2$ and $(\al,\be)\not\in\H$, then $\pc\ba\bb=R_{I_{qi}\cup I_{rj}}=\cg(\Th,M)$,
where $M_{sl}=R$ precisely when $D_{sl}\subseteq I_{qi}\cup I_{rj}$ and $M_{sl}=\Delta$ otherwise,
and where $\theta_s=(m_s,m_s+1)^\sharp$ ($0\leq s\leq q$) for $m_s$ the first index of an $R$-entry in row $s$ of $M$.
\item \label{pc4} If $q\geq3$, $i=j$, $(\al,\be)\in\H$ and $\al\neq \be$, then with $N:=\norm{\pd(\al,\be)}$ we have $\pc\ba\bb=\cg(\Th,M)$ for the C-pair
\[
\Cmatsetup
\begin{array}{r|c|c|c|c|c|c|cc}\hhline{~|-|-|-|-|-|-|~~}
\text{\scriptsize $q$} &\cellcolor{delcol}\De &\cellcolor{delcol}\cdots & \cellcolor{delcol}\De  & \cellcolor{Ncol} N & \cellcolor{Ncol}N & \cellcolor{Ncol} \cdots &\hspace{2mm}& \De_\N\\ \hhline{~|-|-|-|-|-|-|~~} 
&\cellcolor{delcol}\De &\cellcolor{delcol}\cdots & \cellcolor{delcol}\De  & \cellcolor{Rcol} R & \cellcolor{Rcol}R & \cellcolor{Rcol} \cdots && (i,i+1)^\sharp\\ \hhline{~|-|-|-|-|-|-|~~}
&\cellcolor{delcol}\vvdots &\cellcolor{delcol}\vvdots & \cellcolor{delcol}\vvdots  & \cellcolor{Rcol} \vvdots & \cellcolor{Rcol}\vvdots & \cellcolor{Rcol} \vvdots &&\vvdots  \\ \hhline{~|-|-|-|-|-|-|~~}
\text{\scriptsize $0$} &\cellcolor{delcol}\De &\cellcolor{delcol}\cdots & \cellcolor{delcol}\De  & \cellcolor{Rcol} R & \cellcolor{Rcol}R & \cellcolor{Rcol} \cdots && (i,i+1)^\sharp\\ \hhline{~|-|-|-|-|-|-|~~}
\multicolumn{1}{c}{}&\multicolumn{1}{c}{}&\multicolumn{1}{c}{}&\multicolumn{1}{c}{}&\multicolumn{1}{c}{\text{\scriptsize $i$}}&\multicolumn{1}{c}{}&\multicolumn{1}{c}{}&\multicolumn{1}{c}{}
\end{array}.
\]
\item \label{pc5} If $q\geq3$, $i<j$, $(\al,\be)\in\H$ and $\pd(\al,\be)\in\A_q\sm\{\id_q\}$, then with $N:=\norm{\pd(\al,\be)}$ we have $\pc\ba\bb=\cg(\Th,M)$ for the C-pair
\[
\Cmatsetup
\begin{array}{r|c|c|c|c|c|c|cc}\hhline{~|-|-|-|-|-|-|~~}
\text{\scriptsize $q$} &\cellcolor{delcol}\De &\cellcolor{delcol}\cdots & \cellcolor{delcol}\De  & \cellcolor{Ncol} N & \cellcolor{Ncol} N & \cellcolor{Ncol} \cdots &\hspace{2mm}& (i,j)^\sharp\\ \hhline{~|-|-|-|-|-|-|~~} 
&\cellcolor{delcol}\De &\cellcolor{delcol}\cdots & \cellcolor{delcol}\De  & \cellcolor{Rcol} R & \cellcolor{Rcol}R & \cellcolor{Rcol} \cdots && (i,i+1)^\sharp\\ \hhline{~|-|-|-|-|-|-|~~}
&\cellcolor{delcol}\vvdots &\cellcolor{delcol}\vvdots & \cellcolor{delcol}\vvdots  & \cellcolor{Rcol} \vvdots & \cellcolor{Rcol}\vvdots & \cellcolor{Rcol} \vvdots && \vvdots\\ \hhline{~|-|-|-|-|-|-|~~}
\text{\scriptsize $0$} &\cellcolor{delcol}\De &\cellcolor{delcol}\cdots & \cellcolor{delcol}\De  & \cellcolor{Rcol} R & \cellcolor{Rcol}R & \cellcolor{Rcol} \cdots && (i,i+1)^\sharp\\ \hhline{~|-|-|-|-|-|-|~~}
\multicolumn{1}{c}{}&\multicolumn{1}{c}{}&\multicolumn{1}{c}{}&\multicolumn{1}{c}{}&\multicolumn{1}{c}{\text{\scriptsize $i$}}&\multicolumn{1}{c}{}&\multicolumn{1}{c}{}&\multicolumn{1}{c}{}
\end{array}.
\]
\item \label{pc6} If $q\geq3$, $i<j$, $(\al,\be)\in\H$ and $\pd(\al,\be)\not\in\A_q$, then with $d:=j-i$ we have ${\pc\ba\bb=\cgx(\Th,M)}$ for the exceptional C-pair
\[
\Cmatsetup
\begin{array}{r|c|c|c|c|c|c|cc}\hhline{~|-|-|-|-|-|-|~~}
\text{\scriptsize $q$} &\cellcolor{delcol}\De &\cellcolor{delcol}\cdots & \cellcolor{delcol}\De  & \cellcolor{Ncol} \A_q & \cellcolor{Ncol}\A_q & \cellcolor{Ncol} \cdots &\hspace{2mm}& (i,i+2d)^\sharp\\ \hhline{~|-|-|-|-|-|-|~~} 
&\cellcolor{delcol}\De &\cellcolor{delcol}\cdots & \cellcolor{delcol}\De  & \cellcolor{Rcol} R & \cellcolor{Rcol}R & \cellcolor{Rcol} \cdots && (i,i+1)^\sharp\\ \hhline{~|-|-|-|-|-|-|~~}
&\cellcolor{delcol}\vvdots &\cellcolor{delcol}\vvdots & \cellcolor{delcol}\vvdots  & \cellcolor{Rcol} \vvdots & \cellcolor{Rcol}\vvdots & \cellcolor{Rcol} \vvdots && \vvdots\\ \hhline{~|-|-|-|-|-|-|~~}
\text{\scriptsize $0$} &\cellcolor{delcol}\De &\cellcolor{delcol}\cdots & \cellcolor{delcol}\De  & \cellcolor{Rcol} R & \cellcolor{Rcol}R & \cellcolor{Rcol} \cdots && (i,i+1)^\sharp\\ \hhline{~|-|-|-|-|-|-|~~}
\multicolumn{1}{c}{}&\multicolumn{1}{c}{}&\multicolumn{1}{c}{}&\multicolumn{1}{c}{}&\multicolumn{1}{c}{\text{\scriptsize $i$}}&\multicolumn{1}{c}{}&\multicolumn{1}{c}{}&\multicolumn{1}{c}{}
\end{array}.
\]
\item \label{pc7} If $q=2$, $i=j$, $(\al,\be)\in\H$ and $\al\neq \be$, then $\pc\ba\bb=\cg(\Th,M)$ for the C-pair
\[
\Cmatsetup
\begin{array}{r|c|c|c|c|c|c|cc}\hhline{~|-|-|-|-|-|-|~~}
\text{\scriptsize $2$} &\cellcolor{delcol}\De &\cellcolor{delcol}\cdots & \cellcolor{delcol}\De  & \cellcolor{Ncol} \S_2 & \cellcolor{Ncol}\S_2 & \cellcolor{Ncol} \cdots &\hspace{2mm}& \De_\N\\ \hhline{~|-|-|-|-|-|-|~~} 
\text{\scriptsize $1$} &\cellcolor{delcol}\De &\cellcolor{delcol}\cdots & \cellcolor{delcol}\De  & \cellcolor{excepcol} \mu & \cellcolor{mucol}\mu & \cellcolor{mucol} \cdots && \De_\N\\ \hhline{~|-|-|-|-|-|-|~~}
\text{\scriptsize $0$} &\cellcolor{delcol}\De &\cellcolor{delcol}\cdots & \cellcolor{delcol}\De  & \cellcolor{mucol} \mu & \cellcolor{mucol}\mu & \cellcolor{mucol} \cdots && \De_\N\\ \hhline{~|-|-|-|-|-|-|~~}
\multicolumn{1}{c}{}&\multicolumn{1}{c}{}&\multicolumn{1}{c}{}&\multicolumn{1}{c}{}&\multicolumn{1}{c}{\text{\scriptsize $i$}}&\multicolumn{1}{c}{}&\multicolumn{1}{c}{}&\multicolumn{1}{c}{}
\end{array}.
\]
\item \label{pc8} If $q=2$, $i<j$, $(\al,\be)\in\H$ and $\al\neq \be$, then with $d:=j-i$ we have $\pc\ba\bb=\cgx(\Th,M)$ for the exceptional C-pair
\[
\Cmatsetup
\begin{array}{r|c|c|c|c|c|c|cc}\hhline{~|-|-|-|-|-|-|~~}
\text{\scriptsize $2$} &\cellcolor{delcol}\De &\cellcolor{delcol}\cdots & \cellcolor{delcol}\De  & \cellcolor{delcol} \De & \cellcolor{delcol}\De & \cellcolor{delcol} \cdots &\hspace{2mm}& (i,i+2d)^\sharp\\ \hhline{~|-|-|-|-|-|-|~~} 
\text{\scriptsize $1$} &\cellcolor{delcol}\De &\cellcolor{delcol}\cdots & \cellcolor{delcol}\De  & \cellcolor{Rcol} \mu & \cellcolor{Rcol}\mu & \cellcolor{Rcol} \cdots && (i,i+d)^\sharp\\ \hhline{~|-|-|-|-|-|-|~~}
\text{\scriptsize $0$} &\cellcolor{delcol}\De &\cellcolor{delcol}\cdots & \cellcolor{delcol}\De  & \cellcolor{Rcol} \mu & \cellcolor{Rcol}\mu & \cellcolor{Rcol} \cdots && (i,i+d)^\sharp\\ \hhline{~|-|-|-|-|-|-|~~}
\multicolumn{1}{c}{}&\multicolumn{1}{c}{}&\multicolumn{1}{c}{}&\multicolumn{1}{c}{}&\multicolumn{1}{c}{\text{\scriptsize $i$}}&\multicolumn{1}{c}{}&\multicolumn{1}{c}{}&\multicolumn{1}{c}{}
\end{array}.
\]
\item \label{pc9} If $q=1$ and $i\leq j$, and if $(\al,\be)$ belongs to one of $R_{I_1}\sm(\lam_1\cup\rho_1)$, $\lam_1\sm\mu_1$ or $\rho_1\sm\mu_1$, then with $\xi:=R$, $\lam$ or~$\rho$, respectively, we have $\pc\ba\bb=\cg(\Th,M)$ for the C-pair
\[
\Cmatsetup
\begin{array}{r|c|c|c|c|c|c|cc}\hhline{~|-|-|-|-|-|-|~~}
\text{\scriptsize $1$} &\cellcolor{delcol}\De &\cellcolor{delcol}\cdots & \cellcolor{delcol}\De  & \cellcolor{Rcol} \xi & \cellcolor{Rcol}\xi & \cellcolor{Rcol} \cdots &\hspace{2mm}& (i,i+1)^\sharp\\ \hhline{~|-|-|-|-|-|-|~~}
\text{\scriptsize $0$} &\cellcolor{delcol}\De &\cellcolor{delcol}\cdots & \cellcolor{delcol}\De  & \cellcolor{Rcol} \xi & \cellcolor{Rcol}\xi & \cellcolor{Rcol} \cdots && (i,i+1)^\sharp\\ \hhline{~|-|-|-|-|-|-|~~}
\multicolumn{1}{c}{}&\multicolumn{1}{c}{}&\multicolumn{1}{c}{}&\multicolumn{1}{c}{}&\multicolumn{1}{c}{\text{\scriptsize $i$}}&\multicolumn{1}{c}{}&\multicolumn{1}{c}{}&\multicolumn{1}{c}{}
\end{array}.
\]
\item \label{pc10} If $q=1$, $r=0$ and $i>j$, and if $(\al,\be)$ belongs to one of $R_{I_1}\sm(\lam_1\cup\rho_1)$, $\lam_1\sm\mu_1$ or $\rho_1\sm\mu_1$, then with $\xi:=R$, $\lam$ or~$\rho$, respectively, we have $\pc\ba\bb=\cg(\Th,M)$ for the C-pair
\[
\Cmatsetup
\begin{array}{r|c|c|c|c|c|c|c|c|c|cc}\hhline{~|-|-|-|-|-|-|-|-|-|~~}
\text{\scriptsize $1$} &\cellcolor{delcol}\De &\cellcolor{delcol}\cdots & \cellcolor{delcol}\De& \cellcolor{delcol}\De &\cellcolor{delcol}\cdots & \cellcolor{delcol}\De  & \cellcolor{Rcol} \xi & \cellcolor{Rcol}\xi & \cellcolor{Rcol} \cdots &\hspace{2mm}& (i,i+1)^\sharp\\ \hhline{~|-|-|-|-|-|-|-|-|-|~~}
\text{\scriptsize $0$} &\cellcolor{delcol}\De &\cellcolor{delcol}\cdots & \cellcolor{delcol}\De  & \cellcolor{Rcol} \xi & \cellcolor{Rcol}\cdots & \cellcolor{Rcol} \xi & \cellcolor{Rcol} \xi & \cellcolor{Rcol}\xi & \cellcolor{Rcol} \cdots && (j,j+1)^\sharp\\ \hhline{~|-|-|-|-|-|-|-|-|-|~~}
\multicolumn{1}{c}{}&\multicolumn{1}{c}{}&\multicolumn{1}{c}{}&\multicolumn{1}{c}{}&\multicolumn{1}{c}{\text{\scriptsize $j$}}&\multicolumn{1}{c}{}&\multicolumn{1}{c}{}&\multicolumn{1}{c}{\text{\scriptsize $i$}}&\multicolumn{1}{c}{}&\multicolumn{1}{c}{}&\multicolumn{1}{c}{}
\end{array}.
\]
\item \label{pc11} If $q=0$ and $i\leq j$, and if $(\al,\be)$ belongs to one of $R_{I_0}\sm(\lam_0\cup\rho_0)$, $\lam_0\sm\De_{D_0}$ or $\rho_0\sm\De_{D_0}$, then with $\xi:=R$, $\lam$ or~$\rho$, respectively, we have $\pc\ba\bb=\cg(\Th,M)$ for the C-pair
\[
\Cmatsetup
\begin{array}{r|c|c|c|c|c|c|cc}\hhline{~|-|-|-|-|-|-|~~}
\text{\scriptsize $0$} &\cellcolor{delcol}\De &\cellcolor{delcol}\cdots & \cellcolor{delcol}\De  & \cellcolor{Rcol} \xi & \cellcolor{Rcol}\xi & \cellcolor{Rcol} \cdots &\hspace{2mm} (i,i+1)^\sharp\\ \hhline{~|-|-|-|-|-|-|~~}
\multicolumn{1}{c}{}&\multicolumn{1}{c}{}&\multicolumn{1}{c}{}&\multicolumn{1}{c}{}&\multicolumn{1}{c}{\text{\scriptsize $i$}}&\multicolumn{1}{c}{}&\multicolumn{1}{c}{}&\multicolumn{1}{c}{}
\end{array}.
\]
\item \label{pc12} If $q=r=1$ and $i=j$, and if $(\al,\be)$ belongs to one of $\mu_1\sm(\muup\cup\mudown)$, $\muup\sm\De_{D_1}$ or $\mudown\sm\De_{D_1}$, then with $\ze:=\mu$, $\muup$ or~$\mudown$, respectively, we have $\pc\ba\bb=\cg(\Th,M)$ for the C-pair
\[
\Cmatsetup
\begin{array}{r|c|c|c|c|c|c|cc}\hhline{~|-|-|-|-|-|-|~~}
\text{\scriptsize $1$} &\cellcolor{delcol}\De &\cellcolor{delcol}\cdots & \cellcolor{delcol}\De  & \cellcolor{excepcol} \ze & \cellcolor{mucol}\mu & \cellcolor{mucol} \cdots &\hspace{2mm}& \De_\N\\ \hhline{~|-|-|-|-|-|-|~~}
\text{\scriptsize $0$} &\cellcolor{delcol}\De &\cellcolor{delcol}\cdots & \cellcolor{delcol}\De  & \cellcolor{mucol} \mu & \cellcolor{mucol}\mu & \cellcolor{mucol} \cdots && \De_\N\\ \hhline{~|-|-|-|-|-|-|~~}
\multicolumn{1}{c}{}&\multicolumn{1}{c}{}&\multicolumn{1}{c}{}&\multicolumn{1}{c}{}&\multicolumn{1}{c}{\text{\scriptsize $i$}}&\multicolumn{1}{c}{}&\multicolumn{1}{c}{}&\multicolumn{1}{c}{}
\end{array}.
\]
\item \label{pc13} If $q=r=1$, $i<j$, and $(\al,\be)\in\mu_1\sm\De_{D_1}$, then $\pc\ba\bb=\cg(\Th,M)$ for the C-pair
\[
\Cmatsetup
\begin{array}{r|c|c|c|c|c|c|cc}\hhline{~|-|-|-|-|-|-|~~}
\text{\scriptsize $1$} &\cellcolor{delcol}\De &\cellcolor{delcol}\cdots & \cellcolor{delcol}\De  & \cellcolor{Rcol} \mu & \cellcolor{Rcol}\mu & \cellcolor{Rcol} \cdots &\hspace{2mm}& (i,j)^\sharp\\ \hhline{~|-|-|-|-|-|-|~~}
\text{\scriptsize $0$} &\cellcolor{delcol}\De &\cellcolor{delcol}\cdots & \cellcolor{delcol}\De  & \cellcolor{Rcol} \mu & \cellcolor{Rcol}\mu & \cellcolor{Rcol} \cdots && (i,j)^\sharp\\ \hhline{~|-|-|-|-|-|-|~~}
\multicolumn{1}{c}{}&\multicolumn{1}{c}{}&\multicolumn{1}{c}{}&\multicolumn{1}{c}{}&\multicolumn{1}{c}{\text{\scriptsize $i$}}&\multicolumn{1}{c}{}&\multicolumn{1}{c}{}&\multicolumn{1}{c}{}
\end{array}.
\]
\item \label{pc14} If $q=1$, $r=0$, $i\leq j$, and $\wh\al=\be$, then $\pc\ba\bb=\cg(\Th,M)$ for the C-pair
\[
\Cmatsetup
\begin{array}{r|c|c|c|c|c|c|cc}\hhline{~|-|-|-|-|-|-|~~}
\text{\scriptsize $1$} &\cellcolor{delcol}\De &\cellcolor{delcol}\cdots & \cellcolor{delcol}\De  & \cellcolor{Rcol} \mu & \cellcolor{Rcol}\mu & \cellcolor{Rcol} \cdots &\hspace{2mm}& (i,j+1)^\sharp\\ \hhline{~|-|-|-|-|-|-|~~}
\text{\scriptsize $0$} &\cellcolor{delcol}\De &\cellcolor{delcol}\cdots & \cellcolor{delcol}\De  & \cellcolor{Rcol} \mu & \cellcolor{Rcol}\mu & \cellcolor{Rcol} \cdots && (i,j+1)^\sharp\\ \hhline{~|-|-|-|-|-|-|~~}
\multicolumn{1}{c}{}&\multicolumn{1}{c}{}&\multicolumn{1}{c}{}&\multicolumn{1}{c}{}&\multicolumn{1}{c}{\text{\scriptsize $i$}}&\multicolumn{1}{c}{}&\multicolumn{1}{c}{}&\multicolumn{1}{c}{}
\end{array}.
\]
\item \label{pc15} If $q=1$, $r=0$, $i=j+1$ and $\wh\al=\be$, then $\pc\ba\bb=\cg(\Th,M)$ for the C-pair
\[
\Cmatsetup
\begin{array}{r|c|c|c|c|c|c|cc}\hhline{~|-|-|-|-|-|-|~~}
\text{\scriptsize $1$} &\cellcolor{delcol}\De &\cellcolor{delcol}\cdots & \cellcolor{delcol}\De  & \cellcolor{excepcol} \De & \cellcolor{mucol}\mu & \cellcolor{mucol} \cdots &\hspace{2mm}& \De_\N\\ \hhline{~|-|-|-|-|-|-|~~}
\text{\scriptsize $0$} &\cellcolor{delcol}\De &\cellcolor{delcol}\cdots & \cellcolor{delcol}\De  & \cellcolor{mucol} \mu & \cellcolor{mucol}\mu & \cellcolor{mucol} \cdots && \De_\N\\ \hhline{~|-|-|-|-|-|-|~~}
\multicolumn{1}{c}{}&\multicolumn{1}{c}{}&\multicolumn{1}{c}{}&\multicolumn{1}{c}{}&\multicolumn{1}{c}{\text{\scriptsize $j$}}&\multicolumn{1}{c}{\text{\scriptsize $i$}}&\multicolumn{1}{c}{}&\multicolumn{1}{c}{}
\end{array}. 
\]
\item \label{pc16} If $q=1$, $r=0$, $i>j+1$, and $\wh\al=\be$, then with $d:=i-j-1$ we have $\pc\ba\bb=\cg(\Th,M)$ for the C-pair
\[
\Cmatsetup
\begin{array}{r|c|c|c|c|c|c|c|c|c|c|c|cc}\hhline{~|-|-|-|-|-|-|-|-|-|-|-|~~}
\text{\scriptsize $1$} &\cellcolor{delcol}\De &\cellcolor{delcol}\cdots & \cellcolor{delcol}\De& \cellcolor{delcol}\De & \cellcolor{delcol}\De &\cellcolor{delcol}\cdots & \cellcolor{delcol}\De  & \cellcolor{mucol}\mu & \cellcolor{Rcol} \mu & \cellcolor{Rcol}\mu & \cellcolor{Rcol} \cdots &\hspace{2mm}& (i+1,i+1+d)^\sharp\\ \hhline{~|-|-|-|-|-|-|-|-|-|-|-|~~}
\text{\scriptsize $0$} &\cellcolor{delcol}\De &\cellcolor{delcol}\cdots & \cellcolor{delcol}\De  & \cellcolor{mucol}\mu & \cellcolor{Rcol} \mu & \cellcolor{Rcol}\cdots & \cellcolor{Rcol} \mu & \cellcolor{Rcol} \mu & \cellcolor{Rcol} \mu & \cellcolor{Rcol}\mu & \cellcolor{Rcol} \cdots && (j+1,j+1+d)^\sharp\\ \hhline{~|-|-|-|-|-|-|-|-|-|-|-|~~}
\multicolumn{1}{c}{}&\multicolumn{1}{c}{}&\multicolumn{1}{c}{}&\multicolumn{1}{c}{}&\multicolumn{1}{c}{\text{\scriptsize $j$}}&\multicolumn{1}{c}{}&\multicolumn{1}{c}{}&\multicolumn{1}{c}{}&\multicolumn{1}{c}{\text{\scriptsize $i$}}&\multicolumn{1}{c}{}&\multicolumn{1}{c}{}&\multicolumn{1}{c}{}&\multicolumn{1}{c}{}
\end{array}.
\]
\een
The above exhaust all principal congruences on $\Ptw{n}$.
\end{thm}

\pf
In each part, write $\Pair=(\Th,M)$ for the stated C-pair, and let $\Pair'=(\Th',M')$ be the C-pair associated to $\pc\ba\bb$.  

The proof for each part follows the same basic pattern.  
For all parts, except \ref{pc6} and \ref{pc8}, we use Lemma \ref{lem:fg}\ref{it:fg1}, and it is routine to verify that $(\ba,\bb)\in\cg(\Pair)$.
In parts \ref{pc6} and \ref{pc8} we use Lemma \ref{lem:fg}\ref{it:fg2} instead, and it is equally routine to verify that
$(\ba,\bb)\in\cgx(\Pair)\sm\cg(\Pair)$.
It then remains to check that $\Pair\leqc\Pair'$, which amounts to showing that 
\[
\th_s\sub\th_s' \AND M_{sk}\leqc M_{sk}' \qquad\text{for all $s\in\bnz$ and $k\in\N$.}
\]
The first inequality is vacuously true when $\theta_s=\Delta_\N$. Likewise, the second inequality needs to be verified only when $M_{sk}\neq \Delta$, and, furthermore, by Remark \ref{rem:leqc}, when ${M_{sk}\grc M_{s,k-1}}$ (or $k=0$).

Below we give some representative proofs; the rest are similar.  It will be convenient to write~$\si:=\pc\ba\bb$.

\ref{pc7}  As just noted, 
we only need to check that $M_{2i}'\geqc\S_2$ and $M_{1i}',M_{0i}'\geq\mu$.  From the form of $(\ba,\bb)$, we have $\si_{2i}\not=\De_{D_2}$.  (The~$\si_{sk}$ notation is explained in Remark~\ref{rem:amb}.)  It follows from Table \ref{tab:M} that $M_{2i}'\not=\De$, and so from Figure~\ref{fig:M} that $M_{2i}'\geqc\S_2$.  
Combining this with the verticality conditions in Definition \ref{de:Cpair}, it follows that $M_{1i}'\geqc\mu$, and then by examining the row types in Table~\ref{tab:RT} that $M_{0i}'\geqc\mu$.  

\ref{pc8}  As this part involves an exceptional congruence, a little more work is required.  As in 
part \ref{pc7}, we need to show that $M_{1i}',M_{0i}'\geq\mu$, but this time we also need to show that $\th_s'\supseteq\th_s$ for $s=0,1,2$.  

Beginning with the latter, first note that since $M'\leqc M$ (as $\si\sub\cgx(\Pair)$), row $2$ of $M'$ consists entirely of~$\De$s.  Thus, by the form of $(\ba,\bb)$, this pair can only belong to $\si$ via \ref{C9}; it follows that $M'$ is exceptional, with $\hgt(M')=2$.  In particular, we have $\th_2'=(m,m+2e)^\sharp$ for some $m\geq0$ and $e\geq1$, and moreover we have $(m,m+e)^\sharp\sub\th_1'$ (see Definition \ref{def:exc}).  
Since $(\ba,\bb)=((i,\al),(j,\be))$ belongs to $\si$ via \ref{C9}, we have $(i,i+d)=(i,j)\in(m,m+e)^\sharp$, and so $i\geq m$ and $e\mid d$, which gives $\th_2\sub\th_2'$.  The reverse inclusion follows from $\Th'\leqc\Th$, and so $\th_2'=\th_2=(i,i+2d)^\sharp$.
Also, $\th_0=\th_1=(i,i+d)^\sharp=(m,m+e)^\sharp\sub\th_1'\sub\th_0'$.
Moving on to the matrix entries, since $\hgt(M')=2$ and $i=\min\th_2'$, it follows from Definition \ref{def:exc} that $M_{1i}'\geq\mu$, and then again by examining the row types that $M_{0i}'\geq\mu$ as well.

\ref{pc10}  We consider the case in which $(\al,\be)\in\lam_1\sm\mu_1$, the others being similar.  Referring to Theorem \ref{thm:CongPn}, this means that $\wh\al\rL\wh\be({}=\be)$ but $\wh\al\not=\wh\be$.  Examining Definition \ref{de:cg}, it follows that $(\ba,\bb)\in\si$ via \ref{C2} or \ref{C4}, so that $M_{1i}'=M_{0j}'\in\{\lam,R\}$, and so $M_{1i}',M_{0j}'\geqc\lam$.  As usual this is all we need to show regarding the matrix entries.  

Examining Table \ref{tab:RT}, and remembering that $\per\th_s'=1$ if $\xi\not=\mu$ in row $s\leq1$ of $M'$, and that an entry ${}\grc\mu$ can only appear after $\min\th_s'$, we see that $\th_1'\supseteq(i,i+1)^\sharp=\th_1$ and $\th_0'\supseteq(j,j+1)^\sharp=\th_0$.

\ref{pc16}  Since $(\ba,\bb)\in\si$, and since $\ba$ and $\bb$ belong to the $\D$-classes $D_{1i}$ and $D_{0j}$ in different rows, we see that $M_{1i}'=M_{0j}'\geqc\mu$, and again this takes care of the matrices.

Since $\th_1'\sub\th_1=(i+1,i+1+d)^\sharp$ and $\th_0'\sub\th_0=(j+1,j+1+d)^\sharp$, we have $i<\min\th_1'$ and $j<\min\th_0'$.  Thus, since $(\ba,\bb)\in\si\cap(D_{1i}\times D_{0j})$, $M'$ is of type \ref{RT2}, \ref{RT5} or \ref{RT7}.  Since $i>j+1$ we can rule out types \ref{RT2} and \ref{RT5}, and so
\[
\th_0'=(j+1,j+1+e)^\sharp \ANd \th_1'=(i+1,i+1+e)^\sharp \qquad\text{for some $e\geq1$ with $i\equiv j+1\!\!\!\!\pmod e$.}
\]
Thus, since we wish to show that $\th_0'\supseteq\th_0$ and $\th_1'\supseteq\th_1$, it remains to show that $e\mid d=i-j-1$.  But this follows from $i\equiv j+1\pmod e$.
\epf

Our next result, Theorem \ref{thm:fg}, shows that every congruence on $\Ptw n$ can be generated by at most $\lceil\frac{5n}2\rceil$ pairs.  As the proof is somewhat technical, we begin by considering some examples, all with $n=4$.

\begin{exa}\label{ex:fg1}
Consider the Rees congruence $R_I$, where $I:=I_{33}\cup I_{12}\cup I_{00}$.  Then $R_I=\cg(\Pair)$ for the C-pair
\[
\Pi:= \Cmatsetup\begin{array}{|c|c|c|c|c|c|cc}
\hhline{|-|-|-|-|-|-|~~}
\cellcolor{delcol}\Delta & \cellcolor{delcol}\Delta & \cellcolor{delcol}\Delta & \cellcolor{delcol}\Delta & \cellcolor{delcol}\Delta & \cellcolor{delcol}\cdots &\hspace{2mm}& \Delta_\N \\ \hhline{|-|-|-|-|-|-|~~}
\cellcolor{delcol}\Delta & \cellcolor{delcol}\Delta & \cellcolor{delcol}\Delta & \cellcolor{Rcol}R & \cellcolor{Rcol}R & \cellcolor{Rcol}\cdots && (3,4)^\sharp \\ \hhline{|-|-|-|-|-|-|~~}
\cellcolor{delcol}\Delta & \cellcolor{delcol}\Delta & \cellcolor{delcol}\Delta & \cellcolor{Rcol}R & \cellcolor{Rcol}R & \cellcolor{Rcol}\cdots &&  (3,4)^\sharp \\ \hhline{|-|-|-|-|-|-|~~}
\cellcolor{delcol}\Delta & \cellcolor{delcol}\Delta & \cellcolor{Rcol}R & \cellcolor{Rcol}R & \cellcolor{Rcol}R & \cellcolor{Rcol}\cdots &&  (2,3)^\sharp\\ \hhline{|-|-|-|-|-|-|~~}
\cellcolor{Rcol}R & \cellcolor{Rcol}R & \cellcolor{Rcol}R & \cellcolor{Rcol}R & \cellcolor{Rcol}R & \cellcolor{Rcol}\cdots
&& \nabla_\N \\ \hhline{|-|-|-|-|-|-|~~}
\end{array}
\ .
\]
We claim that $R_I = \big\{(\ba,\bb),(\bc,\bd)\big\}^\sharp$ for any $\ba\in D_{33}$, $\bb\in D_{12}$, 
$\bc=(0,\ga)\in D_{00}$, $\bd=(0,\de)\in D_{00}$, with $(\ga,\de)\not\in\L\cup\R$.  Indeed, first note that by Theorem \ref{thm:pc}\ref{pc3} and \ref{pc11} we have ${\pc\ba\bb=R_{I_{33}\cup I_{12}}}$ and $\pc\bc\bd=R_{I_{00}}$.  And then, since $R_{I_1}\vee R_{I_2}=R_{I_1\cup I_2}$ for any ideals $I_1,I_2$, we have
\[
R_I = R_{I_{33}\cup I_{12}} \vee R_{I_{00}} = \pc\ba\bb \vee \pc\bc\bd = \big\{(\ba,\bb),(\bc,\bd)\big\}^\sharp.
\]
Many other pairs of pairs could be chosen to generate $R_I$.  
For example, $\bd$ could be replaced by any element from $I\setminus D_{00}$.
\end{exa}

\begin{exa}\label{ex:fg2}
Consider the C-pair $\Pair=(\Th,M)$ defined by
\[
\Cmatsetup
\Pair := \begin{array}{|c|c|c|c|c|c|c|c|c|c|c|cc}
\hhline{|-|-|-|-|-|-|-|-|-|-|-|~~}
\cellcolor{delcol}\Delta & \cellcolor{delcol}\Delta &\cellcolor{delcol}\Delta &\cellcolor{delcol}\Delta &\cellcolor{delcol}\Delta &\cellcolor{delcol}\Delta &\cellcolor{delcol}\Delta &\cellcolor{delcol}\Delta &\cellcolor{delcol}\Delta &\cellcolor{delcol}\Delta &\cellcolor{delcol}\cdots &\hspace{2mm}&\Delta_\N \\ \hhline{|-|-|-|-|-|-|-|-|-|-|-|~~}
\cellcolor{delcol}\Delta & \cellcolor{delcol}\Delta &\cellcolor{delcol}\Delta &\cellcolor{delcol}\Delta &\cellcolor{delcol}\Delta &\cellcolor{delcol}\Delta &\cellcolor{delcol}\Delta &\cellcolor{delcol}\Delta &\cellcolor{delcol}\Delta &\cellcolor{delcol}\Delta &\cellcolor{delcol}\cdots && (9,13)^\sharp \\ \hhline{|-|-|-|-|-|-|-|-|-|-|-|~~}
\cellcolor{delcol}\Delta & \cellcolor{delcol}\Delta &\cellcolor{delcol}\Delta &\cellcolor{delcol}\Delta &\cellcolor{delcol}\Delta &\cellcolor{delcol}\Delta &\cellcolor{delcol}\Delta &\cellcolor{delcol}\Delta &\cellcolor{delcol}\Delta &\cellcolor{delcol}\Delta &\cellcolor{delcol}\cdots && (8,12)^\sharp \\ \hhline{|-|-|-|-|-|-|-|-|-|-|-|~~}
\cellcolor{delcol}\Delta & \cellcolor{excepcol}\mudown &\cellcolor{mucol}\mu &\cellcolor{mucol}\mu &\cellcolor{mucol}\mu &\cellcolor{mucol}\mu &\cellcolor{Rcol}\mu &\cellcolor{Rcol}\mu &\cellcolor{Rcol}\mu &\cellcolor{Rcol}\mu &\cellcolor{Rcol}\cdots && (6,8)^\sharp \\ \hhline{|-|-|-|-|-|-|-|-|-|-|-|~~}
\cellcolor{delcol}\Delta & \cellcolor{mucol}\mu &\cellcolor{mucol}\mu &\cellcolor{mucol}\mu &\cellcolor{mucol}\mu &\cellcolor{Rcol}\mu &\cellcolor{Rcol}\mu &\cellcolor{Rcol}\mu &\cellcolor{Rcol}\mu &\cellcolor{Rcol}\mu &\cellcolor{Rcol}\cdots && (5,7)^\sharp \\\hhline{|-|-|-|-|-|-|-|-|-|-|-|~~}
\end{array}
\ .
\]
Fix arbitrary partitions $\al\in D_3$, $\be\in D_2$, and $\ga,\de\in D_1$ such that $(\ga,\de)\in\mudown\sm\De_{D_1}$.  We claim that $\cg(\Pair)=\Om^\sharp$, where
\[
\Om = \big\{
((9,\al),(13,\al)),
((8,\be),(12,\be)),
((6,\ga),(8,\ga)),
((1,\ga),(1,\de))
\big\}.
\]
To show this, let $\Pair'=(\Th',M')$ be the C-pair associated to $\Om^\sharp$.  It is easy to check that each pair from $\Om$ belongs to $\cg(\Pair)$.  Thus, it suffices by Lemma \ref{lem:fg}\ref{it:fg1} to show that $\Pair\leqc\Pair'$.  
Again, it is enough to show that 
\[
M_{11}'\geqc\mudown \COMMA M_{12}',M_{01}'\geqc\mu \AND \th_q\sub\th_q' \text{ for $q=0,1,2,3$.}
\]
By Theorem \ref{thm:pc}\ref{pc12}, the C-matrix corresponding to the congruence $((1,\ga),(1,\de))^\sharp$ is of type~\ref{RT2}, with $i=1$ and $\ze=\mudown$.  Since ${((1,\ga),(1,\de))^\sharp\sub\Om^\sharp}$, it follows that $M_{11}'\geqc\mudown$ and $M_{12}',M_{01}'\geqc\mu$.  This takes care of the matrix entries, and in particular we have $M'=M$.

Turning to the C-chains, first note that the pair $((9,\al),(13,\al))\in\Om^\sharp$ tells us that $(9,13)\in\th_3'$ (see Remark \ref{rem:amb}). It follows that $\th_3'\supseteq(9,13)^\sharp=\th_3$, and we similarly obtain $\th_q'\supseteq\th_q$ for $q=1,2$.  
Since $M'=M$, it follows that $\Pair'$ is of type \ref{RT2} or \ref{RT5}, and since $\th_1'=(6,8)^\sharp\not=\De_\N$ the former is ruled out.  It then follows (by the definition of type \ref{RT5}) that $\th_0'=(5,7)^\sharp=\th_0$.
\end{exa}

\begin{exa}\label{ex:fg3}
Consider again the C-pair $\Pair=(\Th,M)$ from Example \ref{ex:fg2}, and keep the notation for the partitions $\al,\be,\ga,\de$ defined there.  Note in fact that $\Pair$ is exceptional, with exceptional row $\hgt(M)=2$.  This time let $\eta\in D_2$ be such that $\be\rH\eta$ but $\be\not=\eta$, and set
\[
\Xi = \big\{
((9,\al),(13,\al)),
((8,\be),(10,\eta)),
((6,\ga),(8,\ga)),
((1,\ga),(1,\de))
\big\}.
\]
Note that the only place in which $\Xi$ differs from $\Omega$ is the second component of the second generating pair.
Note also that $((8,\be),(10,\eta))\in\cgx(\Pair)\sm\cg(\Pair)$.  Thus, we can show that $\cgx(\Pair)=\Xi^\sharp$ by showing that $\Pair\leqc\Pair'$, this time by applying Lemma \ref{lem:fg}\ref{it:fg2}, where $\Pair'$ is the C-pair associated to $\Xi^\sharp$.  We establish $M\leqc M'$, and the inclusions $\th_q\sub\th_q'$ for $q\not=2$, as in Example \ref{ex:fg2}.  For $\th_2\sub\th_2'$ we use the pair $((8,\be),(10,\eta))$, and argue as in case \ref{pc8} in the proof of Theorem \ref{thm:pc}.
\end{exa}

\begin{exa}\label{ex:fg4}
Consider the C-pair
\[
\Pair := \Cmatsetup
 \begin{array}{|c|c|c|c|c|c|c|c|c|c|c|c|cc}
\hhline{|-|-|-|-|-|-|-|-|-|-|-|-|~~}
\cellcolor{delcol}\Delta & \cellcolor{delcol}\Delta &\cellcolor{delcol}\Delta &\cellcolor{delcol}\Delta &\cellcolor{delcol}\Delta &\cellcolor{delcol}\Delta &\cellcolor{Ncol} \mathcal K_4 &\cellcolor{Ncol}\mathcal A_4 &\cellcolor{Ncol}\S_4 &\cellcolor{Ncol}\S_4 &\cellcolor{Ncol}\S_4&\cellcolor{Ncol}\cdots &\hspace{2mm}&(9,11)^\sharp \\ \hhline{|-|-|-|-|-|-|-|-|-|-|-|-|~~}
\cellcolor{delcol}\Delta & \cellcolor{delcol}\Delta &\cellcolor{delcol}\Delta &\cellcolor{delcol}\Delta &\cellcolor{Ncol}\A_3 &\cellcolor{Ncol}\S_3 &\cellcolor{Rcol} R &\cellcolor{Rcol}R &\cellcolor{Rcol}R &\cellcolor{Rcol}R &\cellcolor{Rcol}R&\cellcolor{Rcol}\cdots &&(6,7)^\sharp \\ \hhline{|-|-|-|-|-|-|-|-|-|-|-|-|~~}
\cellcolor{delcol}\Delta & \cellcolor{delcol}\Delta &\cellcolor{Ncol}\S_2 &\cellcolor{Rcol}R &\cellcolor{Rcol}R &\cellcolor{Rcol}R &\cellcolor{Rcol} R &\cellcolor{Rcol}R &\cellcolor{Rcol}R &\cellcolor{Rcol}R &\cellcolor{Rcol}R&\cellcolor{Rcol}\cdots &&(3,4)^\sharp \\ \hhline{|-|-|-|-|-|-|-|-|-|-|-|-|~~}
\cellcolor{delcol}\Delta & \cellcolor{excepcol}\mu &\cellcolor{Rcol}R&\cellcolor{Rcol}R &\cellcolor{Rcol}R &\cellcolor{Rcol}R &\cellcolor{Rcol} R &\cellcolor{Rcol}R &\cellcolor{Rcol}R &\cellcolor{Rcol}R &\cellcolor{Rcol}R&\cellcolor{Rcol}\cdots &&(2,3)^\sharp \\ \hhline{|-|-|-|-|-|-|-|-|-|-|-|-|~~}
\cellcolor{Rcol}R & \cellcolor{Rcol}R &\cellcolor{Rcol}R&\cellcolor{Rcol}R &\cellcolor{Rcol}R &\cellcolor{Rcol}R &\cellcolor{Rcol} R &\cellcolor{Rcol}R &\cellcolor{Rcol}R &\cellcolor{Rcol}R &\cellcolor{Rcol}R&\cellcolor{Rcol}\cdots &&\nabla_\N \\ \hhline{|-|-|-|-|-|-|-|-|-|-|-|-|~~}
\end{array}
\ ,
\]
We claim that $\cg(\Pair)=\Om^\sharp$, where $\Om$ consists of:
\bena
\item \label{fg41} any pair from $D_{36}\times D_{23}$, and any pair from $D_{12}\times D_{00}$ satisfying the conditions of Theorem~\ref{thm:pc}\ref{pc10} with $\xi=R$;
\item \label{fg42} any pair from $D_{11}\times D_{11}$ as in Theorem \ref{thm:pc}\ref{pc12} with $\ze=\mu$;
\item \label{fg43} any pair from 
$D_{22}\times D_{22}$ as in Theorem \ref{thm:pc}\ref{pc7}, and any pairs from 
$D_{34}\times D_{34}$, 
$D_{35}\times D_{35}$, 
$D_{46}\times D_{46}$,
$D_{47}\times D_{47}$ and
$D_{48}\times D_{48}$
as in Theorem \ref{thm:pc}\ref{pc4} with $N=\A_3$, $\S_3$, $\mathcal K_4$, $\A_4$ and~$\S_4$, respectively;
\item \label{fg44} any pair from $D_{49}\times D_{4,11}$ as in Theorem \ref{thm:pc}\ref{pc2}.
\eena
Indeed, once again we have $\Om\sub\cg(\Pair)$, so by Lemma \ref{lem:fg}\ref{it:fg1} it remains to check that $\Pair\leqc\Pair'$, where $\Pair'$ is the C-pair of $\Om^\sharp$.  As in Example \ref{ex:fg1}, the pairs from \ref{fg41} above generate the Rees congruence associated to the ideal class $I(\Om^\sharp) = I_{36}\cup I_{23}\cup I_{12}\cup I_{00}$, and hence `fix' the $R$s of~$M'$; these also fix the congruences $\th_q'$ for $q\not=4$.  The pair in \ref{fg42} fixes the $\mu$ label.  The pairs in \ref{fg43} fix the $N$-symbols.  Finally, the pair in \ref{fg44} fixes the congruence $\th_4'$.
\end{exa}

Here is the main result of this section.  Its constructive proof allows one to find a generating set of suitably bounded size for an arbitrary congruence.

\begin{thm}\label{thm:fg}
Let $n\geq 1$.
\ben
\item \label{fg1} Every Rees congruence on $\Ptw{n}$ can be generated by at most $\lceil \frac{n+1}2\rceil$ pairs. 
\item \label{fg2} Every congruence on $\Ptw{n}$ can be generated by at most $\lceil\frac{5n}2\rceil$ pairs. 
\een
\end{thm}

\begin{proof}
\ref{fg1}  Let $I$ be an ideal of $\Ptw n$.  If $I=\emptyset$, then $R_I=\De_{\Ptw n}=\emptyset^\sharp$.  So suppose $I\not=\emptyset$, so that $I=I_{q_1i_1}\cup\dots\cup I_{q_ki_k}$ for some set of incomparable elements $(q_1,i_1),\dots,(q_k,i_k)$ of the poset $(\bnz,\leq)\times(\N,\geq)$.  Now,
\begin{equation}
\label{eq:Reesjoin}
R_I = R_{I_{q_1i_1}\cup I_{q_2i_2}} \vee R_{I_{q_3i_3}\cup I_{q_4i_4}} \vee \cdots,
\end{equation}
where the last term in the join is either $R_{I_{q_ki_k}}$ or $R_{I_{q_{k-1}i_{k-1}}\cup I_{q_ki_k}}$, depending on the parity of~$k$.  
Since $k\leq n+1$ and each Rees congruence in \eqref{eq:Reesjoin} is principal by
Theorem~\ref{thm:pc}\ref{pc3}, \ref{pc9}--\ref{pc11}, the assertion follows.

\ref{fg2}  Let $\sigma\in\Cong(\Ptw{n})$ be arbitrary, and let $\Pair=(\Theta,M)$ be the associated C-pair, so that~$\si$ is one of $\cg(\Pair)$ or $\cgx(\Pair)$.  We create a set $\Om$ with $\Om^\sharp=\si$, and with $|\Om|$ appropriately bounded, as follows.  Keeping Lemma \ref{lem:fg} in mind, we ensure that:
\bit
\item each pair from $\Om$ belongs to $\si$;
\item some pair from $\Om$ does not belong to $\cg(\Pair)$ if $\si=\cgx(\Pair)$ is exceptional;
\item each entry $M_{qi}$ of the C-matrix $M$ is `fixed' by some pair $(\ba,\bb)$ from $\Om$, in the sense that the $(q,i)$-th entry of the C-matrix corresponding to $\pc\ba\bb$ is ${}\geqc M_{qi}$;
\item each congruence $\th_q$ from the C-chain $\Th$ is also suitably fixed by a pair from $\Om$.
\eit
We now proceed to construct the set $\Om$, as a union $\Om:=\Om_1\cup\Om_2\cup\Om_3\cup\Om_4$, in four steps.  As in part \ref{fg1}, the (possibly empty) ideal $I:=I(\si)$ has the form $I=I_{q_1i_1}\cup\dots\cup I_{q_ki_k}$, where $k\geq0$ and the $(q_t,i_t)$ are incomparable.  Let $m=\max(q_1,\ldots,q_k)$, with the convention that $m=-1$ if $k=0$, and note that $k\leq m+1$.  Note also that 
\[
m=-1 \iff k=0 \iff I=\emptyset \iff M \text{ has no $R$-entries}.
\]
In what follows, the `type' \ref{pc1}--\ref{pc16} of a pair 
$(\ba,\bb)\in\Ptw{n}\times\Ptw{n}$
refers to the enumeration in Theorem \ref{thm:pc}

\setcounter{stepco}{0}

\step \label{fgS1}  As in part \ref{fg1}, we construct a set $\Om_1\sub\si$ such that $R_I=\Om_1^\sharp$ and $|\Om_1|=\lceil\frac k2\rceil$.  
This fixes any $R$-entries of $M$, and also fixes the associated congruences $\th_0,\ldots,\th_m$; each is of the form $\th_q=(i_q,i_q+1)^\sharp$ where $i_q$ is the position of the first $R$ in row $q$.

\step \label{fgS2}  Next we construct a set $\Om_2\sub\si$ to fix any non-trivial congruences among $\th_{m+1},\ldots,\th_n$. For any $q\in\{m+1,\ldots,n\}$ with $\th_q$ non-trivial, $\Om_2$ contains a suitable pair of type \ref{pc2} unless $\si$ is exceptional and $q=\hgt(M)$ in which case we include a pair of type~\ref{pc6} or~\ref{pc8} as appropriate.  In any case, we have $|\Om_2|\leq n-m$.

\step \label{fgS3}  
Next we construct a set $\Om_3\sub\si$ to fix the
entries in rows $0$ and $1$ of $M$ that are distinct from $\Delta$ and $R$.
Depending on the row type \ref{RT1}--\ref{RT7} of $M$, $\Om_3$ consists of the following pairs:
\smallskip

\ref{RT1}: no pairs are required (because $\th_0$ and $\th_1$ have been fixed by $\Om_2$);
\smallskip

\ref{RT2}:  one pair of type \ref{pc12} or \ref{pc15};
\smallskip

\ref{RT3}: at most one pair of type \ref{pc11}; the clause `at most' appears here as no such pair is needed when $\xi=R$ (since the $R$s were fixed by $\Om_1$); similar comments apply in the remaining cases;
\smallskip

\ref{RT4}: at most one pair of type \ref{pc9} or \ref{pc13};
\smallskip

\ref{RT5}: one pair of type \ref{pc12} or \ref{pc15}, and at most one of type \ref{pc2} or \ref{pc11};
\smallskip

\ref{RT6}: if $\ze\not=\De$, then one pair of type \ref{pc12}, and at most one of type \ref{pc2} or \ref{pc11}; if $\ze=\De$ and $\xi\not=\mu$, then at most one pair of type \ref{pc10}; if $\ze=\De$ and $\xi=\mu$, then one pair of type \ref{pc15} and one of type \ref{pc2};
\smallskip

\ref{RT7}: one pair of type \ref{pc16}, and at most one of type \ref{pc11}.
\smallskip

\noindent
By construction we have 
\begin{equation}\label{eq:Om3}
|\Om_3| \leq \begin{cases}
1 &\text{if $m\geq0$}\\
2 &\text{if $m=-1$.}
\end{cases}
\end{equation}

\step \label{fgS4}  Finally, we construct a set $\Om_4\sub\si$ of pairs of type \ref{pc4} or \ref{pc7} to fix the $N$-symbols in~$M$.  As in the proof of Theorem \ref{thm:pc} (see also Example \ref{ex:fg4}), we need only include one pair for each distinct $N$-symbol appearing in $M$.  
We claim that
\begin{equation}\label{eq:Om4}
|\Om_4| \leq \begin{cases}
2n-2 &\text{if $m=n$}\\
2m &\text{if $0\leq m\leq n-1$}\\
1 &\text{if $m=-1$.}
\end{cases}
\end{equation}
To see this, first recall that the $N$-symbols in row $q\geq2$ are the non-trivial normal subgroups of~$\S_q$.  Thus, the number of possible distinct $N$-symbols in rows $2,\ldots,n$ are $1,2,3,2,2,2,\ldots,2$, respectively.
When $m\geq1$, it follows from the verticality conditions in Definition \ref{de:Cpair} that only rows $2,\ldots,m+1$ can have $N$-symbols, so that $|\Om_4|\leq2m$; 
furthermore,
when $m=n$, there is no row $m+1$, so the upper bound is $2n-2$.  
When $m=0$ we have $|\Om_4|=0(=2m)$; indeed, here $M$ is necessarily of type \ref{RT3} (with $\xi=R$), and the verticality conditions prevent $M$ from having any $N$-symbols.
Finally, when $m=-1$ (i.e.~$I=\emptyset$), $M$ can have any of types \ref{RT1}--\ref{RT7} (with $\xi\not=R$), and $\S_2$ can occur in any of types \ref{RT2}--\ref{RT7}, but no other $N$-symbols can occur, meaning that $|\Om_4|\leq1$.

\bigskip

This completes the definition of the set $\Om:=\Om_1\cup\Om_2\cup\Om_3\cup\Om_4$, and, as explained above, we have $\si=\Om^\sharp$.  It remains to check the size of $\Om$.  By construction, and recalling that $k\leq m+1$, we have
\[
\textstyle |\Om| = |\Om_1|+|\Om_2|+|\Om_3|+|\Om_4| \leq \lceil\frac {m+1}2\rceil + (n-m) + |\Om_3| + |\Om_4|.
\]
It is now a matter of checking that this is bounded above by $\lceil \frac{5n}2\rceil$, which we do by combining~\eqref{eq:Om3} and \eqref{eq:Om4}.  When $m=n$, we have
\[
\textstyle
|\Om| \leq \lceil\frac {n+1}2\rceil + 0 + 1 + (2n-2) = \lceil\frac {5n-1}2\rceil.
\]
When $0\leq m\leq n-1$, we have
\[
\textstyle
|\Om|\leq\lceil\frac {m+1}2\rceil + (n-m) + 1 + 2m = \lceil\frac {m+1}2\rceil + n+(m+1) \leq \lceil\frac n2\rceil+2n = \lceil\frac{5n}2\rceil.
\]
When $m=-1$, similar calculations give $|\Om|\leq n+4$.  When $n\geq3$ we have $n+4\leq\lceil \frac{5n}2\rceil$.  Thus, the proof is complete except when $m=-1$ and $n\leq2$.

Now suppose $m=-1$ and $n=2$, and that $\Om$ as constructed above has size $n+4=6$.  Since $\lceil\frac{5n}2\rceil=5$, it suffices to show that at least one pair from $\Om$ is redundant.  By maximality of $|\Om|$ we must have $|\Om_3|=2$, so examining the cases in Step \ref{fgS3} we see that $M$ has one of types \ref{RT5}--\ref{RT7}.  In type \ref{RT5} the second pair is either of type \ref{pc2} or \ref{pc11}; in the first case the pair is redundant because such a pair already comes from $\Om_2$; in the second case the pair of type \ref{pc11} fixes $\th_0$, so the corresponding pair from $\Om_2$ is redundant.  Types \ref{RT6} and \ref{RT7} are treated in similar fashion.

Finally, when $m=-1$ and $n=1$, we have $|\Om|\leq n+4=5$, and we need to show this can be reduced to $\lceil\frac{5n}2\rceil=3$.  First note that in this case we have $\Om_4=\emptyset$ as there are no $N$-symbols when $n=1$, so in fact $|\Om|\leq4$.  If we had $|\Om|=4$, then we could remove a pair from $\Om$ as in the previous case.  
\end{proof}

\begin{rem}
\label{re:nobound}
We observe that there is no uniform constant bound (independent of $n$) for the numbers of pairs needed to generate congruences on $\Ptw{n}$.  Indeed, for any $n\geq 2$, consider the congruence $\sigma=\cg(\Pair)$, for the C-pair
\[
\Pair := \Cmatsetup
 \begin{array}{|c|c|c|c|c|c|c|c|cc}
\hhline{|-|-|-|-|-|-|-|-|~~}
\cellcolor{delcol}\Delta & \cellcolor{delcol}\Delta & \cellcolor{delcol}\cdots &\cellcolor{delcol}\Delta &\cellcolor{Ncol}\S_n &\cellcolor{Rcol}R &\cellcolor{Rcol}R &\cellcolor{Rcol} \cdots &\hspace{2mm}&(n-1,n)^\sharp \\ \hhline{|-|-|-|-|-|-|-|-|~~}
\cellcolor{delcol}\vdots&\cellcolor{delcol}\vdots&\cellcolor{delcol}\iddots&\cellcolor{delcol}\iddots&\cellcolor{delcol}\iddots&\cellcolor{delcol}\vdots&\cellcolor{delcol}\vdots&\cellcolor{delcol}\vdots&&\vdots\\ \hhline{|-|-|-|-|-|-|-|-|~~}
\cellcolor{delcol}\Delta & \cellcolor{delcol}\Delta &\cellcolor{Ncol}\S_4 &\cellcolor{Rcol} R &\cellcolor{Rcol}\cdots &\cellcolor{Rcol}R &\cellcolor{Rcol}R &\cellcolor{Rcol} \cdots &&(3,4)^\sharp \\ \hhline{|-|-|-|-|-|-|-|-|~~}
\cellcolor{delcol}\Delta & \cellcolor{Ncol}\S_3 &\cellcolor{Rcol} R &\cellcolor{Rcol} R &\cellcolor{Rcol}\cdots &\cellcolor{Rcol}R &\cellcolor{Rcol}R &\cellcolor{Rcol} \cdots &&(2,3)^\sharp \\ \hhline{|-|-|-|-|-|-|-|-|~~}
\cellcolor{Ncol}\S_2 & \cellcolor{Rcol}R&\cellcolor{Rcol}R &\cellcolor{Rcol} R &\cellcolor{Rcol}\cdots &\cellcolor{Rcol}R &\cellcolor{Rcol} R &\cellcolor{Rcol}\cdots &&(1,2)^\sharp \\ \hhline{|-|-|-|-|-|-|-|-|~~}
\cellcolor{Rcol}R & \cellcolor{Rcol}R &\cellcolor{Rcol}R&\cellcolor{Rcol} R &\cellcolor{Rcol}\cdots &\cellcolor{Rcol}R &\cellcolor{Rcol}R &\cellcolor{Rcol} \cdots &&(0,1)^\sharp \\ \hhline{|-|-|-|-|-|-|-|-|~~}
\cellcolor{Rcol}R & \cellcolor{Rcol}R &\cellcolor{Rcol}R&\cellcolor{Rcol} R &\cellcolor{Rcol}\cdots &\cellcolor{Rcol}R &\cellcolor{Rcol}R &\cellcolor{Rcol} \cdots &&(0,1)^\sharp \\ \hhline{|-|-|-|-|-|-|-|-|~~}
\end{array}
\ .
\]
If $\si=\Om^\sharp$, then $\Om$ must contain at least one pair from $D_{q,q-2}\times D_{q,q-2}$ for each $q=2,\dots,n$, because otherwise $\Om$
would be wholly contained in the congruence defined by the C-pair obtained from $\Pair$ by replacing the entry $M_{q,q-2}=\S_q$ by $\Delta$.
Hence $|\Om|\geq n-1$, and the assertion is proved.

Congruences minimally generated by more than $n-1$ pairs can easily be constructed, as in Example \ref{ex:fg4}.
\end{rem}

\section{Congruences of finite twisted partition monoids \boldmath{$\Ptw{n,d}$}}
\label{sec:fintwist}

We now turn our attention to the finite $d$-twisted partition monoids $\Ptw{n,d}=(\bdz\times\P_n)\cup\{\zero\}$
introduced in Subsection \ref{subsec:Ptw}.
Their congruences were described in \cite{ERtwisted1}, by viewing $\Ptw{n,d}$ as a quotient of $\Ptw{n}$ and deploying the Correspondence Theorem.
In the course of this argument several simplifications take place: one can dispense with C-chains; C-matrices become finite;  several row types are not possible; and there are no exceptional matrices/congruences.

\begin{defn}[\bf finitary C-matrix]
\label{de:fC}
A matrix ${M=(M_{qi})_{\bnz\times \bdz}}$
is a  \emph{finitary C-matrix}, or \emph{fC-matrix} for short, if the following are satisfied:
\bit
\item
The entries $M_{qi}$ are drawn from $\{ \Delta,\muup,\mudown,\mu,R\}\cup\set{N}{\{\id_q\}\neq N\normal\S_q,\ 2\leq q\leq n}$.
\item
Rows $0$ and $1$ of $M$ must be of one of the finitary row types \ref{fRT1}--\ref{fRT4} shown in Table \ref{tab:fRT}.
\item
Every row $q\geq 2$ must be of  type \ref{fRT5} from Table \ref{tab:fRT}.
\item
An $N$-symbol cannot be immediately above $\Delta$, $\muup$, $\mudown$ or another $N$-symbol.
\item
Every entry equal to $R$ in row $q\geq 1$ must be directly above an $R$ entry from row~$q-1$.  
\eit
\end{defn}

\newcounter{fRT}\renewcommand{\thefRT}{fRT\arabic{fRT}}

\begin{table}[ht]
\begin{center}
\footnotesize
\begin{tabular}{|l|l|l|}
\hline
\bf Type & \bf Row(s) & \bf Conditions\\
\hline\hline
%fRT1
\bf\refstepcounter{fRT}\thefRT
&
$
\Cmatsetup
\begin{array}{r|c|c|c|c|c|c|c|}\hhline{~|-|-|-|-|-|-|}
\renewcommand{\arraystretch}{2}
\text{\scriptsize 1} &\cellcolor{delcol}\Delta &\cellcolor{delcol}\dots & \cellcolor{delcol}\Delta & \cellcolor{Rcol}R&\cellcolor{Rcol}\dots& \cellcolor{Rcol} R\\ \hhline{~|-|-|-|-|-|-|}
\text{\scriptsize 0}&\cellcolor{delcol}\Delta &\cellcolor{delcol}\dots &\cellcolor{delcol}\Delta & \cellcolor{Rcol}R&\cellcolor{Rcol}\dots& \cellcolor{Rcol} R\\ \hhline{~|-|-|-|-|-|-|}
\multicolumn{1}{c}{}&\multicolumn{1}{c}{}&\multicolumn{1}{c}{}&\multicolumn{1}{c}{}&\multicolumn{1}{c}{\text{\scriptsize $k$}}&\multicolumn{1}{c}{}&\multicolumn{1}{c}{}
\end{array}\rule[-5mm]{0mm}{15mm}
$
& $0\leq k\leq d+1$
\\
\hline
%fRT2
\bf\refstepcounter{fRT}\thefRT
&
$
\Cmatsetup
\begin{array}{r|c|c|c|c|c|c|c|c|c|c|c|c|}\hhline{~|-|-|-|-|-|-|-|-|-|-|-|}
\renewcommand{\arraystretch}{2}
\text{\scriptsize 1} &\cellcolor{delcol}\Delta &\cellcolor{delcol}\dots & \cellcolor{delcol}\Delta & 
\cellcolor{excepcol}\zeta & \cellcolor{mucol} \mu & \cellcolor{mucol} \dots &\cellcolor{mucol} \mu &
\cellcolor{mucol} \mu &
\cellcolor{Rcol} R&\cellcolor{Rcol}\dots& \cellcolor{Rcol} R\\ \hhline{~|-|-|-|-|-|-|-|-|-|-|-|}
\text{\scriptsize 0}&\cellcolor{delcol}\Delta &\cellcolor{delcol}\dots &\cellcolor{delcol}\Delta & 
\cellcolor{mucol}\mu & \cellcolor{mucol} \mu & \cellcolor{mucol} \dots &\cellcolor{mucol} \mu &
\cellcolor{Rcol} R &\cellcolor{Rcol} R&\cellcolor{Rcol}\dots& \cellcolor{Rcol} R\\ \hhline{~|-|-|-|-|-|-|-|-|-|-|-|}
\multicolumn{1}{c}{}&\multicolumn{1}{c}{}&\multicolumn{1}{c}{}&\multicolumn{1}{c}{}&\multicolumn{1}{c}{\text{\scriptsize $i$}}&\multicolumn{1}{c}{}&\multicolumn{1}{c}{}&\multicolumn{1}{c}{}&\multicolumn{1}{c}{\text{\scriptsize $k$}}&\multicolumn{1}{c}{}&\multicolumn{1}{c}{}
\end{array}
\rule[-5mm]{0mm}{15mm}
$
& \makecell[l]{$0\leq i< k\leq d,$\\ $\zeta\in \{\mu,\muup,\mudown,\Delta\} $}
 \\
\hline
%fRT3
\bf\refstepcounter{fRT}\thefRT
&
$
\Cmatsetup
\begin{array}{r|c|c|c|c|c|c|c|c|c|c|}\hhline{~|-|-|-|-|-|-|-|-|-|-|}
\renewcommand{\arraystretch}{2}
\text{\scriptsize 1} &\cellcolor{delcol}\Delta &\cellcolor{delcol}\dots & \cellcolor{delcol}\Delta & 
\cellcolor{delcol}\Delta & \cellcolor{delcol}\dots & \cellcolor{delcol}\Delta &
\cellcolor{excepcol} \zeta & \cellcolor{Rcol} R& \cellcolor{Rcol}\dots&\cellcolor{Rcol} R
\\ \hhline{~|-|-|-|-|-|-|-|-|-|-|}
\text{\scriptsize 0}&\cellcolor{delcol}\Delta &\cellcolor{delcol}\dots &\cellcolor{delcol}\Delta & 
\cellcolor{Rcol} R &\cellcolor{Rcol}\dots&\cellcolor{Rcol} R&\cellcolor{Rcol} R&\cellcolor{Rcol} R& \cellcolor{Rcol}\dots&\cellcolor{Rcol} R
\\ \hhline{~|-|-|-|-|-|-|-|-|-|-|}
\multicolumn{1}{c}{}&\multicolumn{1}{c}{}&\multicolumn{1}{c}{}&\multicolumn{1}{c}{}&\multicolumn{1}{c}{\text{\scriptsize $k$}}&\multicolumn{1}{c}{}&\multicolumn{1}{c}{}&\multicolumn{1}{c}{}&\multicolumn{1}{c}{\text{\scriptsize $l$}}&\multicolumn{1}{c}{}&\multicolumn{1}{c}{}
\end{array}
\rule[-5mm]{0mm}{15mm}
$
& 
\makecell[l]{$0\leq k<l\leq d+1$,\\  $\zeta\in \{\mu,\muup,\mudown,\Delta\}$}
 \\
\hline
%fRT4
\bf\refstepcounter{fRT}\thefRT
&
$
\Cmatsetup
\begin{array}{r|c|c|c|c|c|c|c|c|c|c|c|}\hhline{~|-|-|-|-|-|-|-|-|-|-|-|}
\renewcommand{\arraystretch}{2}
\text{\scriptsize 1} &\cellcolor{delcol}\Delta &\cellcolor{delcol}\dots & \cellcolor{delcol}\Delta  & \cellcolor{delcol}\Delta &
\cellcolor{delcol}\Delta & \cellcolor{delcol}\dots & \cellcolor{delcol}\Delta &
\cellcolor{mucol} \mu & \cellcolor{Rcol} R& \cellcolor{Rcol}\dots & \cellcolor{Rcol} R 
\\ \hhline{~|-|-|-|-|-|-|-|-|-|-|-|}
\text{\scriptsize 0}&\cellcolor{delcol}\Delta &\cellcolor{delcol}\dots &\cellcolor{delcol}\Delta  
&\cellcolor{mucol}\mu&
\cellcolor{Rcol} R &\cellcolor{Rcol}\dots&\cellcolor{Rcol} R&\cellcolor{Rcol} R&\cellcolor{Rcol} R& \cellcolor{Rcol}\dots & \cellcolor{Rcol} R
\\ \hhline{~|-|-|-|-|-|-|-|-|-|-|-|}
\multicolumn{1}{c}{}&\multicolumn{1}{c}{}&\multicolumn{1}{c}{}&\multicolumn{1}{c}{}&\multicolumn{1}{c}{}&\multicolumn{1}{c}{\text{\scriptsize $k$}}&\multicolumn{1}{c}{}&\multicolumn{1}{c}{}&\multicolumn{1}{c}{}&\multicolumn{1}{c}{\text{\scriptsize $l$}}
\end{array}
\rule[-5mm]{0mm}{15mm}
$
& $0<k<l-1\leq d$ \\
\hline
%fRT5
\bf\refstepcounter{fRT}\thefRT
&
$
\Cmatsetup
\begin{array}{r|c|c|c|c|c|c|c|c|c|c|}\hhline{~|-|-|-|-|-|-|-|-|-|-|}
\renewcommand{\arraystretch}{2}
\text{\scriptsize $q$} &\cellcolor{delcol}\Delta &\cellcolor{delcol}\dots & \cellcolor{delcol}\Delta  & 
\cellcolor{Ncol} N_i &\cellcolor{Ncol} N_{i+1} & \cellcolor{Ncol}\dots & \cellcolor{Ncol} N_{k-1} &
\cellcolor{Rcol} R & \cellcolor{Rcol}\dots & \cellcolor{Rcol} R
\\ \hhline{~|-|-|-|-|-|-|-|-|-|-|}
\multicolumn{1}{c}{}&\multicolumn{1}{c}{}&\multicolumn{1}{c}{}&\multicolumn{1}{c}{}&\multicolumn{1}{c}{\text{\scriptsize $i$}}&\multicolumn{1}{c}{}&\multicolumn{1}{c}{}&\multicolumn{1}{c}{}&\multicolumn{1}{c}{\text{\scriptsize $k$}}&\multicolumn{1}{c}{}&\multicolumn{1}{c}{}
\end{array}
\rule[-5mm]{0mm}{15mm}
$
& \makecell[l]{$q\geq2$,\\
$0\leq i \leq k\leq d+1$,\\
$\{\id_q\}\not=N_i\leq \dots \leq N_{k-1}$,\\
$N_i,\dots,N_{k-1}\unlhd \S_q$ }\\
\hline
\end{tabular}
\normalsize
\caption{The specification of finitary row types in an fC-matrix.}
\label{tab:fRT}
\end{center}
\end{table}

\setcounter{fRT}{0}
\refstepcounter{fRT}\label{fRT1}
\refstepcounter{fRT}\label{fRT2}
\refstepcounter{fRT}\label{fRT3}
\refstepcounter{fRT}\label{fRT4}
\refstepcounter{fRT}\label{fRT5}

\begin{defn}[\bf Congruence corresponding to a finitary C-matrix]
\label{de:fCcong}
The congruence associated with a finitary C-matrix $M$ is the relation $\cg(M)$ on $\Ptw{n,d}$ consisting of all pairs ${((i,\alpha),(j,\beta))\in \Ptw{n,d}\times \Ptw{n,d}}$ such that one of the following holds, writing $q=\rank\alpha$ and ${r=\rank\beta}$:
\begin{enumerate}[label=\textsf{(fC\arabic*)}, widest=(C8), leftmargin=10mm]
\item
$M_{qi}=M_{rj}=\De$, $i=j$ and $\alpha=\beta$;
\item
$M_{qi}=M_{rj}=R$;
\item
$M_{qi}=M_{rj}=N$, $i=j$, $\alpha\rH\beta$ and $\pd(\alpha,\beta)\in N$;
\item
$M_{qi}=M_{rj}=\mudown$, $\widehat{\alpha}=\widehat{\beta}$ and $\alpha\rL\beta$;
\item
$M_{qi}=M_{rj}=\muup$, $\widehat{\alpha}=\widehat{\beta}$ and $\alpha\rR\beta$;
\item
$M_{qi}=M_{rj}=\mu$, $\widehat{\alpha}=\widehat{\beta}$, and either $(q,i)=(r,j)$ or $i-j = \min_q(M)-\min_r(M)$;
\end{enumerate}
as well as the pairs:
\begin{enumerate}[label=\textsf{(fC\arabic*)}, widest=(C8), leftmargin=10mm]
\addtocounter{enumi}{6}
\item
$((i,\alpha),\zero),(\zero,(i,\alpha))$ with $M_{qi}=R$;
\item
$(\zero,\zero)$.
\end{enumerate}
\end{defn}

Here is the classification of congruences:

\begin{thm}[{\cite[Theorem 7.3]{ERtwisted1}}]
\label{thm:finmain}
For $n\geq1$ and $d\geq0$, the congruences on the $d$-twisted partition monoid $\Ptw{n,d}$ are precisely $\fcg(M)$, where $M$ is any finitary C-matrix. \epfres
\end{thm}

The analogue of Theorem \ref{thm:comparisons}, describing the inclusion ordering for congruences of $\Ptw{n,d}$, also simplifies considerably for $\Ptw{n,d}$.  In particular, only part~\ref{it:comp1} applies as there are no exceptional congruences.  
The main complication remaining is caused by the matching $\mu$s in rows~$0$ and~$1$.
To handle these we use the $\mumin_{0}(M)$ and $\mumin_{1}(M)$ notation introduced before Theorem~\ref{thm:comparisons}, which here applies to fC-matrices of types \ref{fRT2} and \ref{fRT4}.
Also, in the absence of C-chains, we write $\min_q(M)$, for $q=0,1$, to denote the smallest $i\in\{0,\dots,d\}$ such that $M_{qi}=R$ if it exists, and $\min_q(M):=d+1$ otherwise.
We have $\mumin_{1}(M)-\mumin_{0}(M)=\min_1(M)-\min_0(M)$ in types \ref{fRT2} and \ref{fRT4}.

\begin{thm}[{\cite[Theorem 7.4]{ERtwisted1}}]
\label{thm:fincomp}
Let $n\geq 1$ and $d\geq 0$, and let 
$M^1$ and $M^2$ be any two fC-matrices for $\Ptw{n,d}$. Then
$\fcg(M^1)\sub\fcg(M^2)$ if and only if both of the following hold:
\begin{thmsubenumerate}
\item
\label{it:fc1}
$M^1\leqc M^2$;
\item
\label{it:fc2}
If $M^1$ has type \ref{fRT2} or \ref{fRT4}, then at least one of the following hold:
\begin{enumerate}[label=\textup{(b\arabic*)},leftmargin=10mm]
\item \label{fb1} $\min_0(M^2)\leq\mumin_{0}(M^1)$ and $\min_1(M^2)\leq\mumin_{1}(M^1)$, or
\item \label{fb2} $M^2$ also has type \ref{fRT2} or \ref{fRT4}, and $\min_1(M^2)-\min_0(M^2)=\min_1(M^1)-\min_0(M^1)$.  \epfres
\end{enumerate}
\end{thmsubenumerate}
\end{thm}

\begin{rem}
\label{rem:CongPn0}
The congruences of
the 0-twisted partition monoid $\Ptw{n,0}$ are particularly easy to understand, and these will play an important role in subsequent sections when looking at $\Ptw{n,d}$ for arbitrary $d$.
When $d=0$, the fC-matrices are simply fC-\emph{columns}, and they come in two families (with ${\{\id_q\}\neq N\normal\S_q}$ in row $q\geq2$ in the second), and four individual `sporadic' matrices:
\begin{equation}\label{eq:fC0}
{
\Cmatsetup
\begin{array}{|c|}\hline
\renewcommand{\arraystretch}{2}
\cellcolor{delcol}\Delta \\ \hline
\cellcolor{delcol}\vvdots \\ \hline
\cellcolor{delcol}\Delta \\ \hline
\cellcolor{delcol}\Delta \\ \hline
\cellcolor{Rcol}R \\ \hline
\cellcolor{Rcol}\vvdots \\ \hline
\cellcolor{Rcol}R \\ \hline
\end{array}
\ \COMMA 
\begin{array}{c|c|}\hhline{~|-}
 \renewcommand{\arraystretch}{2}
& \cellcolor{delcol}\Delta \\ \hhline{~|-}
& \cellcolor{delcol}\vvdots \\ \hhline{~|-}
& \cellcolor{delcol}\Delta \\ \hhline{~|-}
\text{\scriptsize $q$} & \cellcolor{Ncol}N\\ \hhline{~|-}
& \cellcolor{Rcol}R \\ \hhline{~|-}
& \cellcolor{Rcol}\vvdots \\ \hhline{~|-}
& \cellcolor{Rcol}R \\ \hhline{~|-}
\end{array}
\qquad\AND\qquad
\begin{array}{|c|}\hline
\renewcommand{\arraystretch}{2}
\cellcolor{delcol}\Delta \\ \hline
\cellcolor{delcol}\vvdots \\ \hline
\cellcolor{delcol}\Delta \\ \hline
\cellcolor{delcol}\Delta \\ \hline
\cellcolor{excepcol}\muup \\ \hline
\cellcolor{Rcol}R \\ \hline
\end{array}
\ \COMMA
\begin{array}{|c|}\hline
\renewcommand{\arraystretch}{2}
\cellcolor{delcol}\Delta \\ \hline
\cellcolor{delcol}\vvdots \\ \hline
\cellcolor{delcol}\Delta \\ \hline
\cellcolor{delcol}\Delta \\ \hline
\cellcolor{excepcol}\mudown \\ \hline
\cellcolor{Rcol}R \\ \hline
\end{array}
\ \COMMA
\begin{array}{|c|}\hline
\renewcommand{\arraystretch}{2}
\cellcolor{delcol}\Delta \\ \hline
\cellcolor{delcol}\vvdots \\ \hline
\cellcolor{delcol}\Delta \\ \hline
\cellcolor{delcol}\Delta \\ \hline
\cellcolor{excepcol}\mu \\ \hline
\cellcolor{Rcol}R \\ \hline
\end{array}
\ \COMMA
\begin{array}{|c|}\hline
\renewcommand{\arraystretch}{2}
\cellcolor{delcol}\Delta \\ \hline
\cellcolor{delcol}\vvdots \\ \hline
\cellcolor{delcol}\Delta \\ \hline
\cellcolor{Ncol}\S_2 \\ \hline
\cellcolor{excepcol}\mu \\ \hline
\cellcolor{Rcol}R \\ \hline
\end{array}
\ .
}
\end{equation}
Denoting the congruences associated to the first two types by $R_q$ and $R_N$, and reusing existing symbols 
 to denote the four sporadic congruences by $\muup$, $\mudown$, $\mu$ and $\mu_{\S_2}$, the congruence lattice $\Cong(\Ptw{n,0})$ is shown in Figure \ref{fig:CongPn0}.
\end{rem}

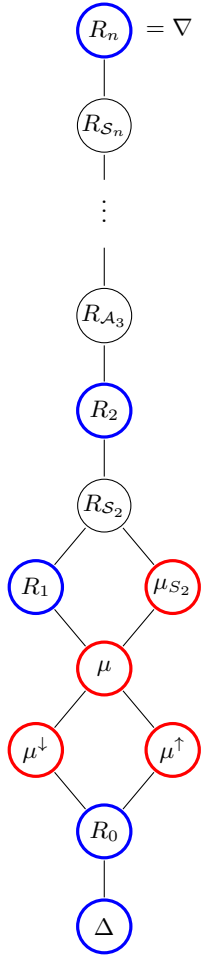
\begin{figure}[ht]
\begin{center}
\begin{tikzpicture}[scale=0.9]
\begin{scope}[minimum size=7mm,inner sep=0.5pt, outer sep=1pt]
\node (D) at (0,0) [draw=blue,fill=white,circle,line width=1.2pt] {\footnotesize $\De$};
\node (R0) at (0,1.4) [draw=blue,fill=white,circle,line width=1.2pt] {\footnotesize $R_0$};
\node (mud) at (-1,2.6) [draw=red,fill=white,circle,line width=1.2pt] {\footnotesize $\mudown$};
\node (muu) at (1,2.6) [draw=red,fill=white,circle,line width=1.2pt] {\footnotesize $\muup$};
\node (mu) at (0,3.8) [draw=red,fill=white,circle,line width=1.2pt] {\footnotesize $\mu$};
\node (R1) at (-1,5) [draw=blue,fill=white,circle,line width=1.2pt] {\footnotesize $R_1$};
\node (muS2) at (1,5) [draw=red,fill=white,circle,line width=1.2pt] {\footnotesize $\mu_{S_2}$};
\node (S2) at (0,6.2) [draw,circle]  {\footnotesize $R_{\S_2}$};
\node (R2) at (0,7.6) [draw=blue,fill=white,circle,line width=1.2pt]{\footnotesize$R_2$};
\node (A3) at (0,9) [draw,circle] {\footnotesize $R_{\A_3}$};
\node (Sn) at (0,11.8)[draw,circle]  {\footnotesize $R_{\S_n}$};
\node (Rn) at (0,13.2)  [draw=blue,fill=white,circle,line width=1.2pt]  {\footnotesize $R_n$};
\node (dots) at (0,10.65) {\footnotesize $\vdots$};
\node () at (.95,13.25) {\footnotesize $=\nab$};
\end{scope}
\draw (D)--(R0) (R0)--(mud) (R0)--(muu) (mud)--(mu) (muu)--(mu) (mu)--(R1) (mu)--(muS2) (R1)--(S2) (muS2)--(S2)
(S2)--(R2) (R2)--(A3) (A3)--(0,10) (0,11)--(Sn)--(Rn);
\end{tikzpicture}
\caption{The Hasse diagram of $\Cong(\Ptw{n,0})$; Rees congruences are indicated in blue outline, `sporadic' congruences in red, and we abbreviate $\De=\De_{\Ptw{n,0}}$ and $\nab=\nab_{\Ptw{n,0}}$.  }
\label{fig:CongPn0}
\end{center}
\end{figure}

\section{Properties of the congruence lattice \boldmath{$\textsf{\textbf{Cong}}({\Ptw{n,d}})$}}
\label{sec:CongPnd}

Let us compare the properties of the lattice $\Cong(\Ptw{n,d})$ with those of $\Cong(\Ptw{n})$ established
in Section \ref{sec:props}.
Since the monoids $\Ptw{n,d}$ are finite, so are their congruence lattices;
we shall compute their sizes exactly in Section \ref{sec:enumeration}.
So, trivially, $\Cong(\Ptw{n,d})$ has no infinite ascending/descending/anti-chains; also, every element has finitely many covers, and is in turn covered by finitely many elements.
The lattice continues to have a single co-atom, as in \eqref{eq:coatom}, and additionally it has a single atom, which is defined by the fC-matrix
\[
\Cmatsetup
\begin{array}{|c|c|c|c|c|}
\hline
\cellcolor{delcol}\De & \cellcolor{delcol}\De & \cellcolor{delcol}\cdots & \cellcolor{delcol}\De & \cellcolor{delcol}\De \\ \hline
\cellcolor{delcol}\vvdots & \cellcolor{delcol}\vvdots & \cellcolor{delcol}\vvdots & \cellcolor{delcol}\vvdots & \cellcolor{delcol}\vvdots \\ \hline
\cellcolor{delcol}\De & \cellcolor{delcol}\De & \cellcolor{delcol}\cdots & \cellcolor{delcol}\De & \cellcolor{delcol}\De \\ \hline
\cellcolor{delcol}\De & \cellcolor{delcol}\De & \cellcolor{delcol}\dots & \cellcolor{delcol}\De & \cellcolor{Rcol}R \\ \hline
\end{array}\ .
\]
In one marked contrast with $\Cong(\Ptw{n})$, this time we have the following:

\begin{thm}
\label{thm:distrib}
Let $n\geq2$.
\ben
\item \label{dist1} The lattice $\Cong(\Ptw{n,0})$ is distributive.  
\item \label{dist2} For $d>0$, the lattice $\Cong(\Ptw{n,d})$ is modular but not distributive.  
\een
\end{thm}

\begin{proof}
\ref{dist1}  This follows by direct inspection of Figure \ref{fig:CongPn0}.  

\ref{dist2}  To prove non-distributivity, we exhibit a five-element diamond sublattice of $\Cong(\Ptw{n,d})$ in Figure \ref{fig:5dia}.  To simplify the diagram, we only indicate the bottom-right $2\times2$ corner of each fC-matrix; all other entries are $\De$.  
Verification of meets and joins are straightforward with the help of Theorem \ref{thm:fincomp}.
To prove modularity we need to show that $\Cong(\Ptw{n,d})$ does not contain a sublattice isomorphic to the pentagon.
This is rather involved, occupies the remainder of this section, and is finally accomplished in Proposition \ref{pr:5gons}.
\end{proof}

\begin{figure}[ht]
\begin{center}
\begin{tikzpicture}[scale=.5]
\node (11) at (0,0) {$\Cmatsetup
\begin{array}{|c|c|}
\hline
\cellcolor{delcol}\De & \cellcolor{excepcol}\mu \\ \hline
\cellcolor{delcol}\De & \cellcolor{Rcol}R \\ \hline
\end{array}$};
\node (21) at (4.5,3) {$\Cmatsetup
\begin{array}{|c|c|}
\hline
\cellcolor{delcol}\De & \cellcolor{mucol}\mu \\ \hline
\cellcolor{mucol}\mu & \cellcolor{Rcol}R \\ \hline
\end{array}$};
\node (11') at (-4.5,3) {$\Cmatsetup
\begin{array}{|c|c|}
\hline
\cellcolor{delcol}\De & \cellcolor{excepcol}\mu \\ \hline
\cellcolor{Rcol}R & \cellcolor{Rcol}R \\ \hline
\end{array}$};
\node (12) at (0,3) {$\Cmatsetup
\begin{array}{|c|c|}
\hline
\cellcolor{delcol}\De & \cellcolor{Rcol}R \\ \hline
\cellcolor{delcol}\De & \cellcolor{Rcol}R \\ \hline
\end{array}$};
\node (22) at (0,6) {$\Cmatsetup
\begin{array}{|c|c|}
\hline
\cellcolor{delcol}\De & \cellcolor{Rcol}R \\ \hline
\cellcolor{Rcol}R & \cellcolor{Rcol}R \\ \hline
\end{array}$};
\draw (11)--(21)--(22)--(11')--(11)--(12)--(22);
\end{tikzpicture}
\caption{A five-element diamond sublattice of $\Cong(\Ptw{n,d})$, where $d>0$.  See the proof of Theorem~\ref{thm:distrib} for more details.}
\label{fig:5dia}
\end{center}
\end{figure}
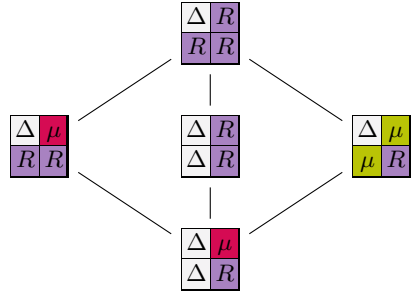

We begin our work towards Proposition \ref{pr:5gons} with some preliminaries that gauge the extent to which the inclusion ordering on $\Cong(\Ptw{n,d})$ differs
from the $\leqc$ ordering on fC-matrices.

\begin{defn}
\label{de:matched}
Let $M$ be an fC-matrix, and let $\sigma:=\cg(M)$.
We say that a pair of entries $M_{qi}$ and $M_{rj}$ of $M$ are \emph{matched} if
$(q,i)\neq(r,j)$, $M_{qi}=M_{rj}=\mu$ and $\sigma\cap (D_{qi}\times D_{rj})\neq\emptyset$.
\end{defn}

We note that if $M_{qi}=M_{rj}$ are matched then necessarily $\{q,r\}=\{0,1\}$.
We also note that every $\mu$ entry in row $0$ is matched, whereas those in row $1$ can be matched or unmatched. However, there can be at most one unmatched $\mu$ in $M$, corresponding to $\ze=\mu$ in types \ref{fRT2} and \ref{fRT4} (see Table \ref{tab:fRT}).

We now have the following obvious consequence of Definition \ref{de:fCcong}:

\begin{lemma}
\label{la:Dqirj}
Let $\sigma=\cg(M)\in \Cong(\Ptw{n,d})$, and let $D_{qi}$ and $D_{rj}$ be two distinct $\D$-classes.
Then $\sigma\cap (D_{qi}\times D_{rj})\neq\emptyset$ if and only if either
$M_{qi}=M_{rj}=R$ or else $M_{qi}=M_{rj}=\mu$ are matched. \qed
\end{lemma}

Next we consider the behaviour of the join and the intersection of two congruences restricted to a single $\D$-class.
We observe that since $\Ptw{n,d}$ has a zero element, the ideal class $I(\sigma)$ of any congruence $\sigma$ coincides with the 
class of $\zero$, and is non-empty in particular.

\begin{lemma}
\label{la:joinneq}
Let $\sigma^1,\sigma^2\in\Cong(\Ptw{n,d})$, and let $q\in \bnz$ and $i\in\N$.
If ${(\sigma^1\vee\sigma^2)\restr_{D_{qi}}\neq \sigma^1\restr_{D_{qi}} \vee \sigma^2\restr_{D_{qi}}}$
then $\sigma^t\cap (D_{qi}\times D_{rj})\neq \emptyset$ for some $t=1,2$ and some $(r,j)\neq (q,i)$.
\end{lemma}

\begin{proof}
Consider an arbitrary $(\ba,\bb)\in (\sigma^1\vee\sigma^2)\restr_{D_{qi}}\setminus ( \sigma^1\restr_{D_{qi}} \vee \sigma^2\restr_{D_{qi}})$.
This means that there exists a sequence
$\ba=\ba_1,\ba_2,\dots,\ba_m=\bb$ such that the successive pairs belong to $\sigma^1\cup\sigma^2$,
and not all the terms belong to $D_{qi}$.
Let $(\ba_l,\ba_{l+1})\in\sigma^t$ be any pair where $\ba_l\in D_{qi}$ and $\ba_{l+1}\not\in D_{qi}$.
If $\ba_{l+1}\in D_{rj}$ for some $(r,j)$ we are finished. Otherwise $\ba_{l+1}=\zero$.
But then $D_{qi}\sub I(\sigma^t)$, which implies 
$\sigma^t\restr_{D_{qi}}=\nabla_{D_{qi}}$, and therefore $(\sigma^1\vee\sigma^2)\restr_{D_{qi}}=\nabla_{D_{qi}}= \sigma^1\restr_{D_{qi}} \vee \sigma^2\restr_{D_{qi}}$, a contradiction.
\end{proof}

\begin{lemma}
\label{la:jirest}
Let $\sigma^1,\sigma^2\in\Cong(\Ptw{n,d})$, and $q\in\bnz$, $i\in \N$. Then:
\begin{thmenumerate}
\item
\label{it:jir1}
$(\sigma^1\cap \sigma^2)\restr_{D_{qi}}=\sigma^1\restr_{D_{qi}}\cap \sigma^2\restr_{D_{qi}}$;
\item
\label{it:jir2}
$(\sigma^1\vee \sigma^2)\restr_{D_{qi}}=\sigma^1\restr_{D_{qi}}\vee \sigma^2\restr_{D_{qi}}$,
with the only possible exceptions when $q\in\{0,1\}$ and $M_{qi}^1,M_{qi}^2\in\{\De,\mudown,\muup,\mu\}$.
\end{thmenumerate}
\end{lemma}

\begin{proof}
\ref{it:jir1} is obvious. \ref{it:jir2} follows from Lemmas \ref{la:Dqirj} and \ref{la:joinneq}, and the observation that if
either $M_{qi}^1$ or $M_{qi}^2$ equals $R$ then 
$(\sigma^1\vee \sigma^2)\restr_{D_{qi}}=\nabla_{D_{qi}}=\sigma^1\restr_{D_{qi}}\vee \sigma^2\restr_{D_{qi}}$.
\end{proof}

Next we move on to the ideal class $I(\si)$.

\begin{lemma}
\label{la:jiid}
For any $\sigma^1,\sigma^2\in\Cong(\Ptw{n,d})$ the following hold:
\begin{thmenumerate}
\item
\label{it:jiid1}
$I(\sigma^1\cap\sigma^2)=I(\sigma^1)\cap I(\sigma^2)$;
\item
\label{it:jiid2}
If $q\geq 2$, then $D_{qi}\subseteq I(\sigma^1\vee\sigma^2)$ if and only if $D_{qi}\subseteq I(\sigma^t)$ for some $t=1,2$.
\end{thmenumerate}
\end{lemma}

\begin{proof}
\ref{it:jiid1} is obvious.
For \ref{it:jiid2} notice that for $q\geq 2$ we either have $D_{qi}\subseteq I(\sigma^t)$, or else
there are no $\sigma^t$-relationships between $D_{qi}$ and $\Ptw{n,d}\setminus D_{qi}$.
\end{proof}

We bring these results together into the following description of the fC-matrices for the intersection and join of two congruences. The meets and joins between the matrix entries are computed under the ordering $\leqc$ as
depicted in Figure \ref{fig:M}.

\begin{lemma}
\label{la:jifC}
Let $\sigma^1,\sigma^2\in\Cong(\Ptw{n,d})$ with $\sigma^t=\cg(M^t)$, and furthermore let
$\sigma^1\cap\sigma^2=\cg(M^\cap)$, $\sigma^1\vee\sigma^2=\cg(M^\vee)$.
For $q\in\bnz$ and $i\in\N$, the following hold:
\begin{thmenumerate}
\item
\label{it:jifC1}
$M_{qi}^\cap=M_{qi}^1\wedge M_{qi}^2$, except when $q=0$, $M_{0i}^1=M_{0i}^2=\mu$, but the matching entries in row $1$ are in different positions in $M^1$ and $M^2$, in which case $M_{qi}^\cap=\Delta$.
\item
\label{it:jifC2}
$M_{qi}^\vee=M_{qi}^1\vee M_{qi}^2$, with possible exceptions for $q\in\{0,1\}$, $M_{qi}^1,M_{qi}^2\in \{\Delta,\muup,\mudown,\mu\}$, and at least one of $M_{qi}^1,M_{qi}^2$ equals $\mu$, in which case
we may have $M_{qi}^\vee=R$.
\end{thmenumerate}
\end{lemma}

\begin{proof}
\ref{it:jifC1}
By Lemma \ref{la:jirest}\ref{it:jir1}, the restriction $(\sigma^1\cap \sigma^2)\restr_{D_{qi}}$ is uniquely determined
by the $\sigma^t\restr_{D_{qi}}$.
Therefore, we can only have $M_{qi}^\cap\neq M_{qi}^1\wedge M_{qi}^2$ when $M_{qi}^\cap$ is not uniquely determined by ${(\sigma^1\cap \sigma^2)\restr_{D_{qi}}}$.
As in Remark \ref{rem:amb}, the latter is only the case when
$q=0$ and $(\sigma^1\cap \sigma^2)\restr_{D_{0i}}=\Delta_{D_{0i}}$, or when $q=n$ and 
$(\sigma^1\cap \sigma^2)\restr_{D_{ni}}=\nabla_{D_{ni}}$.
The first alternative rapidly leads to the assertion, by eliminating the possibility $M_{qi}^t=\Delta$ for some $t=1,2$.
For the second, suppose $n=q$ and $(\sigma^1\cap \sigma^2)\restr_{D_{qi}}=\nabla_{D_{qi}}$.
Then
$M_{qi}^1,M_{qi}^2,M_{qi}^\cap\in\{S_n,R\}$, and, using Lemma \ref{la:jiid}\ref{it:jiid1},
\[
M_{qi}^\cap=R \iff D_{qi}\subseteq I(\sigma^1\cap\sigma^2)=I(\sigma^1)\cap I(\sigma^2)\iff
M_{qi}^1=M_{qi}^2=R,
\]
which implies $M_{qi}^\cap=M_{qi}^1\wedge M_{qi}^2$.

\ref{it:jifC2}
Reasoning as in \ref{it:jifC1}, we see that the only exceptions to $M_{qi}^\vee=M_{qi}^1\vee M_{qi}^2$ may arise if:
\bit
\item
$(\sigma^1\vee \sigma^2)\restr_{D_{qi}}\neq \sigma^1\restr_{D_{qi}}\vee \sigma^2\restr_{D_{qi}}$; or
\item
$M_{qi}^\vee$ is not uniquely determined by $(\sigma^1\vee \sigma^2)\restr_{D_{qi}}$.
\eit
We begin with the first option.  Here Lemma \ref{la:jirest}\ref{it:jir2} gives $q\in\{0,1\}$ and $M_{qi}^1,M_{qi}^2\in \{\Delta,\muup,\mudown,\mu\}$.  If in fact $M_{qi}^1,M_{qi}^2\in \{\Delta,\muup,\mudown\}$, then would be no $\si^t$-relationships between $D_{qi}$ and $\Ptw{n,d}\sm D_{qi}$ ($t=1,2$), and hence $(\sigma^1\vee \sigma^2)\restr_{D_{qi}}= \sigma^1\restr_{D_{qi}}\vee \sigma^2\restr_{D_{qi}}$, a contradiction.  Thus, at least one of $M_{qi}^1,M_{qi}^2$ equals $\mu$.  But then if $M_{qi}^\vee>M_{qi}^1\vee M_{qi}^2$, we can only have $M_{qi}^\vee=R$ (see Figure \ref{fig:M}, remembering that $\lam$ and $\rho$ do not appear in fC-matrices).

For the second option, as in \ref{it:jifC1}, we have
$q=0$ and $(\sigma^1\vee \sigma^2)\restr_{D_{0i}}=\Delta_{D_{0i}}$, or $q=n$ and 
$(\sigma^1\vee \sigma^2)\restr_{D_{ni}}=\nabla_{D_{ni}}$.
In the first of these alternatives we have:
\begin{align*}
M_{0i}^\vee=\mu &\iff 
(\sigma^1\vee\sigma^2)\cap (D_{0i}\times D_{1j})\neq \emptyset
\text{ for some } j\in\N
\\
&\iff \sigma^t\cap (D_{0i}\times D_{1j})\neq \emptyset
\text{ for some } t=1,2,\  j\in\N
\\
&\iff M_{0i}^t=\mu \text{ for some } t=1,2.
\end{align*}
And in the second alternative, using Lemma \ref{la:jiid}\ref{it:jiid2}, we have
\begin{align*}
M_{ni}^\vee=R &\iff D_{ni}\subseteq I(\sigma^1\vee\sigma^2)
\\
&\iff
D_{ni}\subseteq I(\sigma^t) \text{ for some } t=1,2
\\
&\iff
M_{ni}^t=R \text{ for some } t=1,2,
\end{align*}
and as in \ref{it:jifC1} this leads to $M_{ni}^\vee=M_{ni}^1\vee M_{ni}^2$, completing the proof.
\end{proof}

We are now ready to demonstrate the non-existence of pentagons in $\Cong(\Ptw{n,d})$, which is the final step in the proof of Theorem \ref{thm:distrib}.

\begin{prop}
\label{pr:5gons}
For $n\geq1$ and $d\geq0$, the lattice $\Cong(\Ptw{n,d})$ does not contain a sublattice isomorphic to the pentagon.  
\end{prop}

\begin{proof}
Seeking a contradiction, suppose $\Cong(\Ptw{n,d})$ does in fact contain the following five-element sublattice, with (non)inclusions, meets and joins as indicated:
\[
\begin{tikzpicture}[scale=.7]
\node (cap) at (0,0) {$\si^\cap=\cg(M^\cap)$};
\node (1) at (-3.5,3) {$\si^1=\cg(M^1)$};
\node (vee) at (0,6) {$\si^\vee=\cg(M^\vee)$};
\node (2) at (3.5,2) {$\si^2=\cg(M^2)$};
\node (3) at (3.5,4) {$\si^3=\cg(M^3)$};
\draw (cap)--(2)--(3)--(vee)--(1)--(cap);
\end{tikzpicture}
\]

First we claim that there exists $(\ba,\bb)\in\sigma^3\setminus\sigma^2$ with $\ba,\bb\not=\zero$.  To see this, let $(\bc,\bd)\in\sigma^3\setminus\sigma^2$ be arbitrary.  Certainly $\bc$ and $\bd$ are not both $\zero$, say $\bc\not=\zero$; take $\ba:=\bc$.  If $\bd\not=\zero$ we take $\bb:=\bd$, so suppose instead that $\bd=\zero$.  Since $\sigma^2\neq\Delta_{\Ptw{n,d}}$, we may take any $\bb\in I(\si^2)\sm\{\zero\}$, with the required conditions easily checked.

With the claim established, for the rest of the proof we fix some $(\ba,\bb)\in\sigma^3\setminus\sigma^2$ with $\ba\in D_{qi}$ and $\bb\in D_{rj}$.
Notice immediately from $(\ba,\bb)\not\in\sigma^1\cap\sigma^2=\sigma^1\cap\sigma^3$,
and from $\sigma^3\subseteq\sigma^1\vee\sigma^3=\sigma^1\vee\sigma^2$ that
\begin{equation}
\label{eq:5g1}
(\ba,\bb)\not \in \sigma^1 \AND (\ba,\bb)\in\sigma^1\vee\sigma^2.
\end{equation}

From here the proof will proceed in two stages: first we will consider the situation where $(\ba,\bb)\in\D$;
and then afterwards we consider $(\ba,\bb)\not\in\D$, with the additional assumption
that the restrictions of $\sigma^2$ and $\sigma^3$ to every $\D$-class coincide.

\setcounter{stageco}{0}

\stage \label{st:5g1}
$\ba,\bb\in D_{qi}(=D_{qj})$.
From $(\ba,\bb)\in\si^3\restr_{D_{qi}}\sm\si^2\restr_{D_{qi}}$, we obtain
\begin{equation}\label{eq:M>M}
M_{qi}^3>M_{qi}^2.
\end{equation}
We claim that
\begin{equation}
\label{eq:5g3}
M_{qi}^\cap\neq M_{qi}^1\wedge M_{qi}^t \qquad\text{or}\qquad
M_{qi}^\vee\neq M_{qi}^1\vee M_{qi}^t\qquad \text{for some } t=2,3.
\end{equation}
Indeed, if not, then $M_{qi}^1,M_{qi}^2,M_{qi}^3,M_{qi}^\cap,M_{qi}^\vee$ would  form a homomorphic image of the the pentagon.
Furthermore, $M_{qi}^2\lessc M_{qi}^3$, and it quickly follows that in fact these entries form a pentagon.
This is a contradiction, because the lattice of allowable entries in any fixed position in an fC-matrix has no pentagons
(see Figure \ref{fig:M}), and \eqref{eq:5g3} is proved.

Now, use Lemma \ref{la:jifC} to observe that
\begin{align*}
&M_{qi}^\cap\neq M_{qi}^1\wedge M_{qi}^2\implies M_{qi}^\cap\lessc M_{qi}^1\wedge M_{qi}^2\leqc M_{qi}^1\wedge M_{qi}^3,
\\[2mm]
\AND&M_{qi}^\vee\neq M_{qi}^1\vee M_{qi}^3\implies M_{qi}^\vee\grc M_{qi}^1\vee M_{qi}^3\geqc M_{qi}^1\wedge M_{qi}^2.
\end{align*}
Thus \eqref{eq:5g3} is equivalent to
\[
M_{qi}^\cap\lessc M_{qi}^1\wedge M_{qi}^3 \qquad\text{or}\qquad
M_{qi}^\vee\grc M_{qi}^1\vee M_{qi}^2.
\]

Consider the first of these alternatives, i.e.~$M_{qi}^\cap\lessc M_{qi}^1\wedge M_{qi}^3$.
By Lemma \ref{la:jifC}\ref{it:jifC1} this can only happen when
$q=0$ and $M_{0i}^1=M_{0i}^3=\mu$.
But this implies $\sigma^3\restr_{D_{0i}}=\Delta_{D_{0i}}$, in contradiction with 
$(\ba,\bb)\in\sigma^3\setminus\sigma^2\subseteq \sigma^3\setminus\Delta_{\Ptw{n,d}}$.
Therefore we conclude that
\begin{equation}
\label{eq:5g3b}
M_{qi}^\vee\grc M_{qi}^1\vee M_{qi}^2.
\end{equation}
Lemma \ref{la:jifC}\ref{it:jifC2} gives us all the situations in which \eqref{eq:5g3b} may happen.
For the purposes of this proof we organise the cases as follows:
\bit
\item
$q=1$, $M_{1i}^1=M_{1i}^2=\mu$ and both are matched;
\item
$q=1$, $M_{1i}^1$ is one of $\Delta,\muup,\mudown$ or  an unmatched $\mu$,
and $M_{1i}^2=\mu$;
\item
$q=1$, $M_{1i}^1=\mu$, and $M_{1i}^2$ is one of $\Delta,\muup,\mudown$ or an unmatched $\mu$;
\item
$q=0$, $M_{0i}^1\in\{\Delta,\mu\}$ and $M_{0i}^2=\mu$;
\item
$q=0$, $M_{0i}^1=\mu$ and $M_{0i}^2=\Delta$.
\eit
Note that in the first three cases, if $M_{1i}^2=\mu$ then $M_{1i}^3=R$ by \eqref{eq:M>M}. Likewise, in the last two cases, $M_{0i}^3=R$ always (remembering that $(\ba,\bb)\in\si^3\restr_{D_{qi}}\sm\De_{D_{qi}}$ for the fifth).  In all five cases, Lemma \ref{la:jifC}\ref{it:jifC2} gives $M_{qi}^\vee=R$.
Throughout what follows we will make use of the following relation:
\[
\widetilde{\mu}:=\bigset{(l,\gamma),(m,\delta))}{ \gamma,\delta\in I_1,\ \widehat{\gamma}=\widehat{\delta},\ l,m\in\bdz}.
\]
 
\setcounter{caseco}{0}

\case \label{ca:5g1}
$q=1$, $M_{1i}^1=M_{1i}^2=\mu$ and both are matched.
Suppose $M_{1i}^1=\mu$ is matched with $M_{0l}^1=\mu$, and $M_{1i}^2=\mu$ is matched with $M_{0u}^2=\mu$.
From $M_{1i}^3=R$ and $\sigma^2\subseteq\sigma^3$ it follows that $M_{0u}^3=R$.
We split into subcases depending on the relationship between $l$ and $u$.

\setcounter{subcaseco}{0}

\subcase
$l=u$.
Note that there are no $\sigma^t$-relationships ($t=1,2$) between $X:=D_{1i}\cup D_{0l}$ and $\Ptw{n,d}\setminus X$. Hence
\[
(\sigma^1\vee\sigma^2)\restr_X=\sigma^1\restr_X\vee\sigma^2\restr_X=\widetilde{\mu}\restr_X\neq\nabla_X=
(\sigma^1\vee\sigma^3)\restr_X,
\]
a contradiction.

\subcase
$l<u$.
This implies $M_{0u}^1\neq\Delta$.
If $M_{0u}^1=\mu$ then it is matched to some $M_{1v}^1=\mu$ with $v>i$.
But then $M_{0u}^3=M_{1v}^3=R$, and so
\[
(\sigma^1\cap\sigma^3)\cap(D_{0u}\times D_{1v})=\widetilde{\mu}\cap (D_{0u}\times D_{1v})\neq\emptyset
=(\sigma^1\cap\sigma^2)\cap (D_{0u}\times D_{1v}),
\]
a contradiction.
If $M_{0u}^1=R$ then
\[
(\sigma^1\cap\sigma^3)\cap(D_{0u}\times\{\zero\})=D_{0u}\times\{\zero\}\neq\emptyset=(\sigma^1\cap\sigma^2)\cap (D_{0u}\times\{\zero\}),
\]
a contradiction.

\subcase $l>u$.
This time $M_{0l}^3=R$, and hence
\[
(\sigma^1\cap\sigma^3)\cap (D_{0l}\times D_{1i})=\widetilde{\mu}\cap (D_{0l}\times D_{1i})\neq \emptyset
= (\sigma^1\cap\sigma^2)\cap (D_{0l}\times D_{1i}),
\]
a contradiction.

\case \label{ca:5g2}
$q=1$, $M_{1i}^1$ is one of $\Delta,\muup,\mudown$ or  an unmatched $\mu$,
and $M_{1i}^2=\mu$.
If $M_{1i}^2=\mu$ is unmatched, then there are no $\sigma^t$-relationships ($t=1,2$) between
$D_{1i}$ and $\Ptw{n,d}\setminus D_{1i}$, and hence
\[
\nab_{D_{1i}} = \si^\vee\restr_{D_{1i}} = (\sigma^1\vee\sigma^2)\restr_{D_{1i}}=\sigma^1\restr_{D_{1i}}\vee\sigma^2\restr_{D_{1i}}=\widetilde{\mu}\restr_{D_{1i}},
\]
a contradiction.  So suppose that $M_{1i}^2=\mu$ is matched with $M_{0l}^2=\mu$.  Since $M_{1i}^3=R$ and ${\si^3\cap(D_{0l}\times D_{1i})\supseteq\si^2\cap(D_{0l}\times D_{1i})\not=\emptyset}$, we have $M_{0l}^3=R$.
If $M_{0l}^1=\Delta$ a similar argument to the above, but this time restricting to $D_{0l}\cup D_{1i}$, leads to a contradiction. 
If $M_{0l}^1=\mu$, and is matched with $M_{1u}^1=\mu$, then necessarily $u>i$, so that $M_{1u}^3=R$, and therefore
\[
(\sigma^1\cap\sigma^3)\cap (D_{0l}\times D_{1u})=\widetilde{\mu}\cap (D_{0l}\times D_{1u})\neq\emptyset
=(\sigma^1\cap\sigma^2)\cap (D_{0l}\times D_{1u}),
\]
a contradiction. Finally, if $M_{0l}^1=R$ then
\[
(\sigma^1\cap\sigma^3)\cap (D_{0l}\times\{\zero\})=D_{0l}\times\{\zero\}\neq \emptyset
= (\sigma^1\cap\sigma^2)\cap(D_{0l}\times\{\zero\}),
\]
a contradiction.

\case\label{ca:5g2a}
$q=1$, $M_{1i}^1=\mu$, and $M_{1i}^2$ is one of $\Delta,\muup,\mudown$ or an unmatched $\mu$.  As in Case \ref{ca:5g2}, $M_{1i}^1=\mu$ must be matched to some $M_{0u}^1=\mu$.  Next note that
\[
\sigma^2\restr_{D_{1i}}=(\sigma^1\cap\sigma^2)\restr_{D_{1i}}=(\sigma^1\cap\sigma^3)\restr_{D_{1i}}=\widetilde\mu\restr_{D_{1i}}\cap\si^3\restr_{D_{1i}}.
\]
If $M_{1i}^3\in\{R,\mu\}$, it follows that $\sigma^2\restr_{D_{1i}}=\widetilde\mu\restr_{D_{1i}}$, and so $M_{1i}^2=\mu$.  By \eqref{eq:M>M}, the only other possibility is for $M_{1i}^3\in\{\muup,\mudown\}$ and $M_{1i}^2=\De$; but in this case ${\sigma^2\restr_{D_{1i}}=\widetilde\mu\restr_{D_{1i}}\cap\si^3\restr_{D_{1i}}=\sigma^3\restr_{D_{1i}}}$, and so $M_{1i}^2=M_{1i}^3$, a contradiction.  We therefore conclude that $M_{1i}^2=\mu$, which is then unmatched, and also $M_{1i}^3=R$.
If $M_{0u}^2=\Delta$, then there are no $\sigma^t$-relationships ($t=1,2$) between $D_{0u}\cup D_{1i}$ and the rest of $\Ptw{n,d}$, and the usual argument (using $M_{1i}^3=R$) again yields a contradiction.
Also, we cannot have $M_{0u}^2=\mu$, since $u<i$ already, and there is no suitable fC-matrix.
But if $M_{0u}^2=R$, then also $M_{0u}^3=R$, and hence
\[
(\sigma^1\cap\sigma^3)\cap(D_{0u}\times D_{1i})=\widetilde{\mu}\cap(D_{0u}\times D_{1i})\neq\emptyset
=(\sigma^1\cap\sigma^2)\cap(D_{0u}\times D_{1i}),
\]
a contradiction.

\case \label{ca:5g3}
$q=0$, $M_{0i}^1\in\{\Delta,\mu\}$ and $M_{0i}^2=\mu$.
Suppose $M_{0i}^2=\mu$ is matched to $M_{1j}^2=\mu$.
From $M_{0i}^3=R$ and $\sigma^2\subseteq\sigma^3$ it follows that $M_{1j}^3=R$.
Therefore $(\sigma^3\setminus\sigma^2)\restr_{D_{1j}}\neq\emptyset$, which
is the situation treated in Cases \ref{ca:5g1}--\ref{ca:5g2a}.

\case \label{ca:5g4}
$q=0$, $M_{0i}^1=\mu$ and $M_{0i}^2=\Delta$.
Let $M_{0i}^1=\mu$ be matched with $M_{1u}^1=\mu$.
If $M_{1u}^2$ is $\Delta,\muup,\mudown$ or an unmatched $\mu$, there are no $\sigma^t$-relationships ($t=1,2$) between $D_{0i}\cup D_{1u}$ and the rest of $\Ptw{n,d}$, and this leads to the usual contradiction.
If $M_{1u}^2=\mu$ is matched to $M_{0l}^2=\mu$ then necessarily $l>i$ and so $M_{0l}^3=R$;
from $\sigma^2\subseteq\sigma^3$ it follows that $M_{1u}^3=R$ as well, and this takes us back to Case \ref{ca:5g1}.
Finally, if $M_{1u}^2=R$, then also $M_{1u}^3=R$, and so
\[
(\sigma^1\cap\sigma^3)\cap (D_{0i}\times D_{1u})=\widetilde{\mu}\cap(D_{0i}\times D_{1u})\neq\emptyset
=(\sigma^1\cap\sigma^2)\cap (D_{0i}\times D_{1u}),
\]
a contradiction.
This completes the proof of this case, and indeed of Stage \ref{st:5g1}.

\stage \label{st:5g2}
$(\ba,\bb)\not\in \D$, but $\sigma^2\restr_{D_{sl}}=\sigma^3\restr_{D_{sl}}$ for all $s\in \bnz$ and $l\in \N$.
Since $(\ba,\bb)\in\sigma^3$ and $(\ba,\bb)\not\in\D$ it follows from Lemma \ref{la:Dqirj} that one of the following holds:
\bit
\item
$M_{qi}^3=M_{rj}^3=R$; or
\item
$M_{qi}^3=M_{rj}^3=\mu$ are matched.
\eit

\setcounter{caseco}{0}

\case \label{ca:5gg1}
$M_{qi}^3=M_{rj}^3=R$.
By the assumptions for this stage, $\sigma^2\restr_{D_{qi}}=\sigma^3\restr_{D_{qi}}=\nabla_{D_{qi}}$ and $\sigma^2\restr_{D_{rj}}=\sigma^3\restr_{D_{rj}}=\nabla_{D_{rj}}$.
On the other hand, since $(\ba,\bb)\in (\sigma^3\setminus\sigma^2)\cap (D_{qi}\times D_{rj})$, we cannot have
$M_{qi}^2=M_{rj}^2=R$.
Without loss assume that $M_{qi}^2\neq R$. 
As in Remark \ref{rem:amb}
we must have $q=n$ and $M_{ni}^2=\S_n$.
If $M_{ni}^1\not=R$, then there are no $\sigma^t$-relationships ($t=1,2$) between~$D_{ni}$ and $\Ptw{n,d}\setminus D_{ni}$, contradicting $(\ba,\bb)\in\sigma^1\vee\sigma^3=\sigma^1\vee\sigma^2$; therefore $M_{ni}^1=R$.
But then ${(\ba,\zero)\in (\sigma^1\cap\sigma^3)\setminus(\sigma^1\cap\sigma^2)}$, a contradiction.

\case \label{ca:5gg2}
$M_{qi}^3=M_{rj}^3=\mu$ are matched.
Without loss we may assume that $q=0$ and $r=1$, and we note that $i<j$.  By \eqref{eq:M>M} we have $M_{0i}^2=\De$.
From $\sigma^2\restr_{D_{1j}}=\sigma^3\restr_{D_{1j}}$ it follows that $M_{1j}^2=\mu$.
This entry cannot be matched to $M_{0i}^2$, because $(\ba,\bb)\in (\sigma^3\setminus\sigma^2)\cap
(D_{0i}\times D_{1j})$. On the other hand, it cannot be matched to any other entry in row $1$, as that would violate
$\sigma^2\subseteq\sigma^3$.
We conclude that $M_{1j}^2=\mu$ is unmatched.

Since $(\ba,\bb)\in\sigma^1\vee\sigma^3=\sigma^1\vee\sigma^2$, there must be $\sigma^1$-relationships between
$D_{0i}$ and another $\D$-class, and also between $D_{1j}$ and another $\D$-class.
In particular, $M_{0i}^1,M_{1j}^1\in\{\mu,R\}$, and if $M_{1j}^1=\mu$ then this is matched in $M^1$.
We cannot have $M_{0i}^1=M_{1j}^1=R$, as this would imply $(\ba,\bb)\in\sigma^1$, in contradiction with
\eqref{eq:5g1}.
For the same reason we cannot have $M_{0i}^1=M_{1j}^1=\mu$ matched to each other.  

\setcounter{subcaseco}{0}

\subcase $M_{0i}^1=\mu$.  This is necessarily matched to some $M_{1l}^1$, and (since either $M_{1j}^1=R$ or else $M_{1j}^1=\mu$ is not matched to $M_{0i}^1$) we have $l<j$.  Since $M_{1j}^2=\mu$ is unmatched, it follows that $M_{1l}^1=\De=M_{0i}^2$.  Thus, there are no $\si^1\vee\si^2$-relationships between $X:=D_{0i}\cup D_{1l}$ and the rest of~$\Ptw{n,d}$.  But this contradicts $(\ba,\bb)\in\si^1\vee\si^3=\si^1\vee\si^2$, since $\ba\in X$ and $\bb\not\in X$.

\subcase $M_{0i}^1=R$.  As noted above, here we must have $M_{1j}^1=\mu$, and this must be matched to some $M_{0l}^1=\mu$ with $l<i$.  This time we have $M_{0l}^2=\De$, and since $M_{1j}^2=\mu$ is unmatched, it follows that there are no $\si^1\vee\si^2$-relationships between $D_{0l}\cup D_{1j}$ and the rest of $\Ptw{n,d}$.  This leads to the same contradiction as in the previous subcase.

This completes the proof of this case, Stage \ref{st:5g2}, and of the proposition. 
\end{proof}

\begin{rem}\label{rem:sts}
Our proof that the lattice $\Cong(\Ptw{n,d})$ is modular involved showing that it contains no pentagon sublattices.  Another way one might hope to show a congruence lattice is modular is to use a classical result of J\'onsson \cite{Jonsson1953}.  Specifically, it follows from \cite[Theorem 1.2]{Jonsson1953} that if $\si\circ\tau\circ\si = \tau\circ\si\circ\tau$ for all congruences $\si$ and $\tau$ on some algebra $A$, then $\Cong(A)$ is modular.  One might therefore wonder if J\'onsson's condition holds in $\Cong(\Ptw{n,d})$, but it turns out that it does not (apart from trivially small cases).  For example, suppose $n\geq2$ and $d\geq1$, and consider the congruences $\si=\cg(M^1)$ and $\tau=\cg(M^2)$ of $\Ptw{n,d}$, for the fC-matrices
\[
M^1 = \Cmatsetup
\begin{array}{|c|c|c|c|c|}\hline
\renewcommand{\arraystretch}{2}
\cellcolor{delcol}\Delta & \cellcolor{delcol}\Delta & \cellcolor{delcol}\Delta & \cellcolor{delcol}\cdots&\cellcolor{delcol}\De \\ \hline
\cellcolor{delcol}\vvdots & \cellcolor{delcol}\vvdots & \cellcolor{delcol}\vvdots & \cellcolor{delcol}\vvdots&\cellcolor{delcol}\vvdots \\ \hline
\cellcolor{delcol}\Delta & \cellcolor{delcol}\Delta & \cellcolor{delcol}\Delta & \cellcolor{delcol}\cdots&\cellcolor{delcol}\De \\ \hline
\cellcolor{delcol}\Delta & \cellcolor{Rcol}R & \cellcolor{Rcol}R & \cellcolor{Rcol}\cdots& \cellcolor{Rcol}R \\ \hline
\cellcolor{delcol}\Delta & \cellcolor{Rcol}R & \cellcolor{Rcol}R & \cellcolor{Rcol}\cdots& \cellcolor{Rcol}R \\ \hline
\end{array}
\AND
M^2 = \Cmatsetup
\begin{array}{|c|c|c|c|c|}\hline
\renewcommand{\arraystretch}{2}
\cellcolor{delcol}\Delta & \cellcolor{delcol}\Delta & \cellcolor{delcol}\Delta & \cellcolor{delcol}\cdots&\cellcolor{delcol}\De \\ \hline
\cellcolor{delcol}\vvdots & \cellcolor{delcol}\vvdots & \cellcolor{delcol}\vvdots & \cellcolor{delcol}\vvdots&\cellcolor{delcol}\vvdots \\ \hline
\cellcolor{delcol}\Delta & \cellcolor{delcol}\Delta & \cellcolor{delcol}\Delta & \cellcolor{delcol}\cdots&\cellcolor{delcol}\De \\ \hline
\cellcolor{delcol}\Delta & \cellcolor{mucol}\mu & \cellcolor{Rcol}R & \cellcolor{Rcol}\cdots& \cellcolor{Rcol}R \\ \hline
\cellcolor{mucol}\mu & \cellcolor{Rcol}R & \cellcolor{Rcol}R & \cellcolor{Rcol}\cdots& \cellcolor{Rcol}R \\ \hline
\end{array}\ .
\]
Let $\al,\be\in D_1$ be such that $\wh\al\not=\wh\be$.  Then  
\[
(\wh\al,0) \mr\tau (\al,1) \mr\si (\be,1) \mr\tau (\wh\be,0),
\]
and hence $((\wh\al,0),(\wh\be,0))\in \tau\circ\si\circ\tau$.
We claim that $((\wh\al,0),(\wh\be,0))\not\in\si\circ \tau\circ\si$.
To see this suppose
\[
(\wh\al,0) \mr\si \ba \mr\tau \bb \mr\si \bc \qquad\text{where $\ba,\bb,\bc\in\Ptw{n,d}$.}
\]
By the form of $\si$ and $\tau$, we must have $\ba=(\wh\al,0)$, and then either $\bb=(\wh\al,0)$ or else $\bb=(\ga,1)$ for some $\ga\in D_1$ with $\wh\ga=\wh\al$.  In the first case it follows that $\bc=(\wh\al,0)$; in the second, $\bc$ belongs to the ideal class $I(\si)$.  In both cases it follows that $\bc\not=(\wh\be,0)$.  
Hence $\tau\circ\si\circ\tau\neq \si\circ \tau\circ\si$, as desired.
\end{rem}

\begin{rem}
\label{rem:deg}
We conclude with a brief overview of the properties of $\Cong(\Ptw{n,d})$ for $n\leq1$.
When $n=0$, the monoid $\Ptw{n,d}$ is isomorphic to the finite nilpotent monoid $\N/(d+1,d+2)^\sharp$ of order $d+2$,
and its congruence lattice is a chain of length $d+2$.
For $n=1$ and $d=0$ the lattice $\Cong(\Ptw{1,0})$ is a three-element chain. For $n=1$ and $d>0$ we have three families of fC-matrix:
\begin{align*}
&\Cmatsetup\begin{array}{r|c|c|c|c|c|c|c|c|c|c|}\hhline{~|-|-|-|-|-|-|-|-|-|-|}
\renewcommand{\arraystretch}{2}
\text{\scriptsize 1} &\cellcolor{delcol}\Delta &\cellcolor{delcol}\dots & \cellcolor{delcol}\Delta & 
\cellcolor{delcol}\Delta & \cellcolor{delcol}\dots & \cellcolor{delcol}\Delta &
\cellcolor{Rcol} R& \cellcolor{Rcol}\dots&\cellcolor{Rcol} R
\\ \hhline{~|-|-|-|-|-|-|-|-|-|-|}
\text{\scriptsize 0}&\cellcolor{delcol}\Delta &\cellcolor{delcol}\dots &\cellcolor{delcol}\Delta & 
\cellcolor{Rcol} R &\cellcolor{Rcol}\dots&\cellcolor{Rcol} R&\cellcolor{Rcol} R& \cellcolor{Rcol}\dots&\cellcolor{Rcol} R
\\ \hhline{~|-|-|-|-|-|-|-|-|-|-|}
\multicolumn{1}{c}{}&\multicolumn{1}{c}{}&\multicolumn{1}{c}{}&\multicolumn{1}{c}{}&\multicolumn{1}{c}{\text{\scriptsize $i$}}&\multicolumn{1}{c}{}&\multicolumn{1}{c}{}&\multicolumn{1}{c}{\text{\scriptsize $j$}}&\multicolumn{1}{c}{}&\multicolumn{1}{c}{}
\end{array}
&&\hspace{-1.8cm}\text{for $0\leq i\leq j\leq d+1$,}\\[2mm]
&\Cmatsetup\begin{array}{r|c|c|c|c|c|c|c|c|c|c|c|c|}\hhline{~|-|-|-|-|-|-|-|-|-|-|-|}
\renewcommand{\arraystretch}{2}
\text{\scriptsize 1} &\cellcolor{delcol}\Delta &\cellcolor{delcol}\dots & \cellcolor{delcol}\Delta &\cellcolor{delcol}\Delta & \cellcolor{mucol} \mu & \cellcolor{mucol} \dots &\cellcolor{mucol} \mu &
\cellcolor{mucol} \mu &
\cellcolor{Rcol} R&\cellcolor{Rcol}\dots& \cellcolor{Rcol} R\\ \hhline{~|-|-|-|-|-|-|-|-|-|-|-|}
\text{\scriptsize 0}&\cellcolor{delcol}\Delta &\cellcolor{delcol}\dots &\cellcolor{delcol}\Delta & 
\cellcolor{mucol}\mu & \cellcolor{mucol} \mu & \cellcolor{mucol} \dots &\cellcolor{mucol} \mu &
\cellcolor{Rcol} R &\cellcolor{Rcol} R&\cellcolor{Rcol}\dots& \cellcolor{Rcol} R\\ \hhline{~|-|-|-|-|-|-|-|-|-|-|-|}
\multicolumn{1}{c}{}&\multicolumn{1}{c}{}&\multicolumn{1}{c}{}&\multicolumn{1}{c}{}&\multicolumn{1}{c}{\text{\scriptsize $i$}}&\multicolumn{1}{c}{}&\multicolumn{1}{c}{}&\multicolumn{1}{c}{}&\multicolumn{1}{c}{\text{\scriptsize $j$}}&\multicolumn{1}{c}{}&\multicolumn{1}{c}{}
\end{array}
&&\hspace{-1.8cm}\text{for $0\leq i< j\leq d$,}\\[2mm]
&\Cmatsetup\begin{array}{r|c|c|c|c|c|c|c|c|c|c|c|}\hhline{~|-|-|-|-|-|-|-|-|-|-|-|}
\renewcommand{\arraystretch}{2}
\text{\scriptsize 1} &\cellcolor{delcol}\Delta &\cellcolor{delcol}\dots & \cellcolor{delcol}\Delta  & \cellcolor{delcol}\Delta &
\cellcolor{delcol}\Delta & \cellcolor{delcol}\dots & \cellcolor{delcol}\Delta &
\cellcolor{mucol} \mu & \cellcolor{Rcol} R& \cellcolor{Rcol}\dots & \cellcolor{Rcol} R 
\\ \hhline{~|-|-|-|-|-|-|-|-|-|-|-|}
\text{\scriptsize 0}&\cellcolor{delcol}\Delta &\cellcolor{delcol}\dots &\cellcolor{delcol}\Delta  
&\cellcolor{mucol}\mu&
\cellcolor{Rcol} R &\cellcolor{Rcol}\dots&\cellcolor{Rcol} R&\cellcolor{Rcol} R&\cellcolor{Rcol} R& \cellcolor{Rcol}\dots & \cellcolor{Rcol} R
\\ \hhline{~|-|-|-|-|-|-|-|-|-|-|-|}
\multicolumn{1}{c}{}&\multicolumn{1}{c}{}&\multicolumn{1}{c}{}&\multicolumn{1}{c}{}&\multicolumn{1}{c}{}&\multicolumn{1}{c}{\text{\scriptsize $i$}}&\multicolumn{1}{c}{}&\multicolumn{1}{c}{}&\multicolumn{1}{c}{}&\multicolumn{1}{c}{\text{\scriptsize $j$}}
\end{array}
&&\hspace{-1.8cm}\text{for $1\leq i< j-1\leq d$.}
\end{align*}
The lattice $\Cong(\Ptw{1,d})$ is modular:
the proof of Proposition \ref{pr:5gons} is valid for $n=1$, even though many of its cases do not arise.
And the lattice remains non-distributive, as witnessed by the diamond from Figure \ref{fig:pediPtw1}, which `survives' in all $\Cong(\Ptw{1,d})$ for $d\geq 1$.
In fact \cite[Figure 7]{ERtwisted1} shows the Hasse diagram of $\Cong(\Ptw{1,4})$, in which many copies of the diamond are seen.
(The pentagon in Figure \ref{fig:pediPtw1} of course does not survive in any $\Ptw{1,d}$.)
\end{rem}

\section{Generators of congruences of \boldmath{$\Ptw{n,d}$}}\label{sec:fgPnd}

The results of Section \ref{sec:fg} also have analogues for the finite monoids $\Ptw{n,d}$.  
Note that in $\Ptw{n,d}$ the non-zero principal ideals are
\[
I_{qi} :=  \bigcup \set{ D_{rj}}{ 0\leq r\leq q,\ i\leq j\leq d}   \cup  \{\zero\} \qquad\text{for $q\in\bnz$ and $i\in\bdz$.}
\]

\begin{thm}\label{thm:fpc}
Let $\ba=(i,\al)\in D_{qi}$ and $\bb=(j,\be)\in D_{rj}$, both considered as elements of $\Ptw{n,d}$, where $q\geq r$, and $i\leq j$ if $q=r$.
\ben
\item \label{fpc1} If $\al=\be$ and $i=j$ (i.e.~$\ba=\bb$), then $\pc\ba\bb=\De_{\Ptw{n,d}}$.  We also have $\pc\zero\zero=\De_{\Ptw{n,d}}$.
\item \label{fpc2} If $[q\geq2$ and $(\ba,\bb)\not\in\H]$ or if $[q=r\leq1$ and $[i\neq j$ or $\wh\al\neq \wh\be]]$ or if $[q=1$, $r=0$ and $[j\geq i$ or $\wh\al\neq \be]]$, then $\pc\ba\bb=R_{I_{qi}\cup I_{rj}}$.  We also have $\pc\ba\zero=R_{I_{qi}}$.
\item \label{fpc3} If $q\geq3$, $(\ba,\bb)\in\H$ and $\ba\neq \bb$, then with $N:=\norm{\pd(\al,\be)}$ we have $\pc\ba\bb=\cg(M)$ for the fC-matrix
\[
\Cmatsetup
\begin{array}{r|c|c|c|c|c|c|c|}\hhline{~|-|-|-|-|-|-|-}
\renewcommand{\arraystretch}{2}
\text{\scriptsize $q$} &\cellcolor{delcol}\De &\cellcolor{delcol}\cdots & \cellcolor{delcol}\De  & \cellcolor{Ncol} N & \cellcolor{Ncol}N & \cellcolor{Ncol} \cdots & \cellcolor{Ncol}N\\ \hhline{~|-|-|-|-|-|-|-} 
&\cellcolor{delcol}\De &\cellcolor{delcol}\cdots & \cellcolor{delcol}\De  & \cellcolor{Rcol} R & \cellcolor{Rcol}R & \cellcolor{Rcol} \cdots & \cellcolor{Rcol} R\\ \hhline{~|-|-|-|-|-|-|-}
&\cellcolor{delcol}\vvdots &\cellcolor{delcol}\vvdots & \cellcolor{delcol}\vvdots  & \cellcolor{Rcol} \vvdots & \cellcolor{Rcol}\vvdots & \cellcolor{Rcol} \vvdots & \cellcolor{Rcol} \vvdots\\ \hhline{~|-|-|-|-|-|-|-}
\text{\scriptsize $0$} &\cellcolor{delcol}\De &\cellcolor{delcol}\cdots & \cellcolor{delcol}\De  & \cellcolor{Rcol} R & \cellcolor{Rcol}R & \cellcolor{Rcol} \cdots & \cellcolor{Rcol} R\\ \hhline{~|-|-|-|-|-|-|-}
\multicolumn{1}{c}{}&\multicolumn{1}{c}{}&\multicolumn{1}{c}{}&\multicolumn{1}{c}{}&\multicolumn{1}{c}{\text{\scriptsize $i$}}&\multicolumn{1}{c}{}&\multicolumn{1}{c}{}&\multicolumn{1}{c}{}
\end{array}
\ .
\]
\item \label{fpc4} If $q=2$, $(\ba,\bb)\in\H$ and $\ba\neq \bb$, then $\pc\ba\bb=\cg(M)$ for the fC-matrix
\[
\Cmatsetup
\begin{array}{r|c|c|c|c|c|c|c|}\hhline{~|-|-|-|-|-|-|-}
\text{\scriptsize $2$} &\cellcolor{delcol}\De &\cellcolor{delcol}\cdots & \cellcolor{delcol}\De  & \cellcolor{Ncol} \S_2 & \cellcolor{Ncol}\S_2 & \cellcolor{Ncol} \cdots &  \cellcolor{Ncol} \S_2\\ \hhline{~|-|-|-|-|-|-|-} 
\text{\scriptsize $1$} &\cellcolor{delcol}\De &\cellcolor{delcol}\cdots & \cellcolor{delcol}\De  & \cellcolor{excepcol} \mu & \cellcolor{mucol}\mu & \cellcolor{mucol} \cdots &  \cellcolor{mucol} \mu\\ \hhline{~|-|-|-|-|-|-|-}
\text{\scriptsize $0$} &\cellcolor{delcol}\De &\cellcolor{delcol}\cdots & \cellcolor{delcol}\De  & \cellcolor{mucol} \mu & \cellcolor{mucol}\mu & \cellcolor{mucol} \cdots &  \cellcolor{Rcol} R\\ \hhline{~|-|-|-|-|-|-|-}
\multicolumn{1}{c}{}&\multicolumn{1}{c}{}&\multicolumn{1}{c}{}&\multicolumn{1}{c}{}&\multicolumn{1}{c}{\text{\scriptsize $i$}}&\multicolumn{1}{c}{}&\multicolumn{1}{c}{}&\multicolumn{1}{c}{}
\end{array}
\ .
\]
\item \label{fpc5} If $q=r=1$ and $i=j$, and if $(\al,\be)$ belongs to one of $\mu_1\sm(\muup\cup\mudown)$, $\muup\sm\De_{D_1}$ or $\mudown\sm\De_{D_1}$, then with $\ze=\mu$, $\muup$ or~$\mudown$, respectively, we have $\pc\ba\bb=\cg(M)$ for the fC-matrix
\[
\Cmatsetup
\begin{array}{r|c|c|c|c|c|c|c|c|}\hhline{~|-|-|-|-|-|-|-|-}
\text{\scriptsize $1$} &\cellcolor{delcol}\De &\cellcolor{delcol}\cdots & \cellcolor{delcol}\De  & \cellcolor{excepcol} \ze & \cellcolor{mucol}\mu & \cellcolor{mucol} \cdots & \cellcolor{mucol}\mu& \cellcolor{mucol}\mu\\ \hhline{~|-|-|-|-|-|-|-|-}
\text{\scriptsize $0$} &\cellcolor{delcol}\De &\cellcolor{delcol}\cdots & \cellcolor{delcol}\De  & \cellcolor{mucol} \mu & \cellcolor{mucol}\mu & \cellcolor{mucol} \cdots & \cellcolor{mucol}\mu& \cellcolor{Rcol}R\\ \hhline{~|-|-|-|-|-|-|-|-}
\multicolumn{1}{c}{}&\multicolumn{1}{c}{}&\multicolumn{1}{c}{}&\multicolumn{1}{c}{}&\multicolumn{1}{c}{\text{\scriptsize $i$}}&\multicolumn{1}{c}{}&\multicolumn{1}{c}{}&\multicolumn{1}{c}{}
\end{array}\ .
\]
\item \label{fpc6} If $q=1$, $r=0$, $i=j+1$ and $\wh\al=\be$, then $\pc\ba\bb=\cg(M)$ for the fC-matrix
\[
\Cmatsetup
\begin{array}{r|c|c|c|c|c|c|c|c|}\hhline{~|-|-|-|-|-|-|-|-}
\text{\scriptsize $1$} &\cellcolor{delcol}\De &\cellcolor{delcol}\cdots & \cellcolor{delcol}\De  & \cellcolor{excepcol} \De & \cellcolor{mucol}\mu & \cellcolor{mucol} \cdots & \cellcolor{mucol}\mu& \cellcolor{mucol}\mu\\ \hhline{~|-|-|-|-|-|-|-|-}
\text{\scriptsize $0$} &\cellcolor{delcol}\De &\cellcolor{delcol}\cdots & \cellcolor{delcol}\De  & \cellcolor{mucol} \mu & \cellcolor{mucol}\mu & \cellcolor{mucol} \cdots & \cellcolor{mucol}\mu& \cellcolor{Rcol}R\\ \hhline{~|-|-|-|-|-|-|-|-}
\multicolumn{1}{c}{}&\multicolumn{1}{c}{}&\multicolumn{1}{c}{}&\multicolumn{1}{c}{}&\multicolumn{1}{c}{\text{\scriptsize $j$}}&\multicolumn{1}{c}{\text{\scriptsize $i$}}&\multicolumn{1}{c}{}&\multicolumn{1}{c}{}
\end{array}\ .
\]
\item \label{fpc7} If $q=1$, $r=0$, $i>j+1$, and $\wh\al=\be$, then $\pc\ba\bb=\cg(M)$ for the fC-matrix
\[
\Cmatsetup
\begin{array}{r|c|c|c|c|c|c|c|c|c|c|c|}\hhline{~|-|-|-|-|-|-|-|-|-|-|-}
\text{\scriptsize $1$} &\cellcolor{delcol}\De &\cellcolor{delcol}\cdots & \cellcolor{delcol}\De& \cellcolor{delcol}\De & \cellcolor{delcol}\De &\cellcolor{delcol}\cdots & \cellcolor{delcol}\De  & \cellcolor{mucol}\mu & \cellcolor{Rcol} R & \cellcolor{Rcol}\cdots & \cellcolor{Rcol} R \\ \hhline{~|-|-|-|-|-|-|-|-|-|-|-}
\text{\scriptsize $0$} &\cellcolor{delcol}\De &\cellcolor{delcol}\cdots & \cellcolor{delcol}\De  & \cellcolor{mucol}\mu & \cellcolor{Rcol} R & \cellcolor{Rcol}\cdots & \cellcolor{Rcol} R & \cellcolor{Rcol} R & \cellcolor{Rcol} R & \cellcolor{Rcol}\cdots & \cellcolor{Rcol} R \\ \hhline{~|-|-|-|-|-|-|-|-|-|-|-}
\multicolumn{1}{c}{}&\multicolumn{1}{c}{}&\multicolumn{1}{c}{}&\multicolumn{1}{c}{}&\multicolumn{1}{c}{\text{\scriptsize $j$}}&\multicolumn{1}{c}{}&\multicolumn{1}{c}{}&\multicolumn{1}{c}{}&\multicolumn{1}{c}{\text{\scriptsize $i$}}&\multicolumn{1}{c}{}&\multicolumn{1}{c}{}&\multicolumn{1}{c}{}
\end{array}\ .
\]
\een
The above exhaust all principal congruences on $\Ptw{n,d}$.
\qed
\end{thm}

Recalling once more the Correspondence Theorem, from Theorem \ref{thm:fg} we obtain:

\begin{cor}\label{co:ffg}
Every congruence on $\Ptw{n,d}$ can be generated by at most $\lceil\frac{5n}2\rceil$ pairs.  \epfres
\end{cor}

\begin{rem}\label{rem:fg0}
When $d=0$ we can strengthen Corollary \ref{co:ffg} considerably.  Examining Figure \ref{fig:CongPn0}, it is clear that all congruences of $\Ptw{n,0}$ are principal, with the possible exception of $\mu$ and $R_{\S_2}$, and that these are generated by two pairs.  In fact, $\mu$ is also principal, as follows from Theorem~\ref{thm:fpc}\ref{fpc5}.  On the other hand, since $R_{\S_2} = R_1\cup\mu_{\S_2}$ is the union of two proper sub-congruences, it follows that $R_{\S_2}$ is minimally generated by two pairs.  This is reminiscent of the situation with the ordinary partition monoid $\P_n$, whose congruences are all generated by at most two pairs \cite{EMRT2018}.

Corollary \ref{co:ffg} can be similarly improved for other values of $d$ relatively small compared to~$n$.  Indeed, examining the proof of Theorem \ref{thm:fg}, congruences requiring many pairs to generate involve ideals with many `corners' and C-matrices with many distinct $N$-symbols, both of which can only occur when $d$ is suitably large.
\end{rem}

\section{Enumeration of congruences of \boldmath{$\Ptw{n,d}$}}
\label{sec:enumeration}

This section is concerned with determining the numbers $|\Cong(\Ptw{n,d})|$.  
The main result is Theorem~\ref{th:uberenu}, in which we
obtain a closed form for $|\Cong(\Ptw{n,d})|$, and prove that the array formed by these numbers has a rational generating function in two variables.
Furthermore, we show that for fixed $n\geq0$ or $d\geq0$, $|\Cong(\Ptw{n,d})|$ is a polynomial in $d\geq0$ or $n\geq4$, respectively, and give the asymptotic behaviour in Remark \ref{rem:as}.

For $n\geq 1$ and $d\geq 0$ we write $\C_{n,d}$ for the set of all $\bnz\times\bdz$ fC-matrices, and $c_{n,d}:=|\C_{n,d}|$ for the number of such matrices.  It follows from Theorem \ref{thm:finmain} that $|\Cong(\Ptw{n,d})|=c_{n,d}$ for all such~$n$ and $d$.

We begin by recording the values of $|\Cong(\Ptw{n,d})|$ for $n\leq1$:

\begin{lemma}
For all $d\geq 0$ we have:
\label{la:enun01}
\begin{thmenumerate}
\item  \label{it:n011}
$|\Cong(\Ptw{0,d})|=d+2$;
\item \label{it:n012}
$|\Cong(\Ptw{1,d})|=(3d^2+5d+6)/2$.
\end{thmenumerate}
\end{lemma}

\begin{proof}
This follows from Remark \ref{rem:deg}. For $n=0$ the result is explicitly stated, and for $n=1$ it is a straightforward counting of the fC-matrices listed there.
\end{proof}

To deal with larger values of $n$, we assume until further notice that $n\geq2$.  
For any $d\geq 0$ we express $|\Cong(\Ptw{n,d})|$ in terms of certain numbers $c_{n,d}(\sigma)$.
These numbers are defined recursively in $d$, and are also indexed by congruences $\sigma\in\Cong(\Ptw{n,0})$ of the $0$-twisted monoid $\Ptw{n,0}$ as described in Remark \ref{rem:CongPn0} and depicted in Figure \ref{fig:CongPn0}.
For brevity we write $\Delta$ for $\Delta_{\Ptw{n,0}}$, and the interval $[\Delta,\sigma]$ appearing in the last line is in $\Cong(\Ptw{n,0})$.

\begin{lemma}
\label{la:enum}
For $n\geq 2$ and $d\geq 0$ we have
\begin{equation}
\label{eq:Cgcnd}
|\Cong(\Ptw{n,d})|= \sum_{\sigma\in\Cong(\Ptw{n,0})} c_{n,d}(\sigma),
\end{equation}
where the numbers $c_{n,d}(\sigma)$ satisfy the following recursion:
\begin{equation}
\label{eq:rec0}
c_{n,0}(\sigma)=1\quad \text{for all } \sigma\in\Cong(\Ptw{n,0}),
\end{equation}
and, for $d\geq 1$,
\begin{subnumcases}{\label{eq:recn} c_{n,d}(\sigma)=}
1 & if  $\sigma=\Delta$,\label{eq:recna}\\
d+1 & if   $\sigma=R_0,\mudown,\muup$,\label{eq:recnb}\\
6d & if  $\sigma=\mu$,\label{eq:recnc}\\
2d^2+5d & if  $\sigma=\mu_{\S_2}$,\label{eq:recnd}\\
(9d^2-d+4)/2 & if  $\sigma=R_1$,\label{eq:recne}\\
(13d^3+21d^2+2d+12)/6 & if $ \sigma=R_{\S_2}$,\label{eq:recnf}\\
\displaystyle{\sum_{\tau\in [\Delta,\sigma]}} c_{n,d-1}(\tau) & otherwise.\label{eq:recng}
\end{subnumcases}
\end{lemma}

\begin{proof}
Examining the defining conditions for fC-matrices given in Definition \ref{de:fC}, we see that the right-most column of an fC-matrix from $\C_{n,d}$ is itself an fC-matrix from $\C_{n,0}$, as listed in \eqref{eq:fC0}.  (This is not necessarily the case for the non-final columns of fC-matrices.)  These one-column fC-matrices are in one-one correspondence with the congruences on $\Ptw{n,0}$.  For $\si\in\Cong(\Ptw{n,0})$, let us write~$\C_{n,d}(\si)$ for the set of all $\bn_0\times\bd_0$ fC-matrices whose final column corresponds to the congruence~$\si$, and let $c_{n,d}(\si):=|\C_{n,d}(\si)|$.  So of course
\begin{equation}\label{eq:cndsi}
c_{n,d} = \sum_{\si\in\Cong(\Ptw{n,0})}c_{n,d}(\si),
\end{equation}
and it remains to show that the numbers $c_{n,d}(\sigma)$ satisfy the recursion given by \eqref{eq:rec0} and \eqref{eq:recn}.

With \eqref{eq:rec0} and  \eqref{eq:recna} being clear, we begin with \eqref{eq:recnb}.  In these cases, the only possible row type is \ref{fRT3} with $l=d+1$.  Here, $\ze$ is $\De$, $\muup$ or $\mudown$, as appropriate, and $0\leq k\leq d$ can be chosen arbitrarily.  

For \eqref{eq:recnc}, row types \ref{fRT2}--\ref{fRT4} are possible.  In type \ref{fRT2} we have $k=d$, while $0\leq i\leq d-1$ and $\ze\in\{\mu,\muup,\mudown,\De\}$ can both be chosen arbitrarily; there are thus $4d$ possibilities here.  In type \ref{fRT3} we have $l=d+1$ and $\ze=\mu$, while $0\leq k\leq d$ can be chosen arbitrarily, giving $d+1$ possibilities.  Similarly, there are $d-1$ possibilities in type \ref{fRT4}.  Adding gives $6d$.

For \eqref{eq:recnd}, the same row types are possible as for \eqref{eq:recnc}.  For types \ref{fRT3} and \ref{fRT4} we again respectively have $d+1$ and $d-1$ possibilities.  For type \ref{fRT2}, we may choose some number of~$\S_2$-labels on row $2$; let $j$ be minimal such that $M_{2j}=\S_2$, noting that $j\leq d$.  Then $i\leq j\leq d$ if $\ze=\mu$ or $i+1\leq j\leq d$ otherwise; this gives a total of $\sum_{i=0}^{d-1}((d-i+1)+3(d-i)) = 2d^2+3d$ possibilities with this type.  Adding gives the desired formula. 

Items \eqref{eq:recne} and \eqref{eq:recnf} can be checked directly in similar fashion to \eqref{eq:recnd}, but they also follow from \eqref{eq:recng} and induction on $d$.  Indeed, although \eqref{eq:recng} is only stated for $\si\supseteq R_2$, we will shortly prove that it holds for all $\si\supseteq R_1$.  For example, writing $p(d)=(9d^2-d+4)/2$, the inductive step in~\eqref{eq:recne} involves verifying that
\[
p(d-1) + c_{n,d-1}(\mu) + c_{n,d-1}(\muup) + c_{n,d-1}(\mudown) + c_{n,d-1}(R_0) + c_{n,d-1}(\De_{\Ptw{n,0}}) = p(d).
\]

For \eqref{eq:recng}, consider an fC-matrix $M\in\C_{n,d}(\si)$, where $\si\supseteq R_1$.  This means that column~$d$ of~$M$ has $R$s in the bottom two positions, and it follows that the matrix $M'$ obtained from $M$ by removing column~$d$ belongs to $\C_{n,d-1}$.  Moreover, by examining the row types, we see that the congruence~$\tau$ on $\Ptw{n,0}$ corresponding to column $d-1$ of $M'$ satisfies $\tau\sub\si$, and  the claim follows.
\end{proof}

It is clear from Lemma \ref{la:enum} that the numbers $c_{n,d}(\sigma)$ actually do not depend on $n$ in the following sense.
Recall that for $m\geq n$  the lattice $\Cong(\Ptw{n,0})$ naturally embeds as an ideal of $\Cong(\Ptw{m,0})$.  
Furthermore, under this embedding, any congruence $\sigma\in\Cong(\Ptw{n,0})$ and its image in $\Cong(\Ptw{m,0})$
have the same label as per Figure \ref{fig:CongPn0} (using $R_n$ for the universal congruence on~$\Ptw{n,0}$).
Thus, identifying a congruence with its label, we have that
$c_{n,d}(\sigma)=c_{m,d}(\sigma)$, and we will write simply $c_d(\sigma)$ for this number.

Now suppose $\sigma\in\Cong(\Ptw{n,0})$ is any congruence strictly containing $R_{\S_2}$, and let $\sigma'\in\Cong(\Ptw{n,0})$ be the unique congruence that $\sigma$ covers.
Then, for any $d\geq1$, two applications of \eqref{eq:recng} give
\begin{equation}
\label{eq:cdpasc}
c_d(\sigma)=\sum_{\tau\in [\Delta,\sigma]} c_{d-1}(\tau)=\sum_{\tau\in [\Delta,\sigma']} c_{d-1}(\tau) + c_{d-1}(\sigma)
=c_d(\sigma')+c_{d-1}(\sigma).
\end{equation}
(For the $\si=R_2$ case, where $\si'=R_{\S_2}$, recall from the proof of Lemma \ref{la:enum} that \eqref{eq:recng} holds for any $\si\supseteq R_1$.)
Also, for any $d\geq0$, \eqref{eq:Cgcnd} and \eqref{eq:recng} give
\begin{equation}
\label{eq:Cgsid}
|\Cong(\Ptw{n,d})|=\sum_{\sigma\in \Cong(\Ptw{n,0})} c_d(\sigma)
=\sum_{\sigma\in [\Delta,R_n]} c_d(\sigma)=c_{d+1}(R_n).
\end{equation}

If we identify every $\sigma\in\Cong(\Ptw{n,0})$ satisfying $R_{\S_2}\subseteq \sigma$ with its position
$\pos(\sigma)$
in the natural sequence
\[
R_{\S_2},R_2,R_{\A_3}, R_{\S_3},R_3,R_{\mathcal K_4},R_{\A_4},R_{\S_4},R_4,R_{\A_5},R_{\S_5},R_5,\dots,
\]
starting with $\pos(R_{\S_2})=0$, we can interpret the numbers $c_d(\sigma)$ as a 2-dimensional array $(b(k,d))_{k,d\geq 0}$, where
$c_d(\sigma)=b(\pos(\sigma),d)$.
`Forgetting' the first column $(b(k,0))_{k\geq 0}$, and `shifting' the array by one to the left, we obtain a new array
$(a(k,d))_{k,d\geq 0}$ featuring in the following:

\begin{lemma}
\label{la:ast}
Let $(a(k,d))_{k,d\in\N}$ be the array defined recursively as follows:
\begin{align}
\label{eq:ast1} a(0,d) &= (13d^3+60d^2+83d+48)/6 &&\text{for $d\geq 0$,}\\
\label{eq:ast2} a(k,0) &= k+8 && \text{for $k\geq 0$,}\\
\label{eq:ast3} a(k,d) &= a(k-1,d)+a(k,d-1) &&  \text{for $k,d\geq 1$.}
\end{align}
Then, for all $n\geq 2$ and $d\geq 0$, we have
\begin{equation}
\label{eq:Cgast}
|\Cong(\Ptw{n,d})|=a(k,d),\quad \text{where } k=k(n):=\begin{cases} 1 & \text{if } n=2\\ 4 & \text{if } n=3 \\ 3n-4 &\text{if } n\geq 4.\end{cases}
\end{equation}
\end{lemma}

\begin{proof}
The array $(a(k,d))_{k,d\in\N}=(b(k,d+1))_{k,d\in\N}$ is obtained in the way explained before the lemma.
The polynomial featuring in \eqref{eq:ast1} is obtained from that in \eqref{eq:recnf} byevaluating it at $d+1$.
The formula \eqref{eq:ast2} follows, using \eqref{eq:recng} and \eqref{eq:rec0}, from
\[
c_1(\sigma)=\sum_{\tau\in [\Delta,\sigma]}c_0(\tau)=\sum_{\tau\in [\Delta,\sigma]}1
=\bigl|[\Delta,\sigma]\bigr|=\pos(\sigma)+8.
\]
The `Pascal's triangle' recurrence \eqref{eq:ast3} is a direct translation of \eqref{eq:cdpasc}.
Finally, the equality~\eqref{eq:Cgast} is a straightforward translation of   \eqref{eq:Cgsid}, upon noting
that
\[
\pos(R_n)=\begin{cases} 1 & \text{if } n=2\\ 4 &\text{if } n=3\\ 3n-4 & \text{if } n\geq 4.\end{cases}
\qedhere
\]
\end{proof}

The recurrence \eqref{eq:ast1}--\eqref{eq:ast3} for the $a(k,d)$ is sufficiently simple that it allows standard methods to be deployed to compute the generating function
\[
A(x,y) := \sum_{k,d\geq0} a(k,d) x^ky^d,
\]
and even an exact closed form for $a(k,d)$, which we now proceed to do.
While all the manipulations are standard and elementary, the actual calculations have been performed using 
Maple\texttrademark~\cite{Maple}.

Recall that the generating functions for the constant sequence $(1)_{n\geq 0}$ and for the power sequence $(n^m)_{n\geq 0}$, for fixed $m>0$, are given by
\[
\sum_{n\geq0}x^n = \frac1{1-x} \AND 
\sum_{n\geq0}n^mx^n = \sum_{j=0}^m \stirling{m}{j}\frac{j!\cdot x^j}{(1-x)^{j+1}},
\]
where $\stirling{m}{j}$ denotes the Stirling number of the second kind; see \cite[Eq.~(7.46)]{concrete}.
The generating function of any polynomial is then a linear combination of these functions.
In particular, using \eqref{eq:ast1} and \eqref{eq:ast2}, the generating functions for the boundary sequences $(a(k,0))_{k\geq 0}$ and $(a(0,d))_{d\geq 0}$
in variables $x$ and $y$ respectively are:
\[
B_1(x):=\sum_{k\geq0}a(k,0)x^k = \frac{8-7x}{(1-x)^2}
\AND
B_2(y):=\sum_{d\geq0}a(0,d)y^d = \frac{-2y^3+5y^2+2y+8}{(1-y)^4}.
\]
Straightforward manipulations of power series now give 
\begin{align}
\nonumber
A(x,y) &=\frac{B_1(x)\cdot (1-x) + B_2(y)\cdot (1-y) -a(0,0)}{1-x-y}\\
\label{eq:GFa}
&=\frac{1}{1-x-y}\biggl( \frac{x}{1-x} +\frac{-2y^3+5y^2+2y+8}{(1-y)^3}\biggr).
\end{align}
From this we can obtain an exact formula for $a(k,d)$ by writing down the power series for the terms appearing in 
\[
\frac{1}{1-x-y}=\sum_{i,j\geq 0} \binom{i+j}{i} x^iy^j \COMMA
\frac{x}{1-x} = \sum_{i\geq1} x^i \COMMA
\frac{1}{(1-y)^3}=\sum_{j\geq0} \binom{j+2}{2} y^j,
\]
and finding the coefficient of $x^ky^d$:
\begin{align}
\nonumber a(k,d) &= \sum_{i=0}^{k-1}\binom{i+d}{i} + 8\sum_{j=0}^d\binom{k+j}{k}\binom{d-j+2}{2}
+2\sum_{j=0}^{d-1}\binom{k+j}{k}\binom{d-j+1}{2}\\
\nonumber & \hspace{25mm} +5\sum_{j=0}^{d-2}\binom{k+j}{k}\binom{d-j}{2}
-2\sum_{j=0}^{d-3}\binom{k+j}{k}\binom{d-j-1}{2}.
\intertext{Using standard binomial coefficient identities this simplifies to}
\label{eq:aform}
a(k,d) &= \binom{k+d}{d+1}+8\binom{k+d+3}{k+3}+2\binom{k+d+2}{k+3}+5\binom{k+d+1}{k+3}-2\binom{k+d}{k+3},
\end{align}
where we adopt the convention that $\binom{s}{t}=0$ for $t>s$.
Of course, this formula can also be proved directly, by induction, using the recursive definition \eqref{eq:ast1}--\eqref{eq:ast3} of the $a(k,d)$.

We are now ready to state and prove the main result of this section:

\begin{thm}
\label{th:uberenu}
\begin{thmenumerate}
\item \label{it:ub1}
For $n,d\geq0$, the number of congruences of the twisted partition monoid~$\Ptw{n,d}$ is as follows:
\begin{align*}
|\Cong(\Ptw{0,d})|  &=d+2,
\\
|\Cong(\Ptw{1,d})|&=\textstyle{\frac{3d^2+5d+6}{2}},
\\
|\Cong(\Ptw{2,d})|&=\textstyle{\frac{13d^4+106d^3+299d^2+398d+216}{24}},
\\
|\Cong(\Ptw{3,d})|&=\textstyle{\frac{13d^7+322d^6+3262d^5+17920d^4+58597d^3+115318d^2+127128d+60480}{5040}},
\\
|\Cong(\Ptw{n,d})|&=\textstyle{\binom{3n+d-4}{3n-5}+8\binom{3n+d-1}{3n-1}+2\binom{3n+d-2}{3n-1}+5\binom{3n+d-3}{3n-1}-2\binom{3n+d-4}{3n-1}} \text{ for }  n\geq 4.
\end{align*}
\item  \label{it:ub2}
The array $\bigl(|\Cong(\Ptw{n,d})|\bigr)_{n,d\geq 0}$ has a rational generating function.
\item  \label{it:ub3}
For any fixed $n\geq 0$, the function $\N\rightarrow \N$, $d\mapsto |\Cong(\Ptw{n,d})|$, is a polynomial.
For $n\geq 4$ this polynomial has degree $3n-1$ and leading term $13d^{3n-1}/(3n-1)!$.
\item  \label{it:ub4}
For any fixed $d\geq 0$, the function $\{4,5,\dots\}\rightarrow\N$, $n\mapsto |\Cong(\Ptw{n,d})|$, is a polynomial
of degree $d+1$ with leading term $(3n)^{d+1}/(d+1)!$.
\end{thmenumerate}
\end{thm}

\begin{proof}
\ref{it:ub1}
This follows by combining \eqref{eq:Cgast} and \eqref{eq:aform}.
For $|\Cong(\Ptw{n,d})|$ we  rewrite the first term $\binom{3n+d-4}{d+1}$ as $\binom{3n+d-4}{3n-5}$.

\ref{it:ub2}
For simplicity, we will write $c(n,d)=|\Cong(\Ptw{n,d})|$.  A rational form for the generating function
\[
C(x,y) := \sum_{n,d\geq0}c(n,d)x^ny^d
\]
can be obtained by using a sequence of standard manipulations on the generating function $A(x,y)=\sum_{k,d\geq0}a(k,d)x^ky^d$, whose rational form is given in \eqref{eq:GFa}.
Specifically, the following steps need to be performed:
\bit
\item
`Pick out' the terms in $A(x,y)$ corresponding to the rows of $(a(k,d))_{k,d\geq 0}$ indexed by the numbers of the form $k=3l+2$ ($l\geq 0$). Formally, if the array $(f(k,d))_{k,d\geq 0}$ is defined by
\[
f(k,d)=\begin{cases} a(k,d) &\text{if } k\equiv 2 \!\!\!\! \pmod{3}\\ 0 &\text{otherwise},\end{cases}
\]
then the generating function for this array is
\[
F(x,y):= \sum_{k,d\geq0}f(k,d)x^ky^d =\frac{1}{3}\big(A(x,y)+\omega A(\omega x,y)+\omega^2 A(\omega^2 x,y)\big),
\]
where $\omega$ is a primitive cube root of unity in $\mathbb{C}$.
\item
From $F(x,y)$ `remove' the terms corresponding to rows $2$ and $5$ of $(a(k,d))_{k,d\geq 0}$.
This is done by recalling that the sequences $(a(2,d))_{d\geq 0}$ and $(a(5,d))_{d\geq 0}$ are polynomial,
as given by~\eqref{eq:aform}. So their generating functions $A_2(y):= \sum_{d\geq0}a(2,d)y^d$ and $A_5(y):= \sum_{d\geq0}a(5,d)y^d$ can readily be computed, and then the desired generating function is
\[
G(x,y):= F(x,y)-x^2A_2(y)-x^5A_5(y).
\]
\item
Expanding and simplifying $G(x,y)$ we obtain a rational function of the form $x^2 H(x^3,y)$.
\item
The underlying function $H(x,y)$ is nearly our desired generating function $C(x,y)$.
However, it has no terms $x^iy^j$ where $i=0,1$, and the coefficient of a general term $x^iy^j$ ($i\geq 2$) 
is in fact $c(i+2,j)$. Hence, to obtain $C(x,y)$ we need to `shift' $H$ by two, and `insert' generating functions $C_i(y):=\sum_{d\geq0}c(i,d)y^d$ for each $0\leq i\leq3$:
\[
C(x,y)=x^2H(x,y)+C_{0}(y)+xC_{1}(y)+x^2C_{2}(y)+x^3C_{3}(y).
\]
\item
Polynomial expressions for $c(i,d)$, $0\leq i\leq3$, are given in part \ref{it:ub1}, and they can be converted into the corresponding generating functions $C_i(y)$.
Performing the above calculations in Maple, the desired generating function is now:
\eit
{\scriptsize
\begin{multline}\label{eq:GF}
C(x,y)=\frac{1}{(y-1)^9(x-1)(y^3-3y^2+x+3y-1)}
\Bigl((-x^2+x+1)y^{11}+(x^3+7x^2-8x-12)y^{10}
+(-8x^3-19x^2+28x+65)y^9\\
+(19x^3+28x^2-56x-210)y^8+(x^3-34x^2+69x+450)y^7
+(-80x^3+45x^2-49x-672)y^6\\
+(151x^3-32x^2+7x+714)y^5+(-x^5+x^4-122x^3-31x^2+27x-540)y^4+(x^5-x^4+31x^3+87x^2-34x+285)y^3\\
+(8x^5-8x^4+19x^3-76x^2+21x-100)y^2+(4x^5-4x^4-15x^3+31x^2-7x+21)y+x^5-x^4+3x^3-5x^2+x-2\Bigr).
\end{multline}
}

\ref{it:ub3}
That all these functions are polynomial follows from \ref{it:ub1}, given that for fixed integers $s\in\mathbb Z$ and $t\in\N$, $\binom{d+s}t=\frac{(d+s)(d+s-1)\cdots(d+s-t+1)}{t!}$ is a polynomial in $d$ of degree $t$.  For the second statement, it is clear that the highest power of $d$ is $d^{3n-1}$, that only the last four terms contribute such power, and that the coefficient is $(8+2+5-2)/(3n-1)!=13/(3n-1)!$.

\ref{it:ub4}
This time we can rewrite the final formula from \ref{it:ub1} as
\[
|\Cong(\Ptw{n,d})|=\binom{3n+d-4}{d+1}+8\binom{3n+d-1}{d}+2\binom{3n+d-2}{d-1}+5\binom{3n+d-3}{d-2}- 2\binom{3n+d-4}{d-3}.
\]
This is clearly a polynomial in $n$, of degree $d+1$, which comes from the first term with coefficient $3^{d+1}/(d+1)!$.
\end{proof}

\begin{rem}
The first few polynomials in part \ref{it:ub4} are as follows, each valid (only) for $n\geq4$:
\[
|\Cong(\Ptw{n,0})| = 3n+4 \COMMA
|\Cong(\Ptw{n,1})| = \tfrac{9n^2+27n+16}2 \COMMA
|\Cong(\Ptw{n,2})| = \tfrac{9n^3+45n^2+62n+2}2 .
\]
\end{rem}

\begin{rem}\label{rem:as}
The leading terms of the polynomials in parts \ref{it:ub3} and \ref{it:ub4} lead directly to asymptotic expressions:
\bit
\item $|\Cong(\Ptw{0,d})| \sim d$, $|\Cong(\Ptw{1,d})| \sim \frac{3d^2}2$, $|\Cong(\Ptw{2,d})| \sim \frac{13d^4}{24}$ and $|\Cong(\Ptw{3,d})| \sim \frac{13d^7}{5040}$, as $d\to\infty$,
\item $|\Cong(\Ptw{n,d})| \sim \frac{13d^{3n-1}}{(3n-1)!}$ as $d\to\infty$, for fixed $n\geq4$, and
\item $|\Cong(\Ptw{n,d})| \sim \frac{(3n)^{d+1}}{(d+1)!}$ as $n\to\infty$, for fixed $d\geq0$.
\eit
\end{rem}

The numbers of congruences of $\Ptw{n,d}$, for $0\leq n,d\leq 10$, given by Theorem \ref{th:uberenu}  are listed in Table \ref{tab:enum}.
As a verification of our results, we have computed the same numbers 
by generating and counting the fC-matrices, as well as by directly computing the congruences using GAP \cite{GAP4,Semigroups}, although this latter computation is only feasible for smaller combinations of the parameters.

\begin{table}[ht]
\begin{center}
\scalebox{0.8}{
\begin{tabular}{|l|ccccccccccc|}
\hline
$	n\sm d	$ & $	0	$ & $	1	$ & $	2	$ & $	3	$ & $	4	$ & $	5	$ & $	6	$ & $	7	$ & $	8	$ & $	9	$ & $	10	$ \\
\hline
$	0	$ & $	2	$ & $	3	$ & $	4	$ & $	5	$ & $	6	$ & $	7	$ & $	8	$ & $	9	$ & $	10	$ & $	11	$ & $	12	$ \\
$	1	$ & $	3	$ & $	7	$ & $	14	$ & $	24	$ & $	37	$ & $	53	$ & $	72	$ & $	94	$ & $	119	$ & $	147	$ & $	178	$ \\
$	2	$ & $	9	$ & $	43	$ & $	136	$ & $	334	$ & $	696	$ & $	1294	$ & $	2213	$ & $	3551	$ & $	5419	$ & $	7941	$ & $	11254	$ \\
$	3	$ & $	12	$ & $	76	$ & $	329	$ & $	1105	$ & $	3100	$ & $	7608	$ & $	16842	$ & $	34353	$ & $	65560	$ & $	118404	$ & $	204139	$ \\
$	4	$ & $	16	$ & $	134	$ & $	773	$ & $	3456	$ & $	12806	$ & $	41054	$ & $	117273	$ & $	304889	$ & $	732888	$ & $	1648660	$ & $	3503734	$ \\
$	5	$ & $	19	$ & $	188	$ & $	1281	$ & $	6754	$ & $	29413	$ & $	110312	$ & $	366724	$ & $	1103538	$ & $	3053642	$ & $	7865696	$ & $	19043434	$ \\
$	6	$ & $	22	$ & $	251	$ & $	1969	$ & $	11930	$ & $	59547	$ & $	255132	$ & $	965409	$ & $	3293916	$ & $	10294295	$ & $	29832242	$ & $	80951191	$ \\
$	7	$ & $	25	$ & $	323	$ & $	2864	$ & $	19578	$ & $	110012	$ & $	529298	$ & $	2242845	$ & $	8544569	$ & $	29728765	$ & $	95627675	$ & $	287192490	$ \\
$	8	$ & $	28	$ & $	404	$ & $	3993	$ & $	30373	$ & $	189556	$ & $	1010840	$ & $	4737070	$ & $	19912815	$ & $	76266840	$ & $	269426820	$ & $	886585245	$ \\
$	9	$ & $	31	$ & $	494	$ & $	5383	$ & $	45071	$ & $	309114	$ & $	1808352	$ & $	9279855	$ & $	42636438	$ & $	178144941	$ & $	685232184	$ & $	2450483412	$ \\
$	10	$ & $	34	$ & $	593	$ & $	7061	$ & $	64509	$ & $	482051	$ & $	3068039	$ & $	17102328	$ & $	85221356	$ & $	385570064	$ & $	1603380636	$ & $	6189136484	$ \\
\hline
\end{tabular}
}
\caption{The number of congruences on $\Ptw{n,d}$.}
\label{tab:enum}
\end{center}
\end{table}

It is also interesting to compare the enumeration of congruences of $\Ptw{n,d}$ with those of the (un-twisted) partition monoids $\P_n$, as well as the classical transformation monoids, e.g.~the monoids of full transformations $\T_n$, partial transformations $\PT_n$, and partial bijections $\I_n$, on the set $\bn$. 
For the latter, the classical results of Mal'cev \cite{Malcev1952}, Liber \cite{Liber1953} and
{\v{S}}utov \cite{Sutov1961} (see also a more recent, unified presentation in \cite[Section 6.3]{GMbook}) show that for $n\geq 4$ we have
\[
|\Cong(\T_n)|=|\Cong(\PT_n)|=|\Cong(\I_n)|=3n-1.
\]
Furthermore, Theorem \ref{thm:CongPn} and Figure \ref{fig:CongPn}  show that even though the lattice $\Cong(\P_n)$ is more complicated than in the case of transformations, its size $|\Cong(\P_n)|=3n+8$ ($n\geq 4$) remains linear in $n$. 
Our Theorem \ref{th:uberenu} can be viewed as continuing this theme:
 for $n\geq 4$ we have $|\Cong(\Ptw{n,0})|=3n+4$, a linear function, and, for higher $d$, the values $\big(|\Cong(\Ptw{n,d})|\big)_{n\geq 4}$ at least retain the polynomial behaviour, even though with an increasing degree.

\footnotesize
\def\bibspacing{-1.1pt}
\bibliography{biblio}
\bibliographystyle{abbrv}

\end{document}